\documentclass[12,reqno]{amsart}
\usepackage{amsmath, amssymb, esint}
\usepackage[margin=1in]{geometry}
\usepackage{amsmath, amsthm}
\usepackage{bbm,amssymb}
\usepackage{verbatim}

\usepackage{enumerate}
\usepackage{graphicx,color}
\newtheorem{definition}{Definition}[section]
\newtheorem{theorem}{Theorem}[section]
\newtheorem{prop}[theorem]{Proposition}
\newtheorem{lemma}[theorem]{Lemma}
\newtheorem{corollary}[theorem]{Corollary}
\newtheorem{coro}[theorem]{Corollary}

\newtheorem{remark}[theorem]{Remark}

\def\C{\mathbb{C}}

\def\H{\mathbb{H}}
\def\e{\epsilon}

\def\a{\alpha}
\def\b{\beta}
\def\d{\delta}

\def\th{\theta}

\def\f{\varphi}
\def\k{\kappa}
\def\z{\zeta}

\newcommand{\real}{\mathbb{R}}

\newcommand{\nat}{\mathbb{N}}

\newcommand{\eps}{\epsilon}

\newcommand{\chrc}[1]{\mathrm{1}_{#1}}
\newcommand{\del}{\partial}
\newcommand{\hsd}{\mathcal{H}^1}

\title{Structure of one-phase free boundaries in the plane}
\date{}
\author{David Jerison and Nikola Kamburov}
\thanks{The first author was supported by NSF grant
DMS 1069225 and the Stefan Bergman Trust.} 
\address{David Jerison, Department of Mathematics, Massachusetts Institute of
Technology, 77 Massachusetts Avenue, Cambridge, MA 02139, USA}
\email{jerison@math.mit.edu}
\address{Nikola Kamburov, Department of Mathematics, The University of Arizona,
617 N Santa Rita Ave, P.O. Box 210089, Tucson AZ 85721, USA}
\email{kamburov@math.arizona.edu}


\begin{document}

\begin{abstract}
We study classical solutions to the one-phase free boundary problem
in which the free boundary consists of smooth curves and the
components of the positive phase are simply-connected. We show that
if two components of the free boundary are close, then the solution
locally resembles an entire solution discovered by Hauswirth, H\'elein and Pacard, 
whose free boundary has the shape of a double
hairpin. Our results are analogous to theorems of Colding and
Minicozzi characterizing embedded minimal  annuli, and a direct
connection between our theorems and theirs can be made using a
correspondence due to Traizet.
%Indeed, a direct connection between the double hairpin and
%the catenoid was discovered by Traizet.
%, the two subjects being
%remarkably related, as discovered by Traizet.
\end{abstract}

\bibliographystyle{alpha}

\maketitle

\section{Introduction.}

The one-phase free boundary problem in a disk
$B\subseteq \real^2$,
\begin{equation}\label{FBP_0}
\begin{array}{ccl}
 u \ge 0 & \text{in} & B \\
     \Delta u   = 0  & \text{in} & B^+(u):=\{x\in B:u(x)>0\}\\
     |\nabla u|  = 1 & \text{on}& F(u) := \del B^+(u)\cap
     B
\end{array}
\end{equation}
arises as the Euler-Lagrange equation for the functional
\begin{equation}\label{functional}
    I(u,B) = \int_B |\nabla u|^2 + \chrc{\{u>0\}} ~dx \qquad
    u:B\to[0,\infty)
\end{equation}
and appears in a variety of applications (e.g. jet flows in
hydrodynamics, see \cite{CafSalsa}). The interior regularity theory of
minimizers of the functional $I(u,B)$ with fixed boundary conditions on $\partial B$
is well understood.
Alt and Caffarelli \cite{AC} proved that
the free boundary $F(u)$ is locally a graph of a
$C^\infty$ function (and hence analytic by \cite{KindNiren}).
Alt and Caffarelli also proved partial regularity of free boundaries
in higher dimensions and established a strong analogy
between the theory of free boundaries and the
theory of minimal surfaces.

In keeping with \cite{AC} and many subsequent results (\cite{ACF, CafI, CafII, Weiss1, CJK, DSJ, JS})
% I think CJK is worth citing  and possibly even JS  = jerison-savin arxiv \cite{}),
one should expect that most theorems about
minimal surfaces have counterparts in the theory of free
boundaries and vice versa.   Here we consider classical solutions
to \eqref{FBP_0} that are higher critical points rather than minimizers of
the functional $I(u,B)$ with one additional purely topological assumption,
namely that
\begin{equation}\label{topo_assmpt}
\text{no connected component of $F(u)$ is compact
in the open disk $B$.}
\end{equation}
By classical solution we mean one for which $F(u)$ is a finite
union of analytic curves.  The topological assumption is
equivalent to saying that the connected components of the positive phase
are simply-connected.   It is also equivalent to saying that the analytic
curves, although they may become tangent at interior points,
end at $\partial B$.

Our work is inspired by the groundbreaking work of Colding and
Minicozzi  on the structure of limits of sequences of
embedded minimal surfaces of fixed genus in a ball in $\real^3$
(\cite{CM1, CM2, CM3, CM4}).   As it turns out, because
of recent work of Traizet \cite{Traizet}, there is a direct overlap between
our {\it a priori} estimates and rigidity results for families
of solutions to \eqref{FBP_0} and the description
of embedded minimal topological annuli due to Colding and Minicozzi.

Our starting place is the family of simply-connected planar regions
$\Omega_a = a\Omega_1$, discovered by Hauswirth, Helein, and Pacard
\cite{HHPac},
which solve the free boundary problem \eqref{FBP_0}.  They are defined
by
\begin{equation*}
    \Omega_a: =
    \{(x_1, x_2)\in \real^2: |x_1/a|<\pi/2+\cosh(x_2/a)\}, \quad a>0.
\end{equation*}
The boundary $\partial \Omega_a$ consists of two curves that we will refer to as hairpins.    Hauswirth et al found a positive harmonic function $H_a(x) = aH_1(x/a)$
on $\Omega_a$ that satisfies the free boundary conditions $H_a=0$ and
$|\nabla H_a|=1$ on $\partial \Omega_a$.  Extending $H_a$ to be zero in
the complement of $\Omega_a$, we have an entire solution to \eqref{FBP_0}.
(See Section 10 for the explicit formula for $H_a$ using conformal mapping.)

Our first main result characterizes blow-up limits of classical
solutions with simply-connected positive phase.
\begin{theorem}\label{theorem_Main}
Let $u_k$ be a sequence of classical solutions of \eqref{FBP_0} in the
disk $B_{R_k}= B_{R_k}(0)$, with radius $R_k\nearrow \infty$,
satisfying $0\in F(u_k)$ and \eqref{topo_assmpt}.
%\eqref{topo_assmpt} is satisfied.
Then a subsequence converges
uniformly on compact subsets of $\real^2$ to some rigid motion
of one of the following

\begin{enumerate}[(a)]
\item  $P(x) := x_2^+$, a \emph{half-plane solution},

\item $W_b(x) := x_2^+ + (x_2+b)^-$, for
some $b \ge 0$, a \emph{two-plane solution}, or

\item $H_a(x)$, for some $a>0$, a \emph{hairpin solution}  as mentioned
above and defined in \eqref{eq:haipin_def} of Section \ref{Sec:CurvBounds}.
\end{enumerate}

\end{theorem}

Note that unlike property \eqref{topo_assmpt}, connectivity of the positive
phase is not inherited in the limit. For example, blow-up limits of suitable
translates and  dilates of $H_1$ are two-plane solutions.

Theorem \ref{theorem_Main} is closely related to earlier
classifications of entire solutions with simply-connected
positive phase due to Khavinson, Lundberg and Teodorescu
\cite{KhavLundTeo} and Traizet \cite{Traizet}.   Traizet
showed that classical entire solutions satisfying \eqref{topo_assmpt}
must be of the form (a), (b), or (c).    Khavinson et al showed
that the same conclusion is true under a
natural, weak regularity assumption on the free boundary known as
the Smirnov property.   We were not able to use this result to prove our
theorem, and this is a central technical difficulty of the
paper.   Instead, we define another notion of weak solution
that we can show is preserved under blow-up limits.   Our weak solutions will
satisfy both the properties of non-degenerate viscosity solutions introduced by L. Caffarelli and variational solutions introduced by G. Weiss. This PDE-theoretic approach has the benefit that it does not rely on complex function theory and so it could conceivably be extended to a higher-dimensionsional setting.

Our next result says that near points where the
curvature of the free boundary is large, the boundary resembles
a double hairpin.
\begin{theorem}\label{thm:BigPicture} Given $\d>0$
there exist  positive numbers $r$, $\kappa$, $\eps$ and $\eps_1$
with $0<\eps_1< \eps/2< 1/100$, and an integer $N_0\ge 0$ such that if $u$
is a classical solution of \eqref{FBP_0} in $B_1$, satisfying
\eqref{topo_assmpt}, then there are $N\leq N_0$ points
$\{z_j\}_{j=1}^N \subseteq B_{3/4}$, with the
properties:
\begin{enumerate}[(a)]
\item The curvature of $F(u)$ is less than $\kappa$ at any point
of $F(u)\cap \left(B_{1/2}\setminus \bigcup_{j=1}^N
B_{r}(z_j)\right)$.

\item Near $z_j$, $u$ is approximated by a hairpin solution, i.e. there exists some $a_j<
\eps_1 r$ such that
\begin{equation*}
    |u(z_j + x) - H_{a_j}(\rho_j x)| \leq \d a_j \quad \text{for
    all} \quad
    |x|\leq 2a_j/\eps
\end{equation*}
for some rotation $\rho_j$.

\item In $B_{2r}(z_j)\backslash B_{a_j/\eps}(z_j)$ the free
boundary consists of four curves which are graphs in some
common direction with small Lipschitz norms.  More precisely,
there exist $f, \, g: \real\to \real$ such that $f<g$,
\[
\|f\|_{L^\infty} + \|g\|_{L^\infty}  \le \delta r,
\quad
\|f'\|_{L^\infty} + \|g'\|_{L^\infty}  \le \delta,
\]
and
\[
\{u=0\} \cap (B_{2r}(z_j) \backslash B_{a_j/\eps}(z_j)
= z_j + \rho_j (\{x: f(x_1) \le x_2 \le g(x_1)\} \cap
\{x: a_j/\eps < |x| < 2r\})
\]
\end{enumerate}
\end{theorem}

The proof of parts (a) and (b) of Theorem \ref{thm:BigPicture} follow
from the classification of blow-up solutions in Theorem \ref{theorem_Main}.
The proof of part (c) uses conformal mapping and is of independent
interest.   The usual flat $\implies$ Lipschitz step in regularity theory
implies that the boundaries are Lipschitz graphs with small Lipschitz
constant separately on each  dyadic annulus, $2^{k-1} < |x-z_j| < 2^k$
for  $a_j/\eps_0 < 2^k < r_0$.   What part (c)  rules out
is the possibility of a spiral.   It can be viewed as a quantitative
version of the  flat $\implies$ Lipschitz step, in which no
information is used about the solution in a neighborhood
$|x-z_j| < 50a_j$. Colding and Minicozzi call the analogous
bound in the setting of minimal surfaces an effective \emph{removable singularities theorem} \cite[Theorem 0.3]{CM3}.
This crucial estimate plays a large role elsewhere in their work as well.
%Which reference is this? where is it cited in \cite{CMAnnuli}?  What
% is  \cite[Lemma 3.5]{CMAnnuli}?

%Fix a sufficiently small absolute constant value of $\delta= \delta_0$,
%say $0 < \delta_0 < < 1/100$.  (We will also choose $\delta_0$ so
%small that  that $\delta_0$-flatness implies explicit interior bounds on first
%and second derivatives of the free boundary.)

The technique of conformal mapping then allows us
to obtain a more detailed rigidity theorem on a fixed-size
neighborhood of each hairpin-like structure.
\begin{theorem}\label{thm:CurvBounds} There are absolute
constants $r_0$, $\kappa_0$, and $N_0$
such that if $u$ is a classical solution to \eqref{FBP_0} in $B_1$
satisfying \eqref{topo_assmpt}, then there is $N$, $0 \le N \le N_0$
and $N$ saddle points $\{z_j\}_{j=1}^N$ of $u$ with the following
properties:

\begin{enumerate}[(a)]
\item  $F(u)$ has curvature at most $\kappa_0$ on
$ \displaystyle F(u) \cap B_{1/2} \backslash \bigcup_{j=1}^N B_{r_0}(z_j)$.

\item For each $j$, $a_j : = u(z_j) \le r_0/100$, and there is
an injective  conformal mapping
\[
\phi_j: B_{2r_0} \cap \bar \Omega_{a_j} \to \real^2 \quad
\mbox{such that} \quad \phi_j(0) = z_j, \quad \mbox{and} \quad
B_{r_0}(z_j)^+ (u) \subset   \phi_j(B_{2r_0}\cap \Omega_{a_j}) \subset
B_{4r_0}(z_j)^+ (u).
\]
Moreover, there is $\theta_j \in \real$ such that 
for all $z\in B_{2r_0} \cap \Omega_{a_j}$,
\[
|\phi_j'(z) -e^{i\theta_j}| \le |z|/(100 r_0); \quad
|\phi_j''(z) | \le 1/(100 r_0).
\]

\item  If $\kappa$ denotes the curvature of $F(u)$ and $\kappa_{a}$ denotes
the curvature of $\partial \Omega_a$, then
\[
|\kappa(\phi_j(z)) - \kappa_{a_j}(z)| \le 1/(100 r_0), \quad z\in B_{2r_0} \cap
\partial \Omega_{a_j}.
\]
\end{enumerate}
\end{theorem}

To interpret part (c) of this theorem, note that
\[
\kappa_a(z) \sim a/|z|^2, \quad z\in \partial \Omega_a
\]
Hence
\[
|z| \le \sqrt{ar_0} \implies  \kappa_a(z)\gg \frac1{100r_0}
\]
Furthermore, $a$ is comparable to the separation distance between
the two hairpins.  Thus, for points closer to $z_j$ than
the geometric mean of the separation distance between the
two hairpins and the distance $r_0$, the bound
in part (c) says that the curvature
of the approximate hairpins is close to that of the
standard model.    In particular, the two components of the zero set
are convex in this range.  At distances significantly
larger than this geometric mean, one can no longer guarantee
that $\kappa(\phi_j(z))$ is positive, but the bound in part (c)
still implies that $|\kappa(\phi_j(z)| \le 1/(50r_0)$.  This
is a nontrivial bound.    At the largest scale, $r_0 < |z| < 2r_0$
it is the same as the standard interior 2nd derivative bounds for flat
free boundaries, but at smaller dyadic scales it is a stronger
curvature constraint.

In \cite{Traizet}, Traizet found a remarkable change of
variables that converts the free boundary problem into
a problem about minimal surfaces with a plane of symmetry.
If $|\nabla u| <1$, then the minimal surfaces are embedded,
and otherwise they are immersed.   This means
that although neither problem is strictly contained
in the other, there is direct overlap between the results of Colding
and Minicozzi and the results proved here.   The extra
hypothesis $|\nabla u| <1$ removes nearly all the difficulties
from the free boundary classification problem we are considering because
in that case the zero set of $u$ consists of convex components.
Nevertheless, in this simple overlapping case Traizet's
change of variables allows us to make a direct comparison
with results of \cite{CMAnnuli}.

Under Traizet's correspondence, the standard double hairpin becomes
the standard catenoid,
\begin{equation*}
\Sigma_{\rho} = \{(x_1, x_2, x_3)\in \real^3: (x_2/\rho)^2 +
(x_3/\rho)^2 = \cosh^2(x_1/\rho))\}, \quad \rho>0.
\end{equation*}
Denote $\mathcal{B}_r := \{x\in \real^3: |x|<r\}$.

\begin{corollary}\label{thm:minimal_annulus}
Let $M\subseteq\mathcal{B}_R$ be an embedded
minimal surface, homeomorphic to an annulus, with $\del M\subseteq
\del \mathcal{B}_R$.  Suppose that $M$ is symmetric with respect to the
reflection $x_3 \mapsto -x_3$ and that $M^+=M\cap \{x_3>0\}$ is a simply-connected graph over the $x_1x_2$-plane.
Suppose that the shortest closed geodesic of $M$ has length
$\eps$ and passes through the origin in $\mathcal{B}_R$.
There are absolute
constants $R_0<\infty$ and $\eps_0>0$ such that if $R\ge R_0$ and
$\eps \le \eps_0$, then there exists $\rho>0$,
\begin{equation*}
  | 2\pi \rho - \eps| \le \eps/100
\end{equation*}
and an injective conformal mapping $\phi: \Sigma_\rho \cap \mathcal{B}_1\to M$
that is isometric up to a factor 
%$1\pm (|x|+\eps)/100$, and
$1\pm |x|/100$, and
the Gauss curvatures $K$ of $M$ and $K_{\rho}$ of $\Sigma_{\rho}$ are
related by
\begin{equation*}
    | K(\phi(x))-  K_\rho(x)|  \leq 
\begin{cases} (1/100)(\epsilon/|x|^2), & \quad |x| \le  \sqrt{\epsilon} \\
1/100, &  \quad \sqrt{\epsilon} \le |x| \le 1
\end{cases} 
%\big(\sqrt{K_{\rho}}(x) + 1/100\big), \quad x\in \mathcal{B}_1 \cap \Sigma_\rho.
   \end{equation*}
\end{corollary}
Note that because $K_\rho(x)  \sim -\rho^2/|x|^4$ and $\eps \approx \rho$,
%\[
%|x| < \sqrt{\eps} \implies K_{\rho}(x) \gg 1/100 (\sqrt{K_\rho}(x) + 1/100)
%\]
in the range $|x|  < \sqrt{\eps}$, the curvatures are close.
This is the same bound as (but in much less generality than) the sharpest 
result of Colding and Minicozzi (see \cite[Remark 3.8]{CMAnnuli}). 
On the other hand, our corollary gives nontrivial rigidity
  for both distance distortion and curvature
 in the range $\sqrt{\eps} \ll  |x| \ll 1$.
 This range is not addressed in \cite{CMAnnuli},
 and the present result suggests that there
 may be interpolating rigidity estimates all the way to
 unit scale that are  valid in the case of general embedded minimal
 annuli.

\subsection{Outline of the paper.}\label{subsec:outline}

The first seven sections of the paper are devoted to the proof of \mbox{Theorem \ref{theorem_Main}}.  In Section \ref{Sec_Prelims} we establish the universal Lipschitz and \emph{nondegeneracy} bounds
enjoyed by the sequence of solutions $u_k$. Section \ref{Sec_WeakSol}
describes the two notions of weak solutions -- viscosity and variational -- that are preserved under the limit. In Section \ref{Sec_Blows} we recall the Weiss Monotonicity Formula \cite{Weiss1} and use it to
characterize the blow-up/blow-down limit of a weak nondegenerate
solution; there are two possibilities (up to rigid motion): the
half-plane solution $P(x)=x_2^+$ or
\begin{equation*}
    V(x)= s|x_2| \quad
    \text{for some } 0<s\leq 1.
\end{equation*}
Weak solutions approaching the half-plane solution are well understood by
the classical results of Caffarelli \cite{CafI, CafII} and our focus
will be to understand the structure of classical solutions that are
close to $V$. The first step is carried out in
Section \ref{Sec_AuxLemmas}, where we prove some auxiliary lemmas
concerning the structure of their free boundary. We also establish
the key fact that the gradient magnitude of weak solutions, which blow down to $V$, is bounded above by $1$; this, in
turn, translates to the strong geometric property that $F(u)$ has
non-negative curvature wherever it's smooth. The latter will be a
key element in the proof of Theorem \ref{theorem_Main}, carried out in Section \ref{Sec_Classfn}.

In Section \ref{Sec:LocStruct} we start exploring the local
structure of a solution $u$, satisfying \eqref{topo_assmpt}, in the
unit disk $B_1$. We delineate a dichotomy -- if near a point $p$ of
the free boundary there are two connected components of the zero
phase close enough to each other at a distance $O(a)$, then $u$
resembles $|x_2|$ (up to a rigid motion) in a unit-size
neighborhood $B_{r_0}(p)$ (this scenario will ultimately lead to
$u$ resembling a hairpin solution); otherwise, the free boundary has
bounded curvature at $p$. Sections \ref{Sec:lipboundFBstrands} and
\ref{Sec:CurvBounds} are devoted to exploring the first branch of
the dichotomy. In Section \ref{Sec:lipboundFBstrands} we show that
that the free boundary from scale $r_0$ all the way down to scale
$O(a)$ consists of four curves that have bounded turning in the outer scales. In the
penultimate Section \ref{Sec:CurvBounds} we finally see the hairpin
arising in the inner scale and we systematically treat both scales
by constructing an injective holomorphic map (Lemma
\ref{lemma:confmapConstruction}) from the positive phase of $u$ in
$B_{r_0}(p)$ to the positive phase of an appropriate hairpin
solution $H_a$. Obtaining estimates on the second derivative of the
map in Lemma \ref{lemma:psiEstimates} allows us to relate the
curvature of $F(u)$ to the curvature of $F(H_a)$ of a model hairpin
solution.% in the form described by Theorem \ref{thm:CurvBounds}.

In the last Section \ref{Sec:Traizet} we exploit the Traizet
correspondence to prove Corollary \ref{thm:minimal_annulus}.

\section{Notation.}\label{Sec_Notation}

The disk of radius $r$ centered at $x=(x_1,x_2)\in \real^2$ will be
denoted by $B_r(x)$. When the argument is absent, we are referring
to the disk centered at the origin, $B_r:=B_r(0)$. The unit vectors
along $x_1$ and $x_2$ will be denoted by $e_1$ and $e_2$,
respectively. The three-dimensional ball of radius $r$, centered at
$p\in \real^3$, will be denoted by $\mathcal{B}_r(p)$.

If $\Omega$ is an open set of $\real^2$ and $u:\Omega\rightarrow
\real$ is a non-negative function, define the \emph{positive phase}
of $u$ to be
\begin{equation*}
    \Omega^+(u) := \{x\in \Omega: u(x)>0\}
\end{equation*}
and its \emph{free boundary} $F(u):=\del \Omega^+(u)\cap \Omega$.

If $S\subseteq \real^2$, a $\delta$-neighborhood of $S$ will be
denoted by
\begin{equation*}
    \mathcal{N}_{\delta}(S) := \bigcup_{x\in S}B_{\delta}(x).
\end{equation*}
Denote the distance between two non-empty sets $U, V$ by
\begin{equation*}
    d(U,V) = \inf\{|p-q|: p\in U, q\in V\},
\end{equation*}
while the Hausdorff distance between two compact subsets $K_1, K_2$
of $\real^2$ will be denoted by
\begin{equation*}
    d_{H}(K_1, K_2) = \inf\{\delta>0: K_1\subseteq \mathcal{N}_{\delta}(K_2) \text{ and } K_2\subseteq
    \mathcal{N}_{\delta}(K_1)\}.
\end{equation*}
By $\hsd$ we shall refer to the one-dimensional Hausdorff measure.

In all that follows $C, c, c', \tilde{c}, c_0, c_1, c_2$, etc. will
denote positive numerical constants. The constants in the
$O$-notation, wherever used, are also meant to be numerical.

\section{Preliminaries.}\label{Sec_Prelims}

%Let us first state some basic elements from plane topology. A simple
%curve, or an \emph{arc}, is the image of an injective continuous map
%from $[0,1]$ into the plane. A simple closed curve, also referred to
%as a \emph{Jordan curve}, is a subset of the plane homeomorphic to
%the circle $\mathbb{S}^1$. By the Jordan curve theorem, the
%complement of a Jordan curve $\gamma$ consists of two
%\emph{connected} components -- a bounded one (the interior) and the
%unbounded (the exterior), with both having $\gamma$ as their
%boundary.

Let $u$ be a solution of \eqref{FBP_0} in a disk $B\subseteq
\real^2$ that satisfies \eqref{topo_assmpt}. In our forthcoming
arguments we shall often be working with some connected component
$U$ of $[B_{r}(x)]^+(u)$, where $B_{r}(x)\Subset K\Subset B$ for
some compact set $K$. Claim that $U$ is a piecewise smooth domain;
that will provide us with enough regularity to apply the Divergence
Theorem in $U$. It suffices to show that only finitely many
connected components of $F(u)$ intersect $\del B_r(x)$ and that each
intersects it only a finite number of times. Let $\gamma$ be any
connected component of $F(u)$ intersecting $K$. Since for each $p\in
F(u)$, $F(u)\cap{B_{\eps(p)}}(p)$ is locally the graph of a smooth
function when $\eps(p)$ is small enough, the compact $\gamma\cap K$
has a finite subcover $\{B_{\eps(p_i)}(p_i), p_i\in \gamma\cap
K\}_{i=1}^N$, so that
\begin{equation}\label{isolation}
    d\left( \gamma\cap K, (F(u) \setminus
    \gamma)\cap K \right) \geq \delta(\gamma):=\frac{1}{2}\min\{\eps_{p_i}\}_{i=1}^N.
\end{equation}
But $\{\mathcal{N}_{\delta(\gamma)}(\gamma\cap K)\}_{\gamma}$, where
$\gamma$ ranges over all connected curves of $F(u)$ intersecting
$K$, is a cover of the compact $F(u)\cap K$, so it has a finite
subcover $\{\mathcal{N}_{\delta(\gamma_j)}(\gamma_j\cap
K)\}_{j=1}^M$. Because of \eqref{isolation}, each element of the
subcover contains only $\gamma_j\cap K$ and nothing else from
$F(u)\cap K$, so there are only finitely many curves $\gamma$
intersecting $K$ and thus $B_r(x)$. Each such $\gamma$ intersects
$\del B_r(x)$ only a finite number of times, because by the
classical result of \cite{KindNiren}, the free boundary $F(u)$ is
real analytic.

We shall now prove two fundamental regularity properties that
classical solutions of \eqref{FBP_0} given \eqref{topo_assmpt}
satisfy: universal Lipschitz bound and universal non-degeneracy away
from the free boundary. To elucidate the latter part of our claim,
let us state the relevant definition.

\begin{definition} A non-negative function $u:\Omega \to \real$ is
\textbf{non-degenerate} if there exists a constant $c>0$, such that
\begin{equation*}
    \sup_{B_r(x)} u \geq cr
\end{equation*}
for every $B_r(x)\subseteq \Omega$ centered at a point $x_0\in
F(u)$.
\end{definition}

First, let us show that classical solutions enjoy a universal
Lipschitz bound.

\begin{prop}[Lipschitz bound]\label{prop_lipbd}
Let $u$ be a classical solution of \eqref{FBP_0} in $B_R(0)$. If the
largest disk in $B_R^+(u)$, centered at $x$, touches $F(u)$, then
\begin{equation*}
    |\nabla u|(x)\leq C.
\end{equation*}
for some numerical constant $C>0$. In particular, if $0 \in F(u)$
\begin{equation}\label{lip_bd}
\|\nabla u\|_{L^{\infty}(B_{R/2})} \leq C.
\end{equation}
\end{prop}

\begin{proof}
If $u(x) = m$, then by Harnack's inequality $c_1 m\leq u(y)\leq c_2
m$ on $\del B_{r/2}(x)$. % for some numerical constants $c_1,c_2>0$.
Let $h$ be the harmonic function in the annulus
$A_r(x):=B_{r}(x)\setminus B_{r/2}(x)$, whose boundary values are:
\begin{align*}
    h &  = c_1m \quad \text{on}\quad \del B_{r/2}(x) \\
    h & = 0\quad \text{on}\quad \del B_{r}(x).
\end{align*}
By the maximum principle $h\leq u$ in $A_r$ and so by the Hopf
lemma,
\begin{equation*}
    h_{\nu}(p)\leq u_{\nu}(p)=1,
\end{equation*}
where $\nu$ denotes the inner-normal to $B_R^+$ and $p\in F(u)$ is a
point of touching between $F(u)$ and $B_r(x)$. On the other hand,
$h_{\nu}(p) \geq c'm/r$, thus
\begin{equation*}
    m \leq C' r.
\end{equation*}
%for some numerical constants $c',C'$.
Thus,
\begin{equation*}
    |\nabla u|(x) \leq \frac{c_0}{r} \fint_{\del B_{r/2}} u ~d\hsd \leq
    \frac{c_0c_2m}{r} \leq C,
\end{equation*}
for some numerical constant $C$.

Statement \eqref{lip_bd} follows once we point out that for $x \in
B_{R/2}$ the largest ball contained in $B_R^+(u)$ and centered at
$x$, will certainly touch $F(u)$.
\end{proof}

The universal nondegeneracy property is established through the
following proposition.

\begin{prop}\label{prop_supnondeg} Let $u$ be a classical solution of \eqref{FBP_0} in $B_R(0)$, for which \eqref{topo_assmpt} is satisfied. Assume further that $0\in
F(u)$. Then %there is a numerical constant $c>0$, such that
\begin{equation*}
    \sup_{B_r(0)} u  = \max_{\del B_r(0)} u \geq \frac{1}{2\pi}r \quad \text{for all} \quad 0<r< R.
\end{equation*}
\end{prop}
\begin{proof}
Since $u$ is continuous and subharmonic, the maximum principle
implies $\sup_{B_r(0)} u = \max_{\del B_r(0)} u$. Let $\tilde{u}(x)
:= r^{-1}u(rx)$ denote the $r$-rescale of $u$. It suffices to show
that $\sup_{\del B_1} \tilde{u} \geq 1/2\pi$.

Let $\phi:[0,1] \rightarrow \real$ be the function
\begin{equation*}
\phi(t) =    \left\{
\begin{array}{lr} \frac{1}{2} & 0\leq t\leq \frac{1}{2} \\
1-t & \frac{1}{2}<t\leq 1
    \end{array}\right..
\end{equation*}
and let $\psi(x) = \phi(|x|)$. Let $U$ be the component of
$B_{R/r}^+(\tilde{u}) = r^{-1}B_R^+(u)$ in $B_1$ whose boundary
contains the origin. Then if $\tilde{u}_{\nu}$ denotes the inner
normal to $U$,
\begin{equation*}
    -\int_{U} \nabla \psi \cdot \nabla \tilde{u} ~dx = -\int_U \text{div}(\psi \nabla
    \tilde{u}) ~dx = \int_{\del U\cap B_1}\psi \tilde{u}_{\nu} ~d\hsd = \int_{\del U\cap B_1}\psi
    ~d\hsd.
\end{equation*}
On the other hand, if $\hat{r}$ denotes the unit vector field in the
radial direction,
\begin{align*}
    -\int_{U} \nabla \psi \cdot \nabla \tilde{u} ~dx & = \int_{U\setminus
    B_{1/2}} \text{div}(\tilde{u}\hat{r}) - \text{div}(\hat{r})\tilde{u} ~dx
    = \\
    & =  \int_{\del B_1\cap U} \tilde{u}~d\hsd - \int_{\del B_{1/2}\cap U} \tilde{u}~d\hsd - \int_{U\setminus B_{1/2}}\frac{\tilde{u}}{|x|} ~dx.
\end{align*}
Therefore, as $\hsd(\del U \cap B_1) \geq 2$,
\begin{equation*}
    \int_{\del B_1\cap U} \tilde{u}~d\hsd \geq \int_{F(\tilde{u})\cap U}\psi
    ~d\hsd \geq \frac{1}{2}\hsd(\del U \cap B_1) \geq 1.
\end{equation*}
Hence, $\sup_{\del B_1\cap U}\tilde{u} \geq 1/2\pi$.
\end{proof}

\section{Weak solutions.}\label{Sec_WeakSol}

In this section we define the two notions of weak solutions that
will be useful in classifying the limits of sequences of classical
solutions. Let
\begin{equation*}
    I[u,\Omega] = \int_{\Omega} |\nabla u|^2 + \chrc{\{u>0\}}~dx \quad
    \Omega \subseteq \real^2
\end{equation*}
be the one-phase energy functional whose Euler-Lagrange equation is
the free boundary problem \eqref{FBP_0}.

\begin{definition} The function $u\in H^1_{\text{loc}}(\Omega)$ is a
\textbf{variational solution} of \eqref{FBP_0} if $u\in
C(\Omega)\cap C^2(\Omega^+(u))$ and
\begin{equation*}
    0 = L[u](\phi):=\left.\frac{d}{d\eps}\right|_{\eps=0} I[u(x+\eps\phi(x))] =
    \int_{\Omega}\big(|\nabla u|^2 + \chrc{\{u>0\}}\big)\text{div }\phi
    -2\nabla u D\phi (\nabla u)^T ~dx
\end{equation*}
for any $\phi \in C^{\infty}_c(\Omega; \real^2)$.
\end{definition}

The next proposition is standard and says that any globally defined
limit of uniformly convergent variational solutions that are
uniformly Lipschitz continuous and uniformly non-degenerate,
inherits the same properties.

\begin{prop}\label{prop_limitofsolns} Let $\{u_k\} \in H^1_{\text{loc}}(B_{R_k})$, $R_{k}\nearrow\infty$, be a sequence of variational solutions of
\eqref{FBP_0} which satisfies
\begin{itemize}
%\item $0 \in F(u_k)$;
\item (Uniform Lischitz continuity) There exists a constant C, such that $\|\nabla u_k\|_{L^{\infty}(B_{R_k})} \leq C$;
\item (Uniform non-degeneracy) There exists a constant c, such that $\sup_{B_r(x)} u_k \geq cr$ for every $B_r(x) \subseteq
B_{R_k}$, centered at a free boundary point $x\in F(u_k)$.
\end{itemize}
Then any limit $u \in H^1_{\text{loc}}(\real^2)$ of a uniformly
convergent on compacts subsequence $u_k \to u$ satisfies
\begin{enumerate}[(a)]
%\item $u_k \to u$ uniformly on compact subsets of $\real^2$;
\item $\overline{\{u_k>0\}} \rightarrow \overline{\{u>0\}}$ and $F(u_k) \rightarrow F(u)$ locally in the Hausdorff distance;
\item $\chrc{\{u_k>0\}} \rightarrow \chrc{\{u>0\}}$ in $L^1_{\text{loc}}(\real^2)$;
\item $\nabla u_k \rightarrow \nabla u$ a.e.
\end{enumerate}
Moreover, $u$ is a Lipschitz continuous, non-degenerate variational
solution of \eqref{FBP_0}.
\end{prop}
\begin{proof} Obviously, $u$ is a global Lipschitz continuous
function with $\|\nabla u\|_{L^{\infty}(\real^2)}\leq C$ and $u\in
H^1_{\text{loc}}(\real^2)$.  One proves properties a) through c)
arguing as in \cite[Lemma 1.21]{CafSalsa}. The non-degeneracy of $u$
follows from the non-degeneracy of $u_k$ combined with the fact that
$F(u_k) \to F(u)$ locally in the Hausdorff distance.

To show that $u$ is a variational solution as well, note that since
$\nabla u_k \to \nabla u$ a.e. and $|\nabla u_k|$ and $|\nabla u|$
are bounded above by $C$, the Dominated Convergence Theorem implies
that for every $\phi \in C^1_c(\real^2;\real^2)$
\begin{equation*}
    0 = \lim_{k\to\infty} L[u_k](\phi) = L[u](\phi).
\end{equation*}
\end{proof}

The second notion of weak solution that will make use of is that of
a viscosity super/sub-solution (\cite{CafSalsa}).
\begin{definition}\label{def_visco}
A viscosity supersolution (resp. subsolution) of \eqref{FBP_0} is a
non-negative continuous function $w$ in $\Omega$ such that
\begin{itemize}
\item $\Delta w \leq 0$ (resp. $\Delta w \geq 0$) in $\Omega^+(w)$;
\item If $x_0\in F(w)$ and there is a disk $B \subseteq \Omega^+(w)$ (resp. $B
\subseteq\{w=0\}$) that touches $F(w)$ at $x_0$, then near $x_0$ in
$B$ (resp. $B^c$), in every non-tangential region,
\begin{equation*}
w(x) = \alpha \langle x-x_0,\nu\rangle^+ + o(|x-x_0|) \quad
\text{for some} \quad \alpha \leq 1 \text{ (resp. }\alpha \geq 1),
\end{equation*}
where $\nu$ denotes the inner (resp. outer) unit normal to $\del B$
at $x_0$.
%If the inequality above is strict, we call $w$ a \textbf{strict}
%super (resp. sub) solution.
\end{itemize}
A function $w$ is a \textbf{viscosity solution} if w is both a
viscosity super- and subsolution.
\end{definition}

The class of viscosity solutions is well-suited for taking uniform
limits in compact sets.

\begin{lemma}[Limit of viscosity solutions]\label{lemma_limitvisco} Let $u_k\in C(\Omega)$ be a sequence of viscosity solutions of \eqref{FBP_0} in $\Omega$ such
that $u_k\rightarrow u$ uniformly and $u$ is Lipschitz continuous.
Then $u$ is also a viscosity supersolution of \eqref{FBP_0} in
$\Omega$. If, in addition, $\overline{\Omega^{+}(u_k)} \rightarrow
\overline{\Omega^{+}(u)}$ locally in the Hausdorff distance, then
$u$ is a viscosity subsolution, as well.
\end{lemma}
\begin{proof}
Clearly $\Delta u = 0$ in $\Omega^{+}(u)$, so we only need to check
that the appropriate free boundary conditions are satisfied.

Let us show that $u$ satisfies the viscosity supersolution condition
at the free boundary. Assume there is a disk $B$ touching $x_0\in
F(u)$ from the positive phase. Without loss of generality, $x_0 = 0$
and the unit normal of $\del B$ at $0$ is $\nu = e_2$. According to
\cite[Lemma 11.17]{CafSalsa}, $u$ has the linear behaviour:
\begin{equation*}
    u(x) = \alpha x_2 + o(|x|) \quad
    \text{in non-tangential regions of } B
\end{equation*}
for some $0<\alpha < \infty$ where $\nu$ denotes the inner unit
normal to $\del B$ at $x_0$. Claim that $\alpha \leq 1$. Fix $\eps
>0$ small. If we blow up at $0$,
\begin{equation*}
u_{\lambda}(x) := \lambda^{-1}u(x_0+\lambda x)  \rightarrow \alpha
x_2 \quad \text{in } B_1 \cap \{x_2 >\eps\} \quad \text{uniformly as
} \lambda \to 0.
\end{equation*}
Denote $(u_k)_{\lambda}(x) = \lambda^{-1} u(\lambda x)$ the dilate
of $u_k$ at $0$. By the uniform convergence of $u_k$ to $u$, for
some fixed small enough $\lambda>0$
\begin{equation}\label{viscolemma_clos1}
    |(u_k)_{\lambda}(x) - \alpha x_2| < \alpha\eps/2 \quad
    \text{in } B_1 \cap \{x_2>\eps\} \quad  \text{for all large enough } k.
\end{equation}
Consider the perturbation $D_t$ of the domain $B_1 \cap \{x_2>\eps
\}$ defined by
\begin{equation*}
    D_t = \{x\in B_1: x_2 > \eps - t \eta(x_1)\},
\end{equation*}
where $0\leq \eta(x_1)\leq 1$ is a smooth bump function supported in
$|x_1|<1/2$ with $\eta(x_1)=1$ for $|x_1|\leq 1/4$. We know that
$D_0 \Subset \Omega^+((u_k)_{\lambda})$ and since $0\in
F(u_{\lambda})$
\begin{equation}\label{viscolemma_clos2}
F((u_k)_{\lambda})\cap B_{\eps} \neq \emptyset.
\end{equation}
for all large enough $k$. Pick a $k$ such that both
\eqref{viscolemma_clos1} and \eqref{viscolemma_clos2} hold. Then for
some \mbox{$0< t_0 < 2\eps$} the domain
$D_{t_0}\subseteq\Omega^+((u_k)_{\lambda})$ will touch
$F((u_k)_{\lambda})$ at some $p\in F((u_k)_{\lambda}) \cap
\{|x_1|<1/2 \}$. Define a harmonic function $v$ in $D_{t_0}$ with
boundary values
\begin{equation*}
    v(x) = \left\{ \begin{array}{rcl} \alpha x_2 - \alpha\eps & \text{on} & \del B_1 \cap \{x_2 > \eps \} \\ 0 & \text{on} & B_1\cap \{x_2 = \eps - t_0 \eta(x_1) \}\end{array} \right.
\end{equation*}
Thus, by the maximum principle $v \leq (u_k)_{\lambda}$ in
$D_{t_0}$, so that near $p$ in non-tangential regions of $D_{t_0}$,
for some $\tilde{\alpha}\leq 1$
\begin{equation*}
    v(x)\leq (u_k)_{\lambda}(x) = \tilde{\alpha}\langle x-p, \nu(p) \rangle +
    o(|x-x_0|),
\end{equation*}
where $\nu(p)$ is the inner normal to $\del D_{t_0}$ at $p$. On the
other hand, a standard perturbation argument gives $v_{\nu}(p) =
\alpha + O(\eps)$. Since $\eps$ is arbitrary, we conclude $\alpha
\leq 1$.

Let us now assume that $\overline{\Omega^{+}(u_k)} \rightarrow
\overline{\Omega^+(u)}$ in the Hausdorff distance and show that $u$
satisfies the viscosity subsolution condition at the free boundary.
Let there be a disk $B$ touching $F(u)$ at $x_0$ from the zero
phase. Without loss of generality, $x_0 = 0$ and the unit outer
normal at $\del B$ is $e_2$. According to \cite[Lemma
11.17]{CafSalsa}, for some $0\leq \beta < \infty$
\begin{equation*}
    u(x)\leq \beta x_2^+ + o(|x|).
\end{equation*}
Given $\eps>0$ we can dilate $u$ and $u_k$ near $0$ sufficiently, so
that
\begin{equation*}
    (u_k)_{\lambda}(x) \leq u_{\lambda}(x) + \eps/2 \leq \beta
    x_2^{+} +\eps \quad \text{in } B_1
\end{equation*}
for some fixed large $\lambda$ and all large enough $k$. Moreover,
since $\Omega^{+}((u_k)_{\lambda})\rightarrow
\Omega^+(u_{\lambda})$, we can choose $k$ large enough such that
\begin{equation*}
    \Omega^{+}((u_k)_{\lambda})\cap B_1 \Subset \{x_2 > -\eps/2 \} \quad
    \text{and} \quad F((u_k)_{\lambda})\cap B_{\eps/2} \neq \emptyset.
\end{equation*}
Let $E_t$ be the domain
\begin{equation*}
    E_t = \{x\in B_1: x_2 > -\eps + t \eta(x_1)\}
\end{equation*}
and note that for some $0<t_0<2\eps$, $E_{t_0} \supseteq
\Omega^{+}((u_k)_{\lambda})\cap B_1$ and $\del E_{t_0}$ touches
$F((u_k)_{\lambda})\cap B_1$ at some point $q \in F((u_k)_{\lambda}
\cap \{|x_1|<1/2\}$. Define a harmonic function $w$ in $E_{t_0}$
having boundary values:
\begin{equation*}
        w(x) = \left\{ \begin{array}{rcl} \beta x_2^+ + \min\big((2(x_2+\eps)^+),\eps\big) & \text{on} & \del B_1 \cap \{x_2 > -\eps \} \\ 0 & \text{on} & B_1\cap \{x_2 = -\eps + t_0 \eta(x_1) \}\end{array} \right.
\end{equation*}
Thus, the maximum principle implies that near $q$, in non-tangential
regions of $\Omega^{+}((u_k)_{\lambda})$,
\begin{equation*}
    w(x)\geq (u_k)_{\lambda}(x) = \tilde{\beta}\langle x-q, \nu(q)\rangle +
    o(|x-x_0|),
\end{equation*}
for some $\tilde{\beta}\geq 1$. Hence, $w_{\nu}(q) \geq
\tilde{\beta}\geq 1$. On the other hand, a standard perturbation
argument gives $w_{\nu}(q) = \beta + O(\eps)$. Since $\eps$ is
arbitrary, we conclude that $\beta \geq 1$.
\end{proof}

%In what follows, for the sake of convenience, we use the term ``weak
%solution" to refer to the function being both a variational and a
%viscosity solution.
%\begin{definition} The function  $u:\Omega \to \real$ a \textbf{weak
%solution} of \eqref{FBP_0} if $u$ is \textbf{both a variational and
%a viscosity solution}.
%\end{definition}

\section{Characterization of blow-downs and
blow-ups.}\label{Sec_Blows}

The notion of a variational solution is incredibly useful precisely
because it admits the application of the powerful Weiss Monotonicity
Formula.
\begin{lemma}[Weiss' Monotonicity Formula, Theorem 3.1 in \cite{Weiss1}]\label{lemma_WeissMono} Let $u$ be a variational
solution of \eqref{FBP_0} in $\Omega\subseteq\real^n$ and that
$B_R(x_0)\subseteq \Omega$. Then
\begin{equation}\label{WeissMonoFctl}
    \Phi(u,r) := r^{-n}\int_{B_r(x_0)}\big(|\nabla u|^2 +
    \chrc{\{u>0\}}\big) ~dx - r^{-n-1} \int_{\del B_r(x_0)} u^2
    ~d\mathcal{H}^{n-1}
\end{equation}
satisfies the monotonicity formula
\begin{equation}\label{WeissMonoFla}
    \Phi(u,r_2) - \Phi(u,r_1) = \int_{B_{r_2}(x_0)\setminus B_{r_1}(x_0)} 2|x|^{-n-2} \big(\nabla u \cdot (x-x_0) - u\big)^2
    ~dx \geq 0
\end{equation}
for $0<r_1<r_2< R$.
\end{lemma}

\begin{lemma}\label{lemma_varibdown} Let $u$ be a variational
solution of \eqref{FBP_0} in $\real^n$ which is globally Lipschitz.
Assume $0 \in F(u)$ and let $v$ be any limit of a uniformly
convergent on compacts subsequence of
\begin{equation*}
    v_j(x) = R_{j}^{-1}u(R_jx)
\end{equation*}
as $R_j \to \infty$. Then $v$ is Lipschitz continuous and
homogeneous of degree one.
\end{lemma}
\begin{proof}
Denote $v_j(x) = R_j^{-1} u(R_j x)$ and note that $v_j$ are also
global variational solutions of \eqref{FBP_0} and $\Phi(v_j, r) =
\Phi(u, rR_j)$. According to Lemma \ref{lemma_WeissMono} the
quantity $\Phi(u, R)$ is non-decreasing as $R\to \infty$ and,
moreover, it is uniformly bounded since $u$ is Lipshitz continuous.
Hence, for any fixed $0<r_1<r_2$
\begin{equation*}
    0 = \lim_{j\to\infty} \big( \Phi(u, r_2 R_j) - \Phi(u,r_1 R_j) \big) =\lim_{j\to\infty} \big(\Phi(v_j, r_2) -
    \Phi(v_j,r_1)\big)
\end{equation*}
and \eqref{WeissMonoFla} yields
\begin{equation*}
   \lim_{j\to\infty}\int_{B_{r_2}\setminus B_{r_1}} 2|x|^{-n-2} \big(\nabla v_j \cdot x - v_j\big)^2
    ~dx = 0.
\end{equation*}
Possibly passing to a subsequence such that $\nabla v_j
\rightharpoonup \nabla v$ weakly in $L^2$, the lower semicontinuity
of the $L^2$-norm with respect to weak convergence implies
\begin{equation*}
    \int_{B_{r_2}\setminus B_{r_1}} 2|x|^{-n-2} \big(\nabla v \cdot x - v\big)^2
    ~dx  = 0.
\end{equation*}
Thus, $\nabla v \cdot x = v$ a.e. whence it is a standard exercise
to conclude that $v$ is homogeneous of degree one.

%Denote $v_j(x) = R_j^{-1} u(R_j x)$. Then $v_j$ are globally
%Lipschitz and non-degenerate variational solutions of \eqref{FBP_0},
%and Proposition \ref{prop_limitofsolns} implies that $v$ is also a
%globally Lipschitz, non-degenerate variational solution of
%\eqref{FBP_0}. Furthermore, since $|\nabla v_k|^2 \to |\nabla v|^2$,
%$\chrc{\{v_k>0\}}\to\chrc{\{v>0\}}$ in $L^1_{\text{loc}}(\real^n)$
%and $v_k\to v$ uniformly on compacts,
%\begin{equation*}
%    \lim_{j\to\infty}\Phi(v_j,r) = \Phi(v,r) \quad \text{for any }
%    r>0.
%\end{equation*}
%On the other hand, note that $\Phi(v_j, r) = \Phi(u, rR_j)$. By
%Lemma \ref{lemma_WeissMono} the quantity $\Phi(u, rR_j)$ is
%non-decreasing as $R_j\nearrow \infty$; moreover, it is bounded as
%$u$ is Lipschitz. Hence, $\Phi(u, rR_j) \nearrow M$ as $j\to
%\infty$, for some $0\leq M <\infty$ independent of $r$. Therefore,
%\begin{equation*}
%    \Phi(v, r) = \lim_{j\to\infty}\Phi(v_j,r)
%    =\lim_{j\to\infty}\Phi(u,rR_j) = M \quad \text{for every } r>0.
%\end{equation*}
%Weiss' Monotonity Formula \eqref{WeissMonoFla} applied to $v$ and
%$x_0 = 0$ then yields
%\begin{equation*}
%\nabla v \cdot x - v = 0 \quad \text{a.e. }x
%\end{equation*}
%whence it is a standard exercise to conclude that $v$ is homogeneous
%of degree one.
\end{proof}

\begin{prop}[Characterization of blowdowns]\label{prop_blowdowns} Let $u$ be both a viscosity and a variational solution of \eqref{FBP_0} in $\real^2$, which is
Lipschitz-continuous and non-degenerate. Assume $0 \in F(u)$ and let
$v$ be any limit of a uniformly convergent on compacts subsequence
of
\begin{equation*}
    v_j(x) = R_{j}^{-1}u(R_jx)
\end{equation*}
as $R_j \to \infty$. Then $v$ is either $V_1(x) = x_2^+$ or $V_2(x)
= s |x_2|$ for some $0< s\leq 1$ in an appropriately chosen
Euclidean coordinate system.
\end{prop}
\begin{proof}
As a consequence of Proposition \ref{prop_limitofsolns}, Lemma
\ref{lemma_limitvisco} and Lemma \ref{lemma_varibdown} applied to
the sequence $v_j$ we conclude that $v$ is a Lipschitz continuous,
non-degenerate, viscosity and variational solution of \eqref{FBP_0},
which is homogeneous of degree 1. Thus, after possibly rotating the
coordinate axes
\begin{equation*}
    v(x) = c_1 x_2^+ + c_2 x_2^-,
\end{equation*}
where $c_1 \geq c_2 \geq 0$. We have the following two cases.

\paragraph{\textbf{Case 1} ($c_2 = 0$)} By non-degeneracy we must have $c_1 > 0$ and
since every point $x_0\in F(v) = \{x_2=0\}$ has a tangent disk from
both the positive and zero set of $v$, then $c_1 =1$.

\paragraph{\textbf{Case 2} ($c_2 >0$)} Every point $x_0 \in F(v) = \{x_2=0\}$ has a
tangent disk from the positive phase of $v$ only, so from the fact
that $v$ is a viscosity solution we can just conclude that $1\geq
c_1 \geq c_2$. On the other hand, $v$ is also a variational solution
and an easy computation gives
\begin{equation*}
    0=L[v](\phi) = (c_1^2 - c_2^2)\int_{\real}\phi_2(x_1,0) dx_1
\end{equation*}
for any $\phi = (\phi_1, \phi_2)\in C^1_c(\real^2;\real^2).$ Thus,
$c_1 = c_2=s$.
\end{proof}

Exactly analogous arguments apply to \emph{blow-up} limits of
Lipschitz continuous weak solutions, so we have the analogous
characterization:
\begin{prop}[Characterization of blow-ups]\label{prop_blowups} Let $u$ be a % weak solution
both a viscosity and a variational solution of \eqref{FBP_0} in
$\Omega\subseteq\real^2$, which is Lipschitz-continuous and
non-degenerate. Assume $0 \in F(u)$ and let $v:\real^2\to\real$ be
any limit of a uniformly convergent on compacts subsequence of
\begin{equation*}
    v_j(x) = \eps_{j}^{-1}u(\eps_jx)
\end{equation*}
as $\eps_j \to 0$. Then $v$ is either $V_1(x) = x_2^+$ or $V_2(x) =
s |x_2|$ for some $0< s\leq 1$ in an appropriately chosen Euclidean
coordinate system.
\end{prop}

\section{Auxiliary Lemmas.}\label{Sec_AuxLemmas}

\begin{lemma}\label{lemma_gradEst} Let $u$ be a classical solution of \eqref{FBP_0} in
a domain $B_{2}$, which has Lipschitz norm $L$ and such that
\begin{equation}\label{lemma_gradEstIneqHypo}
    \left|u(x) - s|x_2|\right|<\eps \quad \text{in} \quad B_2
\end{equation}
for some $0<s\leq 1$ and some small $\eps>0$. Then there exists a
universal constant $c>0$ such that
\begin{equation*}
    \|\nabla u\|_{L^{\infty}(B_{1/2})} \leq 1 + cL\sqrt{\eps}.
\end{equation*}
\end{lemma}

\begin{proof}
Assumption \eqref{lemma_gradEstIneqHypo} implies $B_2^+(u_R)\subset
\{|x_2|>\eps/s\}$. Thus, at any $p \in \del B_1\cap \{|x_2|>
2M\eps/s\}\}$ for a large $M\leq s/2\eps$, we have $B_{M\eps/s}(p)
\subset B_2^+(u)$ so that $u-s|x_2|$ is harmonic in
$B_{M\eps/s}(p)$. Hence,
\begin{equation*}
    |\nabla u(p) - s\nabla|x_2|(p)|\leq \frac{c'}{M\eps/s}\fint_{\del
    B_{M\eps/s}}\big|u-s|x_2|\big| ~ d\mathcal{H}^1 \leq
    \frac{c'}{M/s},
\end{equation*}
which in turn leads to
\begin{equation}\label{lemma_gradEstIneqGrad}
    |\nabla u(p)|^2 \leq \left(s+\frac{c'}{M/s}\right)^2 \leq 1+ \frac{3c'}{M/s}.
\end{equation}
Define the function $v:B_2\rightarrow \real$
\begin{equation*}
    v = \left\{\begin{array}{lr} |\nabla u|^2 -1 & \text{ in } B_2^+(u) \\
0 & \text{otherwise.}
    \end{array}\right.
\end{equation*}
Then $v$ is continuous in $B_2$ and since
\begin{equation*}
    \Delta |\nabla u|^2 = 2|D^2 u|^2 + 2 \nabla (\Delta
    u)\cdot \nabla u =2|D^2 u|^2 \geq 0 \quad \text{in }
    B_2^+(u),
\end{equation*}
$v$ is subharmonic in $B_2^+(u)$. Let $v_h:\overline{B_1}\to \real$
be the harmonic function whose boundary values on $\del B_1$ are
given by
\begin{equation*}
    v_h(x) = \max\{v(x), 3c's/M\} \quad x\in \del B_1.
\end{equation*}
By the maximum principle, $v_h > 0$ in $B_1$ and $v_h\geq v$ in
$B_1^+(u)$, whence $v\leq v_h$ in $B_1$. By Poisson's formula, for
any $x\in B_{1/2}$
\begin{align}
    v_h(x)\leq c \int_{\del B_1} v_h ~ d\hsd & = c\left(\int_{\del B_1 \cap \{|x_2|\leq
    M\eps/s\}}v_h ~d\hsd + \int_{\del B_1 \cap \{|x_2|>
    M\eps/s\}} v_h ~d\hsd \right) \notag \\
    & \leq \tilde{c}L^2\eps M/s + \frac{\tilde{c}}{M/s},
    \label{lemma_gradEst^v_h}
\end{align}
where the last inequality is a consequence of
\eqref{lemma_gradEstIneqGrad} and the Lipschitz control of $u$.
%\[
%\|\nabla u_R\|_{L^{\infty}(B_2)} = \| \nabla
%u\|_{L^{\infty}(B_{2R})} = L.
%\]
Choosing $M = s/(\sqrt{\eps}L\})$ yields
\begin{equation*}
    v \leq v_h \leq 2\tilde{c}L\sqrt{\eps} \quad \text{in } B_{1/2}
\end{equation*}
which is the desired estimate.
\end{proof}

\begin{lemma}\label{Lemma_TwoPlaneBdwnGrad} Let $u_k$ be a sequence of
classical solutions of \eqref{FBP_0} in $B_{R_k}$, $R_k\nearrow
\infty$ that are uniformly Lischitz continuous and assume the
sequence converges uniformly on compact subsets of $\real^2$ to
$u:\real^2\to\real$ with $0\in F(u)$. If a blowdown limit of $u$
\begin{equation*}
    u_{R_j}(x) = R_{j}^{-1} u(R_j x) \to s|x_2| \quad \text{uniformly on
    compacts} \quad \text{as } R_{j}\to\infty,
\end{equation*}
for some $0<s\leq 1$, then
\begin{equation*}
    |\nabla u| \leq 1 \quad \text{a.e.}
\end{equation*}
\end{lemma}
\begin{proof}
Fix $\eps>0$ and find $j$ large enough so that
\begin{equation*}
    |u_{R_j} - s|x_2|| < \eps/2 \quad \text{in } B_2.
\end{equation*}
Then for all large enough $k$, such that $|u_{R_j} -
(u_k)_{R_j}|<\eps/2$ in $B_2$, we have
\begin{equation*}
    |(u_k)_{R_j} - s|x_2|| < \eps,
\end{equation*}
so that Lemma \ref{lemma_gradEst} yields the estimate
\begin{equation*}
    \|\nabla u_k\|_{L^{\infty}(B_{R_j/2})} = \|\nabla
    (u_k)_{R_j}\|_{L^{\infty}(B_{1/2})} \leq 1 + C \sqrt{\eps},
\end{equation*}
where $C$ is a bound on the Lipschitz norm of $u_k$. At every $x\in
\{u>0\}$, for $d= u(x)$, the disk $B_{d/2C}(x)\subseteq \{u>0\}$ as
well as $B_{d/2C}(x)\subseteq \{u_k>0\}$ for all $k$ large enough.
Since $u_{k}\to u$ uniformly in $B_{d/2C}(x)$, where the functions
are harmonic, we also get $\nabla u_k(x) \to \nabla u(x)$ as
$k\to\infty$. Since $\nabla u(x) = 0$ a.e. $x \in \{u=0\}$,
\begin{equation*}
    \|\nabla u\|_{L^{\infty}(B_{R_j/2})} = \lim_{k\to\infty}\|\nabla u_k\|_{L^{\infty}(B_{R_j/2})}  \leq 1+ C \sqrt{\eps}.
\end{equation*}
Letting $R_{j}\to\infty$, followed by $\eps\to 0$, yields the
result.
\end{proof}

\begin{lemma}\label{lemma_FBCurv} Let $u$ be a classical solution
of \eqref{FBP_0} in $\Omega\subseteq \real^2$. Then the signed
curvature $\kappa$ of $F(u)$ is given by
\begin{equation*}
    \kappa = - \frac{1}{2}\frac{\del(|\nabla u|^2)}{\del\nu},
\end{equation*}
where $\nu$ is the unit normal pointing towards $\Omega^+(u)$.
\end{lemma}
\begin{proof}
The curvature of a level set of a function $v$ at a point where
$|\nabla v| \neq 0$ is given by %\cite{Gray}
\begin{equation*}
    \kappa = \text{div}\left(\frac{\nabla v}{|\nabla v|}\right).
\end{equation*}
(Note that $\kappa >0$ when the curvature vector points in the
direction of $-\nabla v$, e.g. the curvature of the $0$-level set of
$v(x,y)=\log (x^2+y^2)$ is positive $1$). Since in a small enough
neighborhood $U$ of each $x\in F(u)$ we can define a harmonic $v$
which agrees with $u$ on $U\cap \Omega^+(u)$, the curvature of
$F(u)$ is given by the curvature of the $v=0$ level set. Using in
addition $|\nabla v|(x) = 1$, we compute
\begin{equation*}
    \kappa = \frac{\Delta v}{|\nabla v|} - \frac{\nabla v \cdot \nabla |\nabla v|}{|\nabla v|^2} = -\frac{(|\nabla v|^2)_{\nu}}{2|\nabla
    v|^2} = - (|\nabla u|^2)_{\nu}/2.
\end{equation*}
\end{proof}

\begin{remark}\label{rem_poscurv}  If $u$ is a classical solution of \eqref{FBP_0} in $\Omega\subseteq
\real^2$ with $|\nabla u|<1$ in $\Omega^+(u)$, then $F(u)$ has
strictly positive curvature.
\end{remark}
\begin{proof} The result follows immediately from Lemma
\ref{lemma_FBCurv} after an application of the Hopf Lemma to
$|\nabla u|^2$, which is subharmonic in $\Omega^+(u)$:
 \begin{equation*}
    \Delta |\nabla u|^2 = 2 |D^2 u|^2 + 2 \nabla u\cdot \nabla (\Delta
    u) = 2 |D^2 u|^2 \geq 0 \quad \text{in} \quad \Omega^+(u) .
 \end{equation*}
\end{proof}

%Let us introduce some notation. If $u$ is a classical solution of
%\eqref{FBP_0} in $\Omega\in \real^2$ and $U$ is a connected
%component of $\Omega^+(u)$, let $\mathcal{C}_U(p,q)$ denote the set
%of simple, piecewise $C^1$ curves $\gamma:[0,1]\rightarrow
%\overline{U}$ whose ends are at $\gamma(0)=p$ and $\gamma(1)=q$,
%whose interior $\gamma(0,1)\subset U$ and which intersect $F(u)$
%transversally at $p$ and $q$. Then define the \emph{inner distance}
%within $U$ by:
%\begin{equation*}
%    d_U(p,q) := \inf\{|\gamma|: \gamma \in \mathcal{C}_U(p,q)\},
%\end{equation*}
%where $|\gamma|$ denotes the length of $\gamma$.

\begin{lemma}\label{lemma_noShortNonFBseg}
Let $u$ be a classical solution of \eqref{FBP_0} in
$\Omega\subset\real^2$, whose Lipschitz norm is $L<\infty$. If
$V\subseteq \Omega^+(u)$, $V\Subset\Omega$ is a bounded, piecewise
$C^1$ domain, then
\begin{equation*}
    L^{-1}\hsd(\del U\cap F(u)) \leq \hsd(\del U \setminus F(u)).
\end{equation*}
\end{lemma}
\begin{proof}
Applying the Divergence Theorem in $U$:
\begin{equation*}
    0=-\int_U \Delta u ~dx = \int_{\del U\cap F(u)} u_{\nu} ~d\hsd +
    \int_{\del U \setminus F(u)} u_{\nu} ~ d\hsd = \hsd(\del U\cap F(u)) + \int_{\del U \setminus F(u)} u_{\nu}~d\hsd,
\end{equation*}
where $u_{\nu}$ is the inner unit normal to $\del U$. The result
then follows from $|u_{\nu}|\leq L$ \quad $\hsd$-a.e.\! on $\del U$.
\end{proof}

\begin{lemma}\label{lemma_atmostTwoComps}
Let $u$ be a classical solution of \eqref{FBP_0} in $B_3$, which is
Lipschitz continuous with norm $L$ and for which assumption
\eqref{topo_assmpt} is satisfied. There exists $\delta = \delta(L)
>0$ small enough such that if
\begin{equation*}
    \{u=0\}\subset B_3\cap\{|x_2|<\delta \}
\end{equation*}
there are at most two connected components of $B_2^+(u)$ which
intersect $B_1$, namely the connected component(s) containing
$N=(0,1)$ and $S=(0,-1)$.
\end{lemma}
\begin{proof}
%Obviously, the connected component(s) of $B_2^+(u)$ containing
%$p=(0,1)$ and $q=(0,-1)$ intersect $B_1$.
Consider a connected component $U$ of $B_{2}^+(u)$ that contains
neither $N$ nor $S$; then it must be that $U\subseteq
B_2\cap\{|x_2|<\delta \}$. Assuming that $U$ intersects $B_1$, by
assumption \eqref{topo_assmpt} we have $\hsd(\del U\cap F(u)) \geq
2$. On the other hand, $U$ is a piecewise $C^1$ domain with $\del
U\setminus F(u) \subseteq \del B_2 \cap \{|x_2|<\delta \}$, so that
$\hsd(\del U\setminus F(u)) \leq c\delta$ for some numerical
constant $c>0$. But then by Lemma \ref{lemma_noShortNonFBseg}
\begin{equation*}
    2/L \leq L^{-1} \hsd(\del U \cap F(u)) \leq \hsd(\del U\setminus F(u)) \leq
    c\delta,
\end{equation*}
which would be impossible if $\delta < 2/Lc$.
\end{proof}

\begin{lemma}\label{lemma_fbgraphs} Let $u$ be a classical solution of \eqref{FBP_0} in
$B_4$ for which assumption \eqref{topo_assmpt} is satisfied. Assume
further that $(0,1)$ and $(0,-1)$ belong to two separate connected
components of $B_2^+(u)$. Then there exists $\delta_0 > 0$ small
enough such that if for any $0<\delta \leq \delta_0$
\begin{equation*}
    \{u=0\} \subseteq \{|x_2|<\delta \},
\end{equation*}
the free boundary $F(u)$ inside $\{|x_1|<1/2\}$ consists of two
disjoint graphs:
\begin{equation*}
    F(u)\cap \{|x_1|<1/2\} = \{x_2 = \phi_+(x_1): |x_1|<1/2 \} \sqcup
    \{x_2=\phi_-(x_1): |x_1|<1/2 \},
\end{equation*}
for which $\phi_+>\phi_-$ and
\begin{equation*}
    \|\phi_{\pm}\|_{C^{1,\alpha}(-1/2,1/2)} \leq C\delta
\end{equation*}
for some numerical positive constants $C$, $0<\alpha<1$.
\end{lemma}
\begin{proof}
By Proposition \ref{prop_lipbd}, $\|\nabla u\|_{L^{\infty}(B_2)}\leq
L$ for some numerical constant $L$, so that by Lemma
\ref{lemma_atmostTwoComps}, there exists a small enough $\delta_0>0$
such that it is precisely the connected components $U_+$ and $U_-$
of $B_2^+(u)$, containing $(0,1)$ and $(0,-1)$ respectively, that
intersect $B_1$. Define the two functions $u_{+}$ and $u_{-}$ on
$B_1$ by $u_{\pm} = u \chrc{U_{\pm}\cap B_1}$. Then each $u_{\pm}$
is a classical solution of \eqref{FBP_0} in $B_1$ whose free
boundary $F(u_{\pm})$ is contained in a flat strip $|x_2|<\delta$
with $u_+ = 0$ in $B_1\cap \{x_2 <-\delta\}$ and $u_{-}=0$ in
$B_1\cap \{x_2>\delta\}$. By the classical result of Alt and
Caffarelli \cite{AC}, in $|x_1|< 1/2$ the free boundary $F(u_{\pm})$
is the graph of a function $\phi_{\pm}:(-1/2,1/2)\to \real$, which
satisfies
\begin{equation*}
    \|\phi_{\pm}\|_{C^{1,\alpha}(-1/2,1/2)} \leq C\delta
\end{equation*}
for some $\alpha>0, C>0$. Noting that $F(u)\cap B_1 = F(u_{+})\sqcup
F(u_{-})$, we are done.
\end{proof}

\section{Characterization of the limit.}\label{Sec_Classfn}

Recall the setup. We have a sequence $\{u_k\}$ of classical
solutions of \eqref{FBP_0} in expanding disks $B_{R_k}$,
$R_k\nearrow\infty$ with $0\in F(u_k)$. Because of Proposition
\ref{prop_lipbd}, $u_k$ are uniformly Lipschitz on compact subsets
of $\real^2$, so that up to a subsequence, $u_k$ converges uniformly
on compacts to some $u:\real^2\to\real$. Moreover, since $u_k$ are
uniformly non-degenerate by Proposition \ref{prop_supnondeg}, and,
trivially also, weak solutions (variational and viscosity), then by
Proposition \ref{prop_limitofsolns} and Lemma
\ref{lemma_limitvisco}, $u$ is a global weak solution, which is
Lipschitz continuous and non-degenerate. Thus, by Propositions
\ref{prop_blowdowns} and \ref{prop_blowups}, $u$ blows down/blows up
at a free boundary point to a half-plane or a wedge solution.

We shall show that, in for appropriately chosen Euclidean
coordinates, $u$ has to be one of the four:

\begin{itemize}
\item a \emph{half-plane solution} $P(x) = x_2^+$
\item a \emph{two-plane solution} $W_b(x) = x_2^+ + (x_2-b)^-$, for
some $b < 0$
\item a \emph{wedge solution} $W_0(x) = |x_2|$
\item a \emph{hairpin solution} $H_a$, whose $\{H_a>0\}=\{(x_1, x_2):
|ax_1-(1+\pi/2)|<\pi/2+\cosh(ax_2)\}$ for some $a>0$.
\end{itemize}

%\begin{theorem}\label{theorem_Main}
%Let $u_k$ be a sequence of classical solutions of \eqref{FBP_0} in
%$B_{R_k}$, $R_k\nearrow \infty$ with $0\in F(u_k)$, for which
%\eqref{topo_assmpt} is satisfied. The limit $u$ of any convergent
%subsequence is either a half-plane, a two-plane, a wedge, or a
%hairpin solution.
%\end{theorem}

The proof of the classification Theorem \ref{theorem_Main} is
realized in a sequence of propositions. Proposition
\ref{prop_halfplane} covers the scenario when $u$ blows down to a
half-plane solution, while Proposition \ref{prop_twoplane} covers
the case when the blowdown is a wedge solution. In the latter there
is a dichotomy, $u$ can be either a two-plane solution or satisfy
$|\nabla u|<1$ globally in its positive phase. The second scenario
is the more subtle one and its treatment is carried out in several
steps assembled under Proposition \ref{prop_MainTheoAux}: in the
first we employ a 2-point simultaneous blow-up to show that the free
boundary is smooth everywhere but possibly one point; in the second
step we prove the free boundary is actually smooth everywhere using
the strong geometric constraint of positive free boundary curvature
to argue that the zero phase contains a nontrivial sector with
vertex at the exceptional point (cf. Lemmas \ref{lemma_NoExcPoint}
and \ref{lemma_TwoCurves}); in the final step we establish that the
free boundary must consist of exactly two smooth proper arcs, so
that $u$ has to be a hairpin solution, according to \cite[Theorem
12]{Traizet}.

%Before we commence, let us state precisely what we mean by $F(u)$
%being smooth.
%\begin{definition}
%The free boundary $F(u)$ is smooth at $p$ if there exists an $r>0$
%such that for some right cylinder...
%\begin{equation*}
%    \{u>0\}\cap
%\end{equation*}
%\end{definition}

Let us commence with
\begin{prop}\label{prop_halfplane}
Let $u_k$, $u$ be as in Theorem \ref{theorem_Main} and assume $0\in
F(u)$. If a blowdown limit of $u$ at $0$ is a half-plane solution:
\begin{equation*}
    R_{j}^{-1}u(R_j x) \rightarrow x_2^+ \quad \text{as} \quad
    R_j\nearrow \infty,
\end{equation*}
(with coordinates chosen appropriately), then $u=x_2^+$ itself.
\end{prop}
\begin{proof}
Since $u_{R_j} \to x_2^+$ uniformly on compacts, the free boundary
$F(u)$ is asymptotically flat, i.e.
\begin{equation*}
    F(u)\cap B_{R_{j}} \subseteq \{|x_2|\leq \eps_j\}
\end{equation*}
with the aspect ratio $\eps_j/R_{j}\to 0$ as $j\to\infty$. By the
classical theorem of Alt and Caffarelli \cite{AC}, $F(u)$ has to be
the straight line $\{x_2=0\}$ and $u=x_2^+$.
\end{proof}

% Proposition Gradient Bound.
\begin{prop}\label{prop_twoplane} Let $u_k$, $u$ be as in Theorem \ref{theorem_Main} and assume $0\in
F(u)$ and that a blowdown limit of $u$ at $0$ is a wedge solution:
\begin{equation*}
    R_{j}^{-1}u(R_j x) \rightarrow s|x_2| \quad \text{as} \quad
    R_j\nearrow \infty,
\end{equation*}
(with coordinates chosen appropriately) for some $0<s\leq 1$. Then
either $u=x_2^+ + (x_2-b)^-$ for some $b\leq 0$, or
\begin{equation*}
    |\nabla u| < 1 \quad \text{in} \quad \{u >0\}.
\end{equation*}
\end{prop}
\begin{proof}
From Lemma \ref{Lemma_TwoPlaneBdwnGrad} we have the bound $|\nabla
u| \leq 1$ a.e. Noting that $|\nabla u|^2$ is a smooth subhamornic
function in $\{u>0\}$:
\begin{equation}\label{prop_twoplane_SubH}
    \Delta |\nabla u|^2/2 = 2|D^2 u|^2 \geq 0,
\end{equation}
the Strong Maximum Principle entails that if $|\nabla u|^2(x_0) = 1$
at some $x_0\in \{u>0\}$, then $|\nabla u|^2 \equiv 1$ in the entire
connected component $U$ of $x_0$. Equation
\eqref{prop_twoplane_SubH} implies that $|D^2 u|^2 = 0$ in $U$, so
that $u$ is an affine function in $U$. Thus $U$ is a half-plane, say
$U = \{x_2<b\}$ for some $b\leq 0$ and $u = (x_2-b)^-$ in $U$. The
latter also implies that $u$ has to blow down precisely to $|x_2|$,
i.e. $s=1$.

We shall now show that $v=u\chrc{\real^2\setminus U}$ is a viscosity
solution itself. Once this is established, we can apply the previous
Proposition \ref{prop_halfplane} to $v$ (as $v$ has to blow down to
$x_2^+$), so that $v$ will have to be $v=x_2^+$ itself. We will then
be able to conclude that $u = v + (x_2-b)^- = x_2^+ + (x_2-b)^-$.

Note that $v$ is trivially a viscosity supersolution (as $u$ is), so
let us focus on showing that $v$ is also a viscosity subsolution.
Let $p\in F(v)$ and let $B\subseteq \{v=0\}$ be a touching disk to
$F(v)$ at $p$ from the zero phase. If there exists a disk
$B'\subseteq B$ such that $\del B \cap \del B' = p$ and $B'\subseteq
\{u=0\}$, then the subsolution condition for $v$ will be inherited
from $u$. Otherwise, every $B'\subseteq B$ with $\del B \cap \del B'
= p$ will have to intersect the half-plane $U = \{x_2<b\}$, and thus
$B\subseteq U$ and $p\in \del U$. But then, applying Proposition
\ref{prop_blowups}, we see that any blowup of $u$ at $p$ will have
to equal $|(x-p)\cdot e_2|$, where $e_2$ is a unit vector in the
direction of $x_2$. Hence, near $p$
\begin{equation*}
    v(x) = ((x-p)\cdot e_2)^+ + o(|x-p|) \quad \text{in any nontangential region of } B^c,
\end{equation*}
so that the subsolution condition is satisfied again.

\end{proof}

\begin{prop}\label{prop_MainTheoAux} Let $u_k$, $u$ be as in Theorem \ref{theorem_Main}. Further assume that $|\nabla u| < 1$ in
$\{u>0\}$. Then $u$ is a hairpin solution.
\end{prop}
\begin{proof}
%Without loss of generality, $0\in F(u)$.
%Any blowdown limit of $u$ is a half-plane or a wedge solution. Since
%$F(u_k) \to F(u)$ locally in the Hausdorff distance, our hypothesis
%\eqref{prop_MainTheoAux_Asmpt} on the free boundary of $u_k$ is
%inconsistent with the conclusion of Proposition
%\ref{prop_halfplane}, so that any blow-down limit of $u$ must be a
%wedge solution. Moreover, by Proposition \ref{prop_twoplane}, as
%\eqref{prop_MainTheoAux_Asmpt} prevents $u$ from being a half-plane
%or a two-plane solution, we must have that
%\begin{equation}\label{prop_MainTheoAux_Grad}
%    |\nabla u| < 1 \quad \text{in } \{u>0\}.
%\end{equation}
We shall prove the proposition in three steps. In the first we show
that $F(u)$ is smooth everywhere but possibly one point. In the
second step we invoke Lemma \ref{lemma_NoExcPoint} to establish that
$F(u)$ is, in fact, smooth everywhere. In the final step we show
that $F(u)$ consists of two disjoint proper arcs, so that by a
result of Traizet \cite[Theorem 12]{Traizet} $u$ has to be a hairpin
solution.

\paragraph{\textbf{Step 1.}} In order to establish the claim of the first step above, we have
to prove that for no two distinct points $P_1,P_2 \in F(u)$ it can
happen that $u$ blows up to wedge solutions at both $P_1$ and $P_2$.
From this it follows that at every point of $F(u)$ but possibly one,
$u$ has to blow up to the only other alternative -- a half-plane
solution (according to Proposition \ref{prop_blowups}) so that
$F(u)$ is smooth there.

Assume the contrary. Denote $u_{\eps}(x):= \eps^{-1} u(\eps x)$ and
$(u_k)_{\eps}(x) := \eps^{-1}u_k(\eps x)$. % and let
%\begin{equation*}
%    S_{\lambda}(B_r(x), e) := \{y\in B_r(x): |(y-x)\cdot e|< \lambda\}
%\end{equation*}
%be the strip of width $2\lambda$, centered at $x$ inside $B_r(x)$,
%in the direction of the unit vector $e\in\real^2$.
If $u$ blows up to wedge solutions at $P_1$ and $P_2$, Proposition
\ref{prop_limitofsolns}a) implies there exist some unit vectors
$a_1$ and $a_2$ such that given an arbitrary small $\lambda>0$, one
can find a sequence $\eps_j\searrow 0$ small enough such that for
any $\eps = \eps_j$ small enough
\begin{equation*}
    d_{H}\left(\{u_{\eps}=0\}\cap
    \overline{B_4}(P_i/\eps), \{P_i/\eps + t a_i : |t|\leq 4\} \right) < \lambda/2 \qquad i=1,2.
\end{equation*}
Further, for that particular fixed $\eps$, Proposition
\ref{prop_limitofsolns}a) implies that for all $k$ large enough
\begin{align*}
    & d_{H}\left(\{(u_k)_{\eps}=0\}\cap \overline{B_4}(P_i/\eps), \{u_{\eps}=0\}\cap
    \overline{B_4}(P_i/\eps)\right) < \lambda/2 \qquad i=1,2.
%    & d_{H}\left(\{(u_k)_{\eps}=0\}\cap \overline{B_4}(P_2/\eps), \{u_{\eps}=0\}\cap
%    \overline{B_4}(P_2/\eps)\right) < \lambda/2
\end{align*}
As a consequence, for all $k$ large enough
\begin{equation}\label{prop_MainTheoAux_ZeroProx}
    d_H\left( \{(u_k)_{\eps}=0\}\cap \overline{B_4}(P_i/\eps), \{P_i/\eps + t a_i : |t|\leq 4\} \right)<\lambda \qquad i=1,2.
\end{equation}
For ease of notation, denote $v:=u_{\eps}$, $v_{k} :=
(u_{k})_{\eps}$ and $Q_i=P_i/\eps$, $i=1, 2$; let $b_i$ be the
vector $a_i$ rotated by $\pi/2$. According to Lemma
\ref{lemma_atmostTwoComps}, there are at most two connected
components of $B_2(Q_i)^+(v_k)$ that intersect $B_1(Q_i)$ -- the
one(s) that contain $N_i=Q_i+b_i$ and $S_i=Q_i - b_i$. We shall now
show that there has to be just one if $\lambda$ is small enough and
$k$ is large enough.

Assume $N_1$ and $S_1$ belong to two separate connected components
$U_{+,k}$ and $U_{-,k}$ of $B_2(Q_1)^+(v_k)$. Combining this with
\eqref{prop_MainTheoAux_ZeroProx} allows us to invoke Lemma
\ref{lemma_fbgraphs} and conclude that $F(v_k)\cap\{|(x-Q_1)\cdot
a_1| < 1/2\}$ consists of the graphs $\Sigma_{+,k}$ and
$\Sigma_{-,k}$ of some functions $\phi_{+,k} > \phi_{-,k}$ over the
line segment $I=\{Q_1+t a_1: |t|<1/2\}$:
\begin{equation*}
    F(v_k) \cap \{|(x-Q_1)\cdot a_1| < 1/2\} = \Sigma_{+,k} \sqcup
    \Sigma_{-,k} \quad \text{where}\quad \Sigma_{\pm,k} = \{y+ \phi_{\pm,k}(y)b_1: y\in I\}.
\end{equation*}
Moreover, the functions $\phi_{\pm,k}$ satisfy the uniform bound
\begin{equation*}
    \|\phi_{\pm,k}\|_{C^{1,\alpha}(I)} \leq C\lambda.
\end{equation*}
Hence, there exist $C^{1,\alpha}$ functions $\phi_{\pm}:
I\rightarrow \real$ and a subsequence $\phi_{\pm,k_l}$ such that
\begin{equation*}
    \phi_{\pm,k_l} \rightarrow \phi_{\pm} \quad
    \text{in } C^1(I) \quad \text{as } l\to \infty.
\end{equation*}
But since $F(v_k)\to F(v)$ locally in the Hausdorff distance, it
must be that $ F(v)\cap \{|(x-Q_1)\cdot a_1| < 1/2\}$ consists
precisely of the $C^{1,\alpha}$ graphs
\begin{equation*}
    \Sigma_{\pm} = \{y+\phi_{\pm}(y)b_1: y\in I\}
\end{equation*}
and
\begin{equation*}
    B_2(Q_1)^+(v) \cap \{|(x-Q_1)\cdot a_1| < 1/2\} = U_+\sqcup U_{-},
\end{equation*}
where $U_+=\{y+tb_1: t >\phi_+(y): y\in I\} \cap B_2(Q_1)$ and
$U_-=\{y+tb_1: t<\phi_-(y): y\in I\} \cap B_2(Q_1)$. Moreover, since
$v_k\chrc{U_{\pm,k}}$ are viscosity solutions in $B_2(Q_1)\cap
\{|(x-Q_1)\cdot a_1| < 1/2\}$ and $v_k\chrc{(U_{\pm})_k} \to v
\chrc{U_{\pm}}$ uniformly there, Lemma \ref{lemma_limitvisco}
implies that $v\chrc{U_{\pm}}$ are viscosity solutions there as
well. As their free boundary is $C^{1,\alpha}$, they are in fact
classical solutions. But $|\nabla v| < 1$ in $\{v>0\}$, so that
according to Corollary \ref{rem_poscurv}, both $\Sigma_{+}$ and
$\Sigma_-$ have positive curvature. However, this is impossible,
because $0\in \Sigma_{+}\cap\Sigma_{-}$, as $v$ blows up at $0$ to a
wedge solution.

Hence $N_1$ and $S_1$ belong to the same connected component of
$B_2(Q_1)^+(v_k)$ and similarly $N_2$ and $S_2$ belong to the same
connected component of $B_2(Q_2)^+(v_k)$ for $\lambda$ small enough
and all $k$ large enough.

Let
\begin{align*}
    \alpha_{L} &= \del B_4(Q_1) \cap \{x: |(x-Q_1)\cdot b_1| <\lambda, (x-Q_1)\cdot a_1 <0\}
    \\
    \alpha_{R} &= \del B_4(Q_1) \cap \{x: |(x-Q_1)\cdot b_1| <\lambda, (x-Q_1)\cdot a_1 >0\}.
\end{align*}
By our topological assumption $F(v_k)\cap \overline{B_4}(Q_1)$
consists of arcs whose ends lie on $\alpha_L\cup \alpha_R$. Define
$F_L$ (resp. $F_R$) to be the set of points of $F(v_k)\cap
\overline{B_4}(Q_1)$ that lie on arcs whose both ends are on
$\alpha_L$ (resp. $\alpha_R$). Then
\begin{equation*}
    F(v_k) \cap \overline{B_4}(Q_1) = F_L \sqcup F_R.
\end{equation*}
This is so, because the existence of an arc which has one end on
$\alpha_L$ and the other on $\alpha_R$ would contradict the fact
that $N_1$ from $S_1$ belong to the same connected component of
$B_2(Q_1)^+(v_k)$.

Note that $F(v_k) \cap \overline{B_4}(Q_1)$ consists of a finite
number of connected arcs. This follows from the analyticity of
$F(v_k)$ which implies that only finitely many connected arcs of
$F(v_k)$ can intersect $\del B_4(Q_1)$, each intersecting it
finitely many times. As a consequence, the sets $F_L$ and $F_R$ are
closed and being bounded, they are compact. Hence, there exists a
point $p\in F_L$ and a point $q\in F_R$ such that
\begin{equation*}
    |p-q| = \text{dist}(F_L, F_R) < 2\lambda,
\end{equation*}
where the bound follows from $\eqref{prop_MainTheoAux_ZeroProx}$.
Denote by $\gamma_L$ the arc of $F(v_k) \cap \overline{B_4}(Q_1)$
containing $p$, and by $\gamma_R$ -- the arc containing $q$.

Claim that the straight line segment $\tau_1 := \{(1-t)p+tq:
0<t<1\}$, connecting $p$ to $q$, lies in $B_4(Q_1)^+(v_k)$. Since
$p$ and $q$ realize the distance between $F_L$ and $F_R$, it must be
that $\tau_1 \cap F(v_k) = \emptyset$, so that either $\tau_1
\subseteq B_4(Q_1)^+(v_k)$ or $\tau_1 \subseteq\{v_k=0\}\cap
B_4(Q_1)$. The latter alternative, however, is impossible, since
$\gamma_L\cup\tau_1\cup\gamma_R$ would disconnect $N_1$ from $S_1$
in $B_2(Q_1)^+(v_k)$.

Let us look globally at the connected arc(s) of $F(v_k)$ that
contain $p$ and $q$. One possibility is that $p$ and $q$ belong to
the same arc. Let us argue that this is not the case. Denote by
$\gamma$ the arc of $F(v_k)$ with ends $p$ and $q$. Then $\gamma\cup
\tau$ is a simple closed arc and it encloses a piecewise $C^1$
Jordan domain in the positive phase. Applying Lemma
\ref{lemma_noShortNonFBseg} to it, we see that, as
$\hsd(\tau_1)<2\lambda$,
\begin{equation}\label{prop_MainTheoAux_Last}
    2\lambda L > \hsd(\tau_1) L \geq \hsd(\gamma),
\end{equation}
where $L$ denotes the Lipschitz norm of $v_k$. However, since
$\gamma_L\cup\tau_1\cup\gamma_R\subseteq \gamma\sqcup\tau_1$
connects $\alpha_L$ to $\alpha_R$,
\begin{equation*}
    \hsd(\gamma)+\hsd(\tau_1)\geq 2\sqrt{4^2-\lambda^2},
\end{equation*}
so that $\hsd(\gamma) > 7$ for all $\lambda$ small. Since $L$ is
bounded above by a universal constant, taking $\lambda$ small enough
leads to a contradiction in \eqref{prop_MainTheoAux_Last}.

Thus, we may assume that $p$ and $q$ belong to distinct arcs of
$F(v_k)$. We know that $v$ blows down to a wedge solution (otherwise
$|\nabla v(x)| = 1$ at some $x$) and without loss of generality, we
may assume the blowdown is exactly $s|x_2|$. Thus, for any
$\delta>0$ there exists $M=M(\delta)$ large enough such that for all
$k$ large enough
\begin{equation*}
    \{v_k=0\}\cap \overline{B_{M}} \subseteq \{|x_2|<\delta M\}.
\end{equation*}
Note that we may take $M$ large enough so that both $Q_1$ and $Q_2$
belong to $B_{M/3}$ and $\tau_1\subseteq B_{M/2}^+(v_k)$. Denote
\begin{equation*}
    \alpha_{L,M} = \del B_{M} \cap \{x_1<0, |x_2|<\delta M\} \qquad \alpha_{R, M} = \del B_{M} \cap \{x_1>0,
    |x_2|<\delta M\}.
\end{equation*}
and let $\gamma_p$ be the arc of $F(v_k)\cap \overline{B_M}$
containing $p$, and $\gamma_q$ -- the arc of $F(v_k)\cap
\overline{B_M}$ containing $q$. Let us show that $\gamma_p$ and
$\gamma_q$ cannot both have an end on $\alpha_{L,M}$ (and,
similarly, they cannot both have an end on $\alpha_{R,M}$). Assume
they do: let $\tilde{\gamma}_p\subseteq \gamma_p$ be the subarc
connecting $\alpha_{L,M}$ to $p$ and
$\tilde{\gamma}_q\subseteq\gamma_q$ be the subarc connecting
$\alpha_{L,M}$ to $q$ and let $\tilde{\alpha}$ be (smaller) circular
arc on $\del B_M$ from the end $\del B_M\cap \tilde{\gamma}_p$ to
the end $\del B_M \cap \tilde{\gamma}_q$. Then $\tilde{\alpha}\cup
\tilde{\gamma}_p\cup \tilde{\gamma}_q \cup \tau_1$ encloses a Jordan
domain $\tilde{O}$ and let $O\subseteq \tilde{O}^+(v_k)$ be the
connected component of $\tilde{O}^+(v_k)$ whose boundary contains
$\tau_1$. Applying Lemma \ref{lemma_noShortNonFBseg} to $O$, we
quickly reach a contradiction for small $\delta$, as
\begin{equation*}
    \hsd(F(v_k)\cap \del O) \geq
    \hsd(\tilde{\gamma}_p)+\hsd(\tilde{\gamma}_q) \geq M/2 + M/2 = M
\end{equation*}
while
\begin{equation*}
    \hsd(\del O\setminus F(v_k)) \leq \hsd(\tau_1)+\hsd(\tilde{\alpha}) <
    2\lambda + c\delta M < 1+ c\delta M.
\end{equation*}

Therefore, it must be that $\gamma_p$ has both its ends on
$\alpha_{L,M}$ while $\gamma_q$ has both its ends on $\alpha_{R,M}$.
Of course, an analogous scenario holds near $Q_2$ as well: we can
find a straight line segment $\tau_2\subseteq
B_4^+(Q_2)(v_k)\subseteq B_{M/2}^+(v_k)$ of length
$\hsd(\tau_2)<2\lambda$, one end of which belongs to a free boundary
arc with ends on $\alpha_{L,M}$ and the other contained in a free
boundary arc with ends on $\alpha_{R,M}$. Moreover,
$d(\tau_1,\tau_2)\approx d(Q_1, Q_2)=\eps^{-1}d(P_1,P_2)$ can be
taken to be of at least unit size
\begin{equation*}
    d(\tau_1,\tau_2)\geq 1
\end{equation*}
if $\eps$ is initially taken small enough. However, we can now
appeal to Lemma \ref{lemma_TwoPtsBlowup} below to rule out the
arising picture when $\lambda$ and $\delta$ are small enough. This
completes the first step of the proof.

\paragraph{\textbf{Step 2}.} We have just established that $F(u)$ is smooth everywhere but
possibly one point -- without loss of generality, this exceptional
point sits at the origin. Then each component of $F(u)\setminus 0$
is a smooth submanifold of $\real^2$, hence diffeomorphic to either
the circle $\mathbb{S}^1$ or the real line $\real$. Let us establish
that the former possibility does not arise. Assume that there is a
connected component of $F(u)\setminus 0$ that is a smooth, simple
closed curve $\alpha$. Choose $\delta >0$ small enough, such that
$\mathcal{N}_{2\delta}(\alpha) \cap (F(u)\setminus \alpha) =
\emptyset$ (such a $\delta$ exists since $\alpha$ is compact and
$F(u)\setminus \alpha$ is closed). But since $F(u_k) \to F(u)$
locally in the Hausdorff distance, for $K
:=\overline{\mathcal{N}_{3\delta/2}(\alpha)}$
\begin{equation*}
    F(u_k) \cap K \subseteq \mathcal{N}_{\delta}(F(u)\cap K) =
    \mathcal{N}_{\delta}(\alpha)
\end{equation*}
for all $k$ large enough. However, this is impossible, since by the
topological assumption \eqref{topo_assmpt}, the free boundary of
$u_k$ has to ``exit" $\mathcal{N}_{\delta}(\alpha)$:
\begin{equation*}
(F(u_k)\cap K) \setminus \mathcal{N}_{\delta}(\alpha) = F(u_k)\cap
\left(\overline{\mathcal{N}_{3\delta/2}(\alpha)}\setminus
\mathcal{N}_{\delta}(\alpha)\right) \neq \emptyset.
\end{equation*}
This places us in the assumptions of Lemma \ref{lemma_NoExcPoint}
below, through which we establish the smoothness of $F(u)$
everywhere.

\paragraph{\textbf{Step 3}.}

Each connected component of $F(u)$ is diffeomorphic to $\real$ and
thus, the image of some embedding $\beta:\real \to \real^2$. The
embedding has to be proper, i.e. $\lim_{t \to\pm \infty} \beta(t) =
\infty$. Otherwise, there exists a sequence, say $t_{i} \to \infty$,
such that $\lim_{i\to\infty}\beta(t_i) = Q \in\real^2$. But $Q\in
F(u)$ as $F(u)$ is closed, and since $F(u)$ is smooth at $Q$, for a
small enough $r>0$ $F(u)\cap \overline{B_{r}}(Q)$ is a connected arc
$\tilde{\beta}$ that contains $Q$ in its interior. But then it has
to be that $\tilde{\beta}\cap \beta \neq \emptyset$, so that by
connectedness $\tilde{\beta}\subseteq\beta$. The last statement
contradicts the finite limit of $\{\gamma(t_i)\}$. Therefore, each
connected component of $F(u)$ is a smooth curve, which is the image
of a proper embedding of $\real$ into the plane -- we shall call
such curves ``proper arcs". Furthermore, each proper arc of $F(u)$
has strictly positive curvature.

In this last step of the proof of the Proposition, we shall show
that $F(u)$ consists of precisely two proper arcs. Then we can
invoke \cite[Theorem 12]{Traizet} to conclude that $u$ is the
hairpin solution.

As we know, $u$ blows down to a wedge solution $s|x_2|$, so for a
sequence of $\delta_j\searrow 0$ we can find a sequence $R_j
\nearrow \infty$, so that $F(u)\cap B_{R_j} \subseteq
\{|x_2|<\delta_j R_j\}$ and $\{x_2=0\}\cap B_{R_j}\subseteq
\mathcal{N}_{\delta_j R_j}(F(u))$. Define
\begin{equation*}
       \alpha_{L,j} = \del B_{R_j} \cap \{x_1<0, |x_2|<\delta_j R_{j}\} \qquad \alpha_{R,j} = \del B_{R_j} \cap \{x_1>0,
    |x_2|<\delta_j R_{j}\}.
\end{equation*}
Claim that each connected arc $\gamma \in F(u)$ that intersects
$B_{R_j}$ ``enters and exits" $B_{R_{j}}$ either through $\alpha_{L,
j}$ or $\alpha_{R,j}$ if $R_j$ is large enough, i.e.
\begin{equation*}
    \del{B_{R_j}}\cap \gamma \subseteq \alpha_{L, j} \quad
    \text{or} \quad \del{B_{R_j}}\cap \gamma \subseteq \alpha_{R,
    j}.
\end{equation*}
If not, then let $U$ be the connected component of $\{u>0\}$ such
that $\gamma \subseteq \del U$. Then it's easy to see that
$u\chrc{U}$ is a viscosity solution of \eqref{FBP_0} whose free
boundary is asymptotically flat, and thus $u\chrc{U}$ has to be a
half-plane solution, which is impossible since $|\nabla u|<1$.

As a consequence of the above argument, combined with the fact that
$F(u_{R_j})\to \{x_2=0\}$ locally in the Hausdorff distance, it must
be that $F(u)$ consists of at least two proper arcs. Assume that
$F(u)$ has at least three: $\gamma_1$, $\gamma_2$ and $\gamma_3$,
and let $j_0$ be large enough such that $\del B_{R_{j_0}}\cap
\gamma_i \subseteq \alpha_{L,{j_0}}$ or $\del B_{R_{j_0}}\cap
\gamma_i \subseteq \alpha_{R,{j_0}}$, $i=1,2,3$. Note that because
$\gamma_i$ has positive curvature, if say $\del B_{R_{j_0}}\cap
\gamma_i \subseteq \alpha_{L,{j_0}}$, then $\del B_{R_{j}}\cap
\gamma_i \subseteq \alpha_{L,{j}}$ for all $j\geq j_0$ (similarly,
if $\gamma_i$ ``enters and exits from the right" at scale $R_{j_0}$,
it will do so at any larger scale). It has to be that at least two
of these arcs ``enter and exit" from the same side, say $\gamma_1$
and $\gamma_2$ ``enter and exit" from the right. Let $0<M<R_j/3$ be
large enough such that $\{x_1 = M\}$ intersects both $\gamma_1$ and
$\gamma_2$, and consider any connected component $V$ of
\begin{equation*}
    \{u > 0\} \cap \{M<x_1<M+R_{j}/3\}.
\end{equation*}
Applying Lemma \ref{lemma_noShortNonFBseg} to the piecewise $C^1$
Jordan domain $V$,
\begin{equation*}
    \hsd(\del V\cap F(u)) \leq C \hsd(\del V\setminus F(u)) \leq 4
    \delta_j R_j,
\end{equation*}
while clearly $\hsd (\del V\cap F(u)) \geq 2 R_j/3$. This leads to a
contradiction when $j\to \infty$ as $\delta_j\to 0$.

\end{proof}

%\begin{lemma}\label{lemma_TwoCurvesOneEnd}
%Let $u$ be a classical solution of \eqref{FBP_0} in $B_{2M}$ for
%some large $M$, whose Lipschitz norm is $L$. Assume that \eqref{topo_assmpt} is
%satisfied. Then if $\delta>0$ and $\lambda=\lambda(L)>0$ are small
%enough, there do not exist two distinct free boundary arcs
%$\gamma_1, \gamma_2\subseteq F(u)\cap \overline{B_M}$ with both ends
%on
%\begin{equation*}
%    \alpha_{L}=\del B_{M} \cap \{x_1<0, |x_2|<\delta M\}
%\end{equation*}
%and which contain points $p\in B_1\cap \gamma_1$, $q\in
%B_1\cap\gamma_2$ such that the line segment $\tau$ from $p$ to $q$
%belongs to $B_M^+(u)$.

The proof of the proposition is complete modulo the following three
technical lemmas.

\begin{lemma}\label{lemma_TwoPtsBlowup} Let $u$ be a classical
solution of \eqref{FBP_0} in $B_{2M}$ for some large $M$, whose
Lipschitz norm is $L$. Assume that \eqref{topo_assmpt} is satisfied
and that
\begin{equation*}
    \{u=0\}\cap
\overline{B_M}\subseteq \{|x_2|<\delta M\}
\end{equation*}
for some $\delta>0$. Assume further that $F(u)\cap \overline{B_{M}}$
consists of arcs, each having its two ends either in $\alpha_L$ or
in $\alpha_R$, where
\begin{equation*}
       \alpha_{L} = \del B_{M} \cap \{x_1<0, |x_2|<\delta M\} \qquad \alpha_{R} = \del B_{M} \cap \{x_1>0,
    |x_2|<\delta M\}.
\end{equation*}
Let $F_L$ (resp. $F_R$) denote the set of points of $F(u)\cap
\overline{B_M}$ that lie on arcs both whose ends belong to
$\alpha_L$ (resp. $\alpha_R$). Then there exist small positive
$\delta_0=\delta_0(L)$ and $\lambda=\lambda(L)<1$  such that if
$0<\delta < \delta_0$, one cannot find two straight-line open
segments $\tau_1$ and $\tau_2$ of length less that $\lambda$ in
$B_{M/2}^+(u)$, each having one end in $F_L$ and one in $F_R$, and
such that $\text{dist}(\tau_1, \tau_2)\geq 1$.
\end{lemma}
\begin{proof}
Assume that for some $\delta$, $\lambda$ such segments exist; we'll
derive a contradiction by taking $\delta$ and $\lambda$ small enough
and universal. Let $\tau_1$ have ends $p_L\in F_L$ and $p_R\in F_R$,
and $\tau_2$ connect $q_L\in F_L$ to $q_R\in F_R$. The following
three different scenarios regarding the relation between these
points may hold.

\paragraph{\textbf{Scenario 1}}

The points $p_L$, $p_R$, $q_L$, $q_R$ belong to distinct arcs in
$F_L$ and $F_R$: $\gamma_{p,L}$, $\gamma_{p,R}$, $\gamma_{q,L}$ and
$\gamma_{q,R}$, respectively. Each of these arcs is divided into two
subarcs by its respective point -- that start on the point and end
on $\del B_M$; let us choose one of these two and denote it by
$\gamma'_{[\cdot],[\cdot]}$, say our choice of a subarc on
$\gamma_{p,L}$ will be denoted by $\gamma'_{p,L}$. Then note that
$\Gamma_p:=\gamma'_{p,L}\cup\tau_1\cup\gamma'_{p,R}$ and
$\Gamma_q:=\gamma'_{q,L}\cup\tau_2\cup\gamma'_{q,R}$ are disjoint
simple curves and $B_M\setminus (\Gamma_p \cup \Gamma_q)$ consists
of three connected components, only one of which is contained in
$\{|x_2|<\delta M\}\cap B_M$; let us call it $D$. Consider a
connected component $U$ of $D^+(u)$ that has $\tau_1$ as part of its
boundary. Obviously, $U$ is piecewise $C^1$ with
\begin{equation*}
    \hsd(F(u)\cap \del U)\geq 2M\sqrt{1-\delta^2} -\hsd(\tau_1) \geq 2M\sqrt{1-\delta^2} - \lambda \quad
    \text{and} \quad \hsd(\del U\setminus F(u)) \leq
    4\arcsin(\delta)M
\end{equation*}
Applying Lemma \ref{lemma_noShortNonFBseg} to $U$, we reach the
inequality
\begin{equation*}
    2M\sqrt{1-\delta^2} - \lambda \leq 4\arcsin(\delta)ML,
\end{equation*}
which cannot be satisfied if $\delta$ and $\lambda$ are small
enough.

\paragraph{\textbf{Scenario 2}}

Two of the points that ``connect to one side" belong to the same
arc, while their counterparts belong to distinct arcs: say the
points $p_L$ and $q_L$ belong to the same arc $\gamma_L$ in $F_L$,
while $p_R$ and $q_R$ belongs to two distinct arcs $\gamma_{p,R}$
and $\gamma_{q,R}$, respectively, in $F_R$. This time let
$\gamma'_L$ denote the subarc of $\gamma_L$ whose ends are $p_L$ and
$q_L$ and let $\gamma'_{p,R}$, $\gamma'_{q,R}$ be determined in the
same fashion as in Scenario 1 above. Then $\Gamma:= \gamma'_{p,R}
\cup \tau_1\cup\gamma'_L\cup\tau_2\cup\gamma'_{q,R}$ is a simple
curve in $B_M$ with ends on $\alpha_R$, so that $B_M\setminus
\Gamma$ consists of two connected components, only one of which is
contained in $\{|x_2|<\delta M\}$; let us again call it $D$.
Consider a connected component $U$ of $D^+(u)$ that has $\tau_1$ as
part of its boundary (and thus $\tau_2$ and $\gamma'_L$). Then $U$
is piecewise $C^1$ with
\begin{equation*}
    \hsd(F(u)\cap \del U)\geq \hsd(\gamma'_L)+ \min\{\hsd(\gamma'_{p,R}),
    \hsd(\gamma_{p,R}\setminus\gamma'_{p,R})\}+\min\{\hsd(\gamma'_{q,R}), \hsd(\gamma_{q,R}\setminus\gamma'_{q,R})\}
\end{equation*}
since it is either $\gamma'_{p,R}$ or
$(\gamma_{p,R}\setminus\gamma'_{p,R})$ (and similarly,
$\gamma'_{q,R}$ or $(\gamma_{q,R}\setminus\gamma'_{q,R})$) that
belongs to $\del U$. But each of these curves intersects both $\del
B_{M/2}$ and $\del B_{M}$, so that its length is at least $M/2$. As
$\hsd(\gamma'_L)\geq \text{dist}(\tau_1,\tau_2)\geq 1$, we get
\begin{equation*}
    \hsd(F(u)\cap \del U)\geq 1+ M,
\end{equation*}
while on the other hand
\begin{equation*}
    \hsd(\del U\setminus F(u)) \leq 2\arcsin(\delta) M.
\end{equation*}
Applying Lemma \ref{lemma_noShortNonFBseg} to $U$ we see that
\begin{equation*}
    1+M \leq 2\arcsin(\delta) M L,
\end{equation*}
which is violated when $\delta$ is small enough.

\paragraph{\textbf{Scenario 3}}

In this last scenario, $p_{L}$ and $q_L$ belong to the same arc
$\gamma_L$ of $F_L$, and $p_R$ and $q_R$ belong to the same arc
$\gamma_R$ of $F_R$. Let $\gamma'_L$ denote the subarc of $\gamma_L$
with ends $p_L, q_L$ and $\gamma'_R$ denote the subarc of $\gamma_R$
with ends $p_R, q_R$. Then $\Gamma :=
\tau_1\cup\gamma'_L\cup\tau_2\cup\gamma'_R$ is a simple closed curve
that encloses a piecewise $C^1$ Jordan domain $U\subseteq B_M^+(u)$
with
\begin{equation*}
    \hsd(F(u)\cap \del U) = \hsd(\gamma'_L)+\hsd(\gamma'_R) \geq 2
    \text{dist}(\tau_1,\tau_2)\geq 2,
\end{equation*}
while
\begin{equation*}
    \hsd(\del U \setminus F(u)) = \hsd(\tau_1)+\hsd(\tau_2) \leq
    2\lambda.
\end{equation*}
Lemma \ref{lemma_noShortNonFBseg} then yields
\begin{equation*}
    2\leq 2\lambda L,
\end{equation*}
which is impossible if $\lambda$ is small enough.
\end{proof}

\begin{lemma}\label{lemma_NoExcPoint} Let $u:\real^2\to \real$ be a non-degenerate
viscosity and variational solution of \eqref{FBP_0} with $|\nabla
u|<1$ in $\{u>0\}$. Assume further that $F(u)$ is smooth everywhere
but possibly the origin $0 \in F(u)$ and that $F(u)\setminus 0$ has
no connected component that is a closed curve. Then $F(u)$ is smooth
at the origin as well, and $u$ is a classical solution of
\eqref{FBP_0} globally.
\end{lemma}

\begin{proof} Every connected component $\gamma$ of $F(u)\setminus 0$ is a smooth connected submanifold of $\real^2$, and since it is not
diffeomorphic to circle $\mathbb{S}^1$ by hypothesis, it has to be
diffeomorphic to the real line $\real$. Thus, it is the image of an
embedding $\gamma:\real \to \real^2$ and it must be that for any
sequence $t_n\to\infty$ (or $t_n\to-\infty$) $\lim_{n\to\infty}
\gamma(t_n)$ is either $\infty$ or $0$. Otherwise, there would be
some sequence $t_n\to\pm\infty$ such that $\gamma(t_n)$ converges to
some finite limit $q\in F(u)$, where $q\neq 0$, because $F(u)$ is
closed. But $F(u)$ is smooth at $q$, so that for a small enough
$r>0$, $F(u)\cap \overline{B_{r}}(q)$ is a connected arc $\beta$
that contains $q$ in its interior. Since $\beta\cap \gamma \neq
\emptyset$, it follows that $\beta\subseteq\gamma$ by connectedness,
which contradicts the convergence $\gamma(t_n) \to q$.

As a first step, we claim that $F(u)\cap \del B_1$ consists of
finitely many points. Otherwise, there would exist a sequence of
points $p_n\in F(u)\cap \del B_1$ that converges to some $p\in
F(u)\cap\del B_1$; denote by $\gamma$ the connected component of
$F(u)\setminus 0$ containing $p$. Since $F(u)$ is smooth at $p$ it
would have to be that for all large enough $n$, $p_n\in \gamma$.
However, that contradicts the fact that $\gamma$ is an analytic
curve different from the circle $\del B_1$.

Second, assume that $F(u)\setminus 0$ has a connected component
$\alpha$, an image of a smooth $\alpha:\real\to\real^2$, such that
$\lim_{t\to\pm\infty}\alpha(t)=0$. Then $\overline{\alpha}$ is a
simple closed curve, so that it encloses a bounded connected domain
$U$. Obviously $U\subseteq \{u=0\}$, as the Strong Maximimum
Principle prevents $U$ from containing points $x$ where $u(x)>0$.
Because of Corollary \ref{rem_poscurv}, $\alpha$ has positive
curvature, so we can apply Lemma \ref{lemma_TwoCurves} to conclude
that $\{u=0\}\supseteq U$ contains a non-trivial sector based at
$0$. As a result, the blow-up of $u$ at $0$ cannot be the wedge
solution and can only be the one-plane solution. Therefore $F(u)$
has to be smooth at $0$.

Thus, we may assume that for each connected component $\gamma$ of
$F(u)\setminus 0$, $\gamma(t_n)\to\infty$ for some sequence $t_n \to
\infty$ or $t_n\to -\infty$. In particular, each connected component
that intersects the unit ball $B_1$ will exit it at least once.
Thus, there are finitely many such connected components, as
$F(u)\cap \del B_1$ consists of finitely many points. Furthermore,
the very same reason implies that it is impossible to have
$\gamma(t_n)\to\infty$ for one sequence $t_n\to\infty$ (or
$-\infty$), while $\gamma(\tilde{t}_n)\to 0$ for another sequence
$\tilde{t}_n\to\infty$ (or $-\infty$). Thus, either
$\lim_{t\to\pm\infty} \gamma(t) = 0$ or $\lim_{t\to\pm \infty}
\gamma(t) =\infty$.

Next we note that $0$ cannot be an isolated point of $F(u)$, so
there exists a sequence of points $P_n \in F(u)$ such that $P_n\to
0$. Since there are only finitely many connected components of
$F(u)\setminus 0$ intersecting $B_1$, there exists a subsequence
$P_{n_k}$ that belongs to a single connected component $\gamma_1$,
so that $\lim_{t\to\infty}\gamma_1(t)= 0$. Then it must be
$\lim_{t\to -\infty}\gamma_1(t)= \infty$, so the latest ``entry
time" for $\gamma_1$ into $B_1$,
\begin{equation*}
    T:=\sup\{t:\gamma_1(t)\in (B_1)^c\}
\end{equation*}
satisfies $|T|<\infty$. Consider the connected component $V$ of
$(\{u=0\}\cap B_1)^{\circ}$ having $\gamma_1\big([T,\infty]\big)$ as
part of it boundary. Obviously, $0\in \del V$ and claim that there
exists another curve $\gamma_2:(-\infty,0]\to\real^2$ such that
$\gamma_2\big((-\infty,0]\big) \subset\del V\cap (F(u)\setminus 0)$
with $\lim_{t\to-\infty}\gamma_2(t)=0$. If not, $\Big(\del
V\setminus \gamma_1\big([T,\infty]\Big) \cap F(u)$ consists of
finitely many free boundary arcs with ends $p_{2k}$ and $p_{2k+1}$
on the unit circle $\del B_1$, $k=0,1, 2,\ldots,l$. Here we have
chosen the enumeration of the points $\{p_i\}$ in such a way that
the shorter circular arc $\widehat{p_1p_2}\subseteq \del V$, and
that has $p_{i+1}$ following $p_i$ in the direction (clockwise or
counterclockwise) set by $p_1$ and $p_2$. In this way, the circular
arcs $\widehat{p_{2k+1}p_{2k+2}}\subseteq \del V$,
$k=0,1,\ldots,l-1$. Let $q\in F(u)\cap \del B_1$ be the next point
after $p_{2l+1}$ on the unit circle as we traverse it in the same
direction. Then it must be that $\widehat{p_{2l+1} q}\subseteq \del
V$ and that the connected component of $(F(u)\setminus 0)\cap
\overline{B_1}$, having $q$ as one of its ends, is also part of
$\del V$. Since the other end of that component can neither lie on
$\del B_1$ nor be $0$, it has to be that $q=p_0$. This is, however,
impossible as $u$ cannot be zero on both sides of $\gamma_1$. So,
there exists a free boundary curve $\gamma_2\subset \del V\cap
(F(u)\setminus 0)$, disjoint from $\gamma_1$, with
$\gamma_2(-\infty)=0=\gamma_1(\infty)$. From here it is not hard to
see that $V$ is a Jordan domain. Again, Corollary \ref{rem_poscurv}
says that $\gamma_1$ and $\gamma_2$ have positive curvature, and we
can invoke Lemma \ref{lemma_TwoCurves} to establish that $V$
contains a non-trivial sector based at $0$. As before, the blow-up
limit of $u$ at zero has to be the half-plane solution, so that
$F(u)$ is smooth.
\end{proof}

\begin{lemma}\label{lemma_TwoCurves} Let $U\subseteq \real^2$ be a
Jordan domain with $0\in \del U$ and let $\gamma_1\in
C^2([0,\infty),\real^2)$ and $\gamma_2\in C^2((-\infty,0], \real^2)$
be some regular parameterizations ($\gamma_i'\neq 0$, $i=1,2$) of
two simple disjoint subarcs of $\del U$, for which
$\lim_{t\to\infty} \gamma_1(t)=\lim_{t\to-\infty}\gamma_2(t)=0$, and
such that traversing $\del U$ in the counterclockwise direction
corresponds to $t$ increasing. Assume further that their curvatures
are strictly positive. Then $B_r\cap U$ contains a non-trivial
sector of $B_r$.
\end{lemma}
\begin{proof}
First let us introduce some notation. For a point $p$ in $\gamma_i$,
$i=1,2$,  let $L(p)$ be the tangent line to $\gamma_i$ at $p$. If
$p=\gamma_i(t_0)$ for some $t_0$, let
$\tau(p)=\gamma_i'(t_0)/|\gamma_i'(t_0)|$ be the unit tangent vector
to $p$ in the direction of $\gamma_i'(t_0)$; let $\nu(p)$ be the
unit normal vector to $\gamma_i$ at $p$ that one gets by rotating
$\tau(p)$ by $\pi/2$. Denote by $H^+(p)$ and $H^-(p)$ the two
half-planes:
\begin{equation*}
    H^{\pm}(p) = \{x:\in \real^2: (x-p)\cdot (\pm \nu(p)) > 0\}.
\end{equation*}
For any two points $p=\gamma_i(t_1)$ $q=\gamma_i(t_2)$, $t_1<t_2$,
define $\theta(p,q)$ to be the angle $\gamma_i'(t)$ sweeps as $t$
increases from $t_1$ to $t_2$. Then the fact that $\gamma_i$ has
positive curvature is equivalent to $\theta(\gamma_i(t),
\gamma_i(t+s))$ being a positive, strictly increasing function of
$s$ for $s>0$ and any fixed $t$. For any $p\in \gamma_i(t)$ and
$\alpha>0$, denote by $T^{\alpha}(p)$ the point $q\in\gamma_i$ such
that $\theta(p,q)=\alpha$, if it exists. Let $s(p, q)$ be the open
segment of $\gamma_i$ with ends $p, q\in \gamma_i$.

Now we proceed with the argument. We shall show that
$\theta(\gamma_1(0), \gamma_1(t))$ is bounded from above as a
function of $t$. Assume not; then $T^{\alpha}(p)$ exists for any
$p\in \gamma_1$ and any $\alpha>0$. Claim that there exists a
$p_0\in \gamma_1$ such that $T^{2\pi}(p_0)\in \overline{H^+(p_0)}$.
If not, then for any $p\in \gamma_1$ and $k\in \nat$,
\begin{align*}
    & T^{(2k+2)\pi}(p) \in H^-(T^{2k\pi}(p)) \Subset H^{-}(p) \quad
    \text{as well as} \\
    & T^{(2k+3)\pi}(p)\in H^-(T^{(2k+1)\pi}(p))\Subset
H^{-}(T^{\pi}(p)).
\end{align*}
However, note that because $\gamma_1$ has positive curvature, we
have $s(p,T^{\pi}(p))\subseteq H^{+}(p)$, as the smallest $\alpha>0$
for which $s(p, T^{\alpha}(p))$ can intersect $L(p)$ must be greater
than $\pi$. Thus, $H^{-}(T^{\pi}(p))\Subset H^{+}(p)$. But since
$\gamma_1(t)\to 0$ as $t\to\infty$ and $\gamma_1'\neq 0$, then both
\begin{equation*}
    T^{(2k+2)\pi}(p) \to 0 \quad \text{and} \quad
    T^{(2k+3)\pi}(p)\to 0 \quad \text{as } k\to\infty.
\end{equation*}
That contradicts the fact that $T^{(2k+2)\pi}(p) \in H^{-}(p)$
whereas $T^{(2k+3)\pi}(p) \in H^{+}(p)$.

Thus, for some $p_0$, $T^{2\pi}(p_0)\in \overline{H^+(p_0)}$, so
that the whole segment $s(p_0,T^{2\pi}(p_0))\subseteq H^{+}(p_0)$.
Denote
\begin{equation*}
    p_j := T^{j\pi}(p_0)   \quad j\in\nat.
\end{equation*}
Claim we then have $s(p_{2k-2}, p_{2k}) \subseteq H^+(p_{2k-2})$ for
all $k\in\nat$. Argue by induction. Note that since $s(p_{2k-1},
p_{2k})\subseteq H^{+}(p_{2k-1})\cap H^{+}(p_{2k})$, there are
exactly two intersection points between $\overline{s}(p_{2k-2},
p_{2k})$ and $L(p_{2k})$, namely $p_{2k}$ and a point $q_{2k}\in
\overline{s}(p_{2k-2}, p_{2k-1})$ (see Figure
1). If it were the case that $p_{2k+2}\in
H^{-}(p_{2k})$, the segment $s(p_{2k},p_{2k+2})$ would have to leave
the convex domain $D_{2k}\subseteq H^{+}(p_{2k})\cap
H^{+}(p_{2k-1})$, enclosed by $s(q_{2k},p_{2k})$ and the
straight-line segment $\sigma_{2k}:=p_{2k}q_{2k}$. But obviously
$s(p_{2k}, p_{2k+1})\subseteq D_{2k}$, so it would have to be
$s(p_{2k+1},p_{2k+2})$ that exits $D_{2k}$. Construct as above the
point $q_{2k+1}\in s(p_{2k-1},p_{2k})$ being the second intersection
point of $L(p_{2k+1})$ with $\overline{s}(p_{2k-1},p_{2k+1})$ and
let $\sigma_{2k+1}\subset D_{2k}$ be the straight-line segment
$p_{2k+1}q_{2k+1}$. Then the convex domain $D_{2k+1}$, enclosed by
$\sigma_{2k+1}$ and $s(q_{2k+1},p_{2k+1})$, is contained within the
convex $D_{2k}$, so that $s(p_{2k+1},p_{2k+2})$ which enters
$D_{2k+1}$ would have to exit $D_{2k+1}$ before it exits $D_{2k}$.
That is, however, impossible as $s(p_{2k+1},p_{2k+2})\subseteq
D_{2k+1}$. The induction step is complete.

\begin{figure}\label{figure_ConvCurves}
\centering 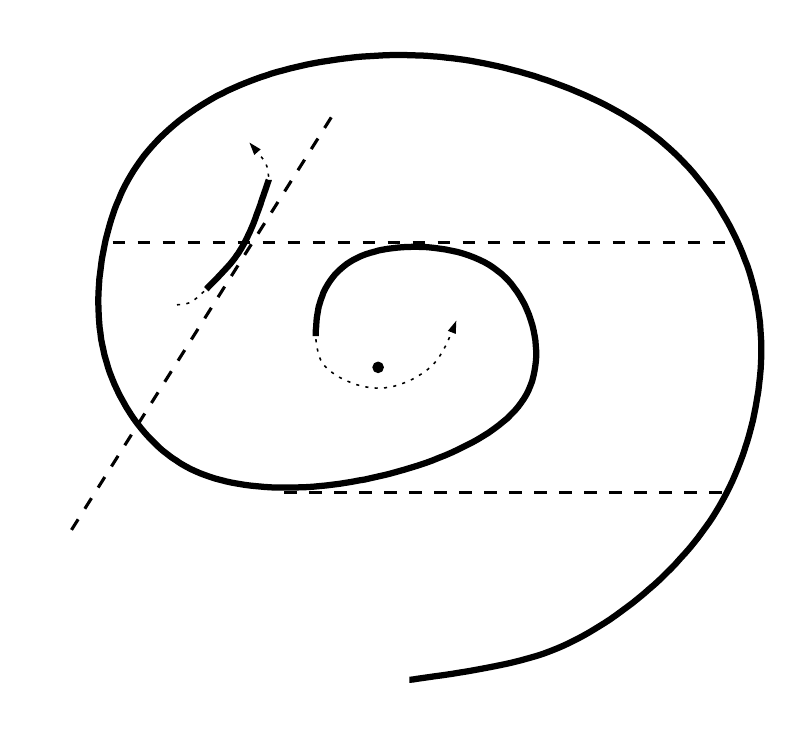 \caption{The curves $\gamma_1$
and $\gamma_2.$}
\end{figure}

Let now $K\in\nat$ be large enough such that for all $k\geq K$,
$s(p_{2k-2}, p_{2k})\subseteq B_{\delta}(0)$ where $\delta >0$ is
small enough, such that $\gamma_2(0)\in B_{2\delta}(0)^c$ (such a
$K$ exists since $\gamma_1(t)\to 0$ as $t\to\infty$). Since
$\gamma_2(t)\in D_{2k+1}$ for all large $|t|$, its last `time of
exit' from $D_{2k+1}$
\begin{equation*}
    T:=\sup\{t: \gamma_2(t)\in D_{2k+1}\}
\end{equation*}
exists and must satisfy $T<0$. Obviously, $\gamma_2(T)$ must belong
to $\sigma_{2k+1}$ and $\gamma_2$ must intersect $\sigma_{2k+1}$
transversally, for otherwise the fact that $\gamma_2$ has positive
curvature would imply that for some small $\eps<0$ $\gamma([T-\eps,
T+\eps])$ would lie on one side of $\sigma_{2k+1}$ which contradicts
the definition of $T$. Let $\tilde{q}_{2k+1}$ denote the
intersection point of $L(p_{2k+1})$ and $s(p_{2k-2},p_{2k-1})$, and
let $\tilde{\sigma}_{2k+1}$ be the straight-line segment
$p_{2k+1}\tilde{q}_{2k+1}$. Note that since $\gamma_2$ intersects
$\sigma_{2k+1}$ transversally, $\tilde{\sigma}_{2k+1}\subseteq
H^{-}(\gamma_2(T))$. Also, since $\gamma_2(0)\notin B_{\delta}(0)$,
$\gamma_2([T,0])$ must exit the domain $\tilde{D}_{2k+1}\subseteq
B_{\delta}(0)$, enclosed by $s(\tilde{q}_{2k+1},q_{2k+1})$ and the
straightline segment $q_{2k+1}\tilde{q}_{2k+1}$, having once entered
it. Thus $\gamma_2([T,0])$ intersects
$\tilde{\sigma}_{2k+1}\subseteq H^{-}(\gamma_2(T)))$, so that
$\gamma_2((T,0])$ must cross $L(\gamma_2(T))$. Let $T_1$ be the
first time $\gamma((T,0])$ crosses $L(\gamma_2(T))$:
\begin{equation*}
    T_1:=\inf\{t>T, \gamma_2(t)\in L(\gamma_2(T))\}.
\end{equation*}
Note that $T_1>T$ as $\gamma_2((T,T+\delta))\subseteq
H^{+}(\gamma_2(T))$) for all small enough $\delta>0$. Furthermore,
it must be that $\theta(\gamma_{2}(T),\gamma_2(T_1))\geq \pi$ but
$\theta(\gamma_{2}(T),\gamma_2(T_1))\leq 2\pi$. The former bound is
obvious; the latter is true for otherwise $T^{2\pi}(\gamma_2(T))\in
H^{+}(\gamma_2(T))$, so that by the same argument as before we would
have all of $\gamma_2((T,0])\subseteq H^{+}(\gamma_2(T))$, which
would prevent it from crossing $L(\gamma_2(T))$. As a result, it
must be that
\begin{equation*}
    (\gamma_2(T_1)-\gamma_2(T)\big)\cdot \gamma_2'(T)<0,
\end{equation*}
which in turn implies that $\gamma_2((T,T_1))$ must cross the
straighline segment $\gamma_2(T)q_{2k+1}$, which is impossible.

Therefore, $\theta(\gamma_1(0),\gamma(t))$ is bounded from above, so
that $\tau_1:=\lim_{t\to\infty}\gamma_1'(t)/|\gamma_1'(t)|$ exists.
Exchanging the roles of $\gamma_1$ and $\gamma_2$ in the argument
above, we can show that
$\tau_2:=\lim_{t\to-\infty}\gamma_2'(t)/|\gamma_2'(t)|$ exists, as
well. As a result, for all small enough $r>0$, $\gamma_1\cap B_r$
and $\gamma_2\cap B_r$ are flat graphs over the radii along $\tau_1$
and $\tau_2$, respectively. Let $A_i=\del B_r \cap \gamma_i$,
$i=1,2$ be the points of intersection of $\del B_r$ with $\gamma_1$
and $\gamma_2$. Because of the positivity of the curvature, the open
straight-line segments connecting $0$ to $A_1$ and $0$ to $A_2$ are
contained in $U$ for $r$ small enough. Also, since $\del B_r \cap
\del U = \{A_1, A_2\}$, the whole open circular arc
$\widehat{A_2A_1}$ (as we trace the circle in the counter-clockwise
direction) must be contained in the Jordan domain $U$. Thus, $U$
contains the entire open circular sector with vertex $0$ and arc
$\widehat{A_2A_1}$.

%\begin{align*}
%    \gamma_1\cap B_r(0) & = \{s \tau_1 + \phi_1(s)\nu_1: -r_1< s<
%    0\} \\
%     \gamma_2\cap B_r(0) & = \{s \tau_2 + \phi_2(s)\nu_2: 0< s<
%     r_2\},
%\end{align*}
%for some $0<r_1, r_2<r$ and where $\nu_1$ and $\nu_2$ are the
%rotated $\tau_1$ and $\tau_2$ by $\pi/2$.

\end{proof}

%%%%%%%%%%%%%%%%%%%%%%%%%%%%%%%%%%%%%%%%%%%%%%%%%%%%%%%%%%%%%%%%

\section{Local structure.} \label{Sec:LocStruct}

In this section we shall study the shape of the free boundary of
solutions of \eqref{FBP_0}, defined in the unit disk, satisfying the
topological assumption \eqref{topo_assmpt}. This will be carried out
by examining blow-up limits of sequences of solutions in $B_1$, for
which exact purpose the classification Theorem \ref{theorem_Main}
was developed. We encounter the following dichotomy: if a component
of the zero phase is well separated by the rest of the zero phase,
its boundary has bounded curvature (in terms of the separation) --
this is the content of Proposition \ref{prop:bounded curvature}
below. Once the separation becomes small enough relative a certain
universal scale, we shall see the signs of a hairpin-like structure
arising -- this is described in Propositions
\ref{prop:smalldistImpliesHPin} and \ref{prop:two_strands}.

Let us make the following definition for ease of reference.
\begin{definition} We shall call the free boundary $F(u)$ of a solution $u$ of \eqref{FBP_0}
$\d$-flat in $B=B_r(p)$ if for some rotation $\rho$,
\begin{equation*}
    |u(p+\rho x)- x_2^+| \leq \d r \quad \text{for} \quad x\in B_r(0).
\end{equation*}
\end{definition}

\begin{remark}\label{remark:deltaFlatness}
Denote by $\d_0$ the small universal constant, such that if
$0<\delta<\d_0$ small enough, the Alt-Caffarelli regularity theory
\cite{AC} states that $F(u)\cap B_{r/2}$ is a graph in the direction
of $\rho(e_2)$ with Lipschitz norm at most $C\d$. For such $\d$ we
also have the bound
\begin{equation*}
\|\rho \nabla u - e_2\|_{L^{\infty}(B_{r/2}^+(u))} + r\|D^2
u\|_{L^{\infty}(B_{r/2}^+(u))} \leq c\d.
\end{equation*}
It implies, in particular, that the curvature of $F(u)$ in
$B_{r_0/2}$ is $O(\d)$.
\end{remark}

The next proposition treats the scenario where a point of the free
boundary $F(u)$ is distance at least $s$ away from all other
components of the zero phase; then we expect a curvature bound on
$F(u)$ at the point.

\begin{prop}\label{prop:bounded curvature} Let $u$ be a classical
solution of \eqref{FBP_0} in $B_1$ that satisfies
\eqref{topo_assmpt} and assume $0\in F(u)$. Denote by $Z$ the
connected component of $0$ in $\{u=0\}$. For any $0<s<1$ there
exists $\k= \k(s)<\infty$ such that if \[d(0,\{u=0\}\setminus Z)
\geq s\] then the curvature of $F(u)$ at $0$ is at most $\k$.
\end{prop}
\begin{proof}
Assume the proposition is false. Then we have a sequence of
counterexamples $u_l$ for which the curvature $\kappa_l$ of $F(u_l)$
at zero is
\begin{equation*}
    \kappa_l \geq l^2.
\end{equation*}
Define the rescales
\begin{equation*}
    \tilde{u}_l(x):=l u_l(x/l) \quad \text{for} \quad x\in B_{l}.
\end{equation*}
Then the curvature $\tilde{k}_l$ of $F(\tilde{u}_l)$ at $0$
satisfies
\begin{equation}\label{prop:bounded curvature:eqCurv}
    \tilde{\k}_l = \k_l/l \geq l.
\end{equation}
By our classification Theorem \ref{theorem_Main} we see that, up to
taking a subsequence, the $\tilde{u}_l$ converge uniformly on
compact subsets to a global solution $\tilde{u}$ that is either a
one-plane, a two-plane, a hairpin or a wedge solution.

Let $\delta_0>0$ be the small universal constant defined in Remark
\ref{remark:deltaFlatness}.
%small enough so that if the free boundary of a solution of
%$\eqref{FBP_0}$ is $\delta_0$-flat in $B_1$, then its free boundary
%in $B_{1/2}$ is a graph of a function satisfying a $C^{2}$-bound and
%hence a curvature bound $C(\delta_0)$ at $0$.
If $\tilde{u}$ is a one-plane solution, then for all large enough
$l$, in some Euclidean coordinates
\begin{equation*}
    |\tilde{u}_l-x_2^+| < \delta_0/2 \quad \text{in} \quad B_1,
\end{equation*}
hence $F(\tilde{u}_l\cap B_1)$ is $\d_0$-flat and $\tilde{k}_l\leq
C\delta_0$ which contradicts \eqref{prop:bounded curvature:eqCurv}.
Similarly, if $\tilde{u}$ is a two-plane solution, for some $b<0$
and all large enough $l$
\begin{equation*}
    |\tilde{u}_l - (x_2^+ + (x_2-b)^-)| < \min\{\delta_0/2, b/10\} \quad \text{in} \quad
    B_1.
\end{equation*}
Thus,
\begin{equation*}
    v_l:=\tilde{u}_l \chrc{B_1\cap \{x_2>b/2\}}
\end{equation*}
is a classical solution of \eqref{FBP_0} in $B_1$, whose free
boundary is $\d_0$-flat in $B_1$, so $\tilde{\k}_l\leq C\d_0$ -- a
contradiction.

Analogously, we can rule out $\tilde{u}$ being a hairpin solution.
Assume that it is; then we can find a scale $s_0$ such that for
every $p\in F(\tilde{u})$
\[d_H( F(\tilde{u})\cap B_{s_0}(p), L(x)\cap B_{s_0}(p)) <\delta_0 s_0/2,\]
where $L(p)$ denotes the tangent line to $F(\tilde{u})$ at $p$. Now
for all large enough $l$,
    \[d_H( F(\tilde{u}_l)\cap B_{s_0}, L(0)\cap B_{s_0}) < \delta_0 s_0 \]
so that $w_l(y):=\tilde{u}_l(s_0y)/s_0$ has a $\d_0$-flat free
boundary in $B_1$ and the curvature of $F(w_l)$ at $0$ is bounded by
$C\d_0$. Thus, the curvature of $F(\tilde{u}_l)$ at $0$
\begin{equation*}
    \kappa_l \leq C\d_0/s_0,
\end{equation*}
which again contradicts \eqref{prop:bounded curvature:eqCurv}.

Finally, assume that $\tilde{u}=|x_2|$ is the wedge-solution. Then
for all $l$ large enough
\begin{equation}\label{prop:bounded curvature:eqWedgeProx}
    d_H(F(\tilde{u}_l)\cap \overline{B_4}, \{x_2=0\}\cap \overline{B_4})\leq \d_0 .
\end{equation}
Let $N=(0,1)$ and $S=(0,-1)$. Note that $N$ and $S$ cannot belong to
two separate components of $B_2^+(\tilde{u}_l)$, for according to
Lemma \ref{lemma_fbgraphs}, $ F(\tilde{u}) \cap \{|x_1|<1/2\} \cap
B_4$ consists of two graphs of Lipschitz norm at most $c\d_0$, so
that we again get an upper bound for $\tilde{k}_l$ for all large
$l$. This means that if $F(\tilde{u}_l)\cap \overline{B_3}$ consists
of finitely many arcs, each of which ``attaches" either to
$\alpha_L$ or $\alpha_R$, where
\begin{equation*}
    \alpha_L = \del B_{3} \cap \{x_1<0, |x_2|<\d_0 \} \qquad \alpha_{R} = \del B_{3} \cap \{x_1>0,
    |x_2|<\d_0\}.
\end{equation*}
Thus, if $F_L$ ($F_R$) denotes the union of the arcs of
$F(\tilde{u}_l)\cap \overline{B_3}$ that attach to $\alpha_L$
($\alpha_R$), then $F_L$ and $F_R$ are disjoint compact sets and so
$d(F_L, F_R)$ is realized for some $p\in F_L$ and $q\in F_R$.
Moreover, the straight line (open) segment $\tau$ with ends $p$ and
$q$ is contained in $B_3^+(\tilde{u}_l)$ and because of
\eqref{prop:bounded curvature:eqWedgeProx}, we have
\[|p-q|=\hsd(\tau)\le6\d_0.\] On the other hand, note that if $\tilde{Z}_l$
denotes the connected component of $0$ in $\{\tilde{u}_l=0\}$ in
$B_l$, we have by assumption
\[
d(0, \{\tilde{u}_l=0\}\setminus \tilde{Z}_l)\geq l s \gg 1,
\]
hence it must be that both $p$ and $q$ belong to the same boundary
arc of $\del \tilde{Z}_l$ (they cannot belong to different boundary
arcs of $\del \tilde{Z}_l$, for $p$ and $q$ would have to lie on the
boundary of two different connected components of
$B_l^+(\tilde{u})$). Let $\beta\subseteq F(\tilde{u}_l)$ denote the
arc connecting $p$ to $q$. Then $\beta\cap \tau$ encloses a
piecewise-$C^2$ Jordan domain $V\subseteq B_l^+(\tilde{u}_l)$ and
applying Lemma \ref{lemma_noShortNonFBseg} to $V$, we find that
\begin{equation*}
6\d_0 \ge \hsd(\tau) \geq c \hsd(\beta)
\end{equation*}
which is impossible for small $\d_0$ as $\hsd(\beta)\geq 2$. This
completes the proof.
\end{proof}

\begin{prop}\label{prop:smalldistImpliesHPin}
Let $u$ be a classical solution of \eqref{FBP_0} in $B_1$ that
satisfies \eqref{topo_assmpt} and assume $0\in F(u)$. Denote by $Z$
the connected component of $0$ in $\{u=0\}$. Then for any
$0<\delta<\delta_0$ small enough there exist $0<\eps_0\ll 1$ and
$r_0>0$ such that if for any one $0<r\leq r_0$
\begin{equation*}
    d(0,\{u=0\}\setminus Z) < \eps_0 r
\end{equation*}
then for some rotation $\rho$,
\begin{equation*}
    |u(\rho x) - |x_2|| < \delta r \quad \text{in } B_r.
\end{equation*}
\end{prop}
\begin{proof}
Fix $0<\delta <\delta_0$. By the scale-invariance of the problem it
suffices to show the conclusion of the proposition holds only for
$r=r_0$. Assume not; then for any sequences of $\eps_k \to 0$, $r_k
\to 0$, there exists a corresponding sequence of counterexamples
$u_k$ in $B_1$: namely, if $Z_k$ denotes the component of $0$ in
$\{u_k=0\}$, we have $d(0, \{u_k=0\}\setminus Z_k)\leq \eps_k r_k$,
but
\begin{equation}\label{prop:smalldistImpliesHPin:eqConPos}
    \|u_k(\rho x) - |x_2|\|_{L^{\infty}(B_{r_k})} > \delta r_k
\end{equation}
for all rotations $\rho$. Define the rescaled
\[
\tilde{u}_k(x):= u_k(r_k x)/r_k \quad\text{in } B_{1/r_k}.
\]
According to the Classification Theorem \ref{theorem_Main}, up to
taking a subsequence, $\tilde{u}_k$ converge uniformly on compact
subsets of $\real^2$ to $\tilde{u}$, being either the half-plane,
the wedge, a two-plane or a hairpin solution. %Let $\delta_0>0$ be
%small enough, such that if the free boundary of a solution of
%$\eqref{FBP_0}$ is $\delta_0$-flat  in $B_1$, then its free boundary
%in $B_{1/2}$ is a graph of a function with Lipschitz constant
%$C\delta_0$ for some numerical constant $C$.

If $\tilde{u} =x_2^+$ in an appropriate coordinate system, then for
all large enough $k$
\begin{equation*}
 \{\tilde{u}_k=0\}\cap B_{1}\cap \{|x_1|<1/2\} = \{x\in B_{1}: |x_1|<1/2, x_2<\phi(x_1)\}
\end{equation*}
for some $C\delta_0$--Lipschitz function $\phi:(-1/2,1/2)\to \real$.
In particular $\{\tilde{u}_k=0\}\cap B_{1/2}$ consists of a single
component (the one containing $0$). Hence, going back to the
original scale,
\[
    d(0, \{u_k=0\}\setminus Z_k)\geq r_k/2 > \eps_k r_k
\]
for $k$ large enough, which contradicts our assumption. Similarly,
we rule out the case when $\tilde{u}$ is the two-plane solution. If
$\tilde{u}$ is a hairpin, we can find a scale $s_0$, such that for
every $x\in F(\tilde{u})$
\[d_H( F(\tilde{u})\cap B_{s_0}(x), L(x)\cap B_{s_0}(x)) <\delta_0 s_0/2,\]
 where $L(x)$ denotes the tangent line through $x$ to the hairpin $F(\tilde{u})$. Then, for all large enough $k$,
\[
d_H( F(\tilde{u}_k)\cap B_{s_0}, L(0)\cap B_{s_0}) \leq d_H( F(\tilde{u})\cap B_{s_0}, L(0)\cap B_{s_0}) +d_H( F(\tilde{u}_k)\cap B_{s_0}, F(\tilde{u})\cap B_{s_0}) \leq s_0 \delta_0/2 + s_0 \delta_0/2 = s_0\delta_0,
\]
so that in $B_{s_0/2}$, $\{\tilde{u}_k=0\}\cap B_{s_0/2}$ consists
of a single component. Going back to scale $r_k$, we see that
\[
    d(0, \{u_k=0\}\setminus Z_k)\geq s_0 r_k/2 > \eps_k r_k
\]
which is a contradiction when $k$ is large enough.

Therefore, $\tilde{u}$ must be the wedge solution: $\tilde{u}=|x_2|$
in an appropriately rotated coordinate system. This leads to a
contradiction with \eqref{prop:smalldistImpliesHPin:eqConPos},
however, because it implies that for all $k$ large enough, we
actually have
\[
    \|u_k(x) - |x_2|\|_{L^{\infty}(B_{r_k})} \leq \delta r_k.
\]
\end{proof}

\begin{prop}\label{prop:two_strands}
For any given $0<\delta<\d_0$, let $\eps_0$, $r_0$ and $u:B_1\to
\real$ be as in Proposition \ref{prop:smalldistImpliesHPin}. Let $Z$
denote the component of $0$ in $\{u=0\}$. Then for any $0<r\leq r_0$
such that \[ d(0,\{u=0\}\setminus Z) < \eps_0 r,\] the free boundary
$F(u)\cap \overline{B_{r/2}}$ consists of exactly two arcs
$F_L\subseteq Z$ and $F_R\subseteq \{u=0\}\setminus Z$. Those are
contained in $\rho(S_{r/2,\delta r})$ for an appropriate rotation
$\rho = \rho_{r}$, where
\begin{equation*}
    S_{r,t}:= \{x\in\real^2: |x_1|\leq r, |x_2| \leq t\},
\end{equation*}
with the two ends of $F_L$ in $\rho(\alpha_{L,r/2})$ and the two
ends of $F_R$ in $\rho(\alpha_{R,r/2})$, where
\[
    \alpha_{L,r} = \{x\in \del B_{r}: x_1<0, |x_2|<\delta r \} \quad \text{and} \quad \alpha_{R,r} = \{x\in \del B_{r}: x_1>0, |x_2|<\delta r \}.
\]
Moreover, the minimum distance between the corresponding two
components of $\{u=0\}\cap \overline{B_{r/2}}$ is realized for some
points $p\in F_L$, $q\in F_R$ with both $p,q\in \rho(S_{r/3,\delta
r})$.
\end{prop}
\begin{proof}
Fix $r$ and choose Euclidean coordinates appropriately so that
$\rho_{r}$ is the identity. Let $\gamma$ be the arc of $F(u)\cap
B_r$, containing $0$. Claim that the two ends of $\gamma$ both
belong to either $\alpha_{L,r}$ or $\alpha_{R,r}$. Assume not. Then
the points $N=(0,r/2)$ and $S=(0,r/2)$ belong to two distinct
connected components of $B_r^+(u)$, so that according to Lemma
\ref{lemma_fbgraphs}, $F(u)\cap B_r\cap \{|x_1|<r/4\}$ consists of
two disjoint graphs $\Sigma_{\pm} = \{x_2 = \phi_{\pm}(x_1):
|x_1|<r/4\}$ of Lipschitz norm at most $C\delta$ and
\begin{equation*}
    u(x) = 0 \quad \text{for}\quad  x\in \{\phi_-(x_1)\leq x_2\leq \phi_+(x_1): |x_1|<r/4\}
\end{equation*}
But then $d(0, \{u=0\}\setminus Z) \geq r/4 > \eps_0 r$, which
contradicts our hypothesis. Hence, we may assume that $\gamma$
attaches on $\alpha_{L,r}$.

Look now at the free boundary in $B_{r/3}(P)$, where $P=(-r/2,0)$.
Since $\gamma\subseteq S_{r,\delta r}$ connects $\alpha_L$ to $0$,
it must be that $\gamma$ disconnects $\tilde{N}_L:=P+(0,r/6)$ from
$\tilde{S}_L:=P+(0,-r/6)$ in $B_{r/2}(p)^+(u)$. Invoking Lemma
\ref{lemma_fbgraphs} again, we see that $F(u)\cap B_{r/3}(P)\cap
\{|x_1+r/2|<r/6\}$ consists of two graphs of Lipschitz norm at most
$C\delta$. As a result, the connected component $\tilde{Z}_L$ of $0$
in $\{u=0\}\cap B_{r/2}$ is bounded by a single free boundary arc
$F_L$  and a circular subarc of $\a_{L,r/2}$ that share ends.
Another even more significant consequence is that $F(u)\cap B_r$
contains no other arcs besides $\gamma$ that intersect $V_L:=B_r\cap
\{|x_1+r/2|<r/6\}$. Since $\{u=0\}$ has a component different from
$Z$ that is at most $\eps_0 r$ away from $0$, $F(u)\cap B_r$
contains at least one more arc $\tilde{\gamma}\neq \gamma$.
According to the observation above, $\tilde{\gamma}$ doesn't cross
into the region $V$, so it has to attach on $\alpha_{R,r}$. Consider
$F(u) \cap B_{r/3}(Q)$, where $Q=(r/2,0)$. Since $\tilde{\gamma}\cap
B_{\eps_0 r} \neq \emptyset$,  it must be that $\tilde{\gamma}$
disconnects $\tilde{N}_R=Q+(0,r/6)$ from $\tilde{S}_R = Q+(0,-r/6)$
in $B_{r/3}(Q)^+(u)$. Thus, by Lemma \ref{lemma_fbgraphs}, $F(u)\cap
B_{r/3}(Q)\cap \{|x_1-r/2|<r/6\}$ consists of two graphs of
Lipschitz norm at most $C\delta$. Hence, $\{u=0\}\cap B_{r/2}$ has
only one other connected component $\tilde{Z}_R$ and $\del Z_R \cap
F(u)$ consists of a single free boundary arc $F_R$. As $F_R$ cannot
intersect $V_L$ and, similarly, $F_L$ cannot intersect $V_R:=B_r\cap
\{|x_1-r/2|<r/6\} $, it must be that the minimum distance between
$\tilde{Z}_L$ and $\tilde{Z}_R$ is realized for some points $p\in
F_L$ and $q\in F_R$ with $|x_1(p)|< r/2-r/6 = r/3$ and $|x_1(q)|<
r/3$.
\end{proof}

\begin{remark}\label{remark: recentering} Assume we are in the situation of Propositions \ref{prop:smalldistImpliesHPin} and \ref{prop:two_strands} above for some fixed small $0<\delta<\delta_0$.
Then $F(u)\cap \overline{B_{r_0/2}}$ consists of two arcs $F_L$ and
$F_R$, and the minimum distance $s=d(F_L, F_R)$ is realized for some
points $p\in F_L$, $q\in F_R$ with both $p, q \in \rho_{r_0}
\big(S_{r_0/3,\delta r_0}\big)$. Now apply again Propositions
\ref{prop:smalldistImpliesHPin} and \ref{prop:two_strands} to the
translate
\begin{equation*}
    \tilde{u}(y):=u(p+y) \quad y\in B_{1/2}.
\end{equation*}
Call $\tilde{Z}$ the connected component of $\{\tilde{u}=0\}$
containing $0$. We establish that for every $r$ such that $s/\eps_0
< r \leq r_0$, there is a rotation $\tilde{\rho}=\tilde{\rho}_{r}$
such that
\[
    |\tilde{u}(\tilde{\rho} y) - |y_2|| < \delta r \quad \text{in } B_r
\]
and the free boundary in $B_{r/2}$, $F(\tilde{u})\cap
\overline{B_{r/2}}\subseteq \tilde{\rho} (S_{r/2,\delta r})$
consists of two arcs $\tilde{F}_L\subseteq \tilde{Z}$ and
$\tilde{F}_R\subseteq \{\tilde{u}=0\}\setminus Z$, the minimum
distance between which is realized for $0\in \tilde{F}_L$ and
$q-p\in \tilde{F}_R$.
\end{remark}

\section{Lipschitz bound of free boundary
strands.}\label{Sec:lipboundFBstrands} In this section we shall
further elaborate on the finer-scale structure of the free boundary
of a solution that falls under the scenario of Proposition
\ref{prop:smalldistImpliesHPin}. More specifically, we shall show
that if the separation $s$ between two components of the zero phase
becomes small enough, it forces the free boundary outside that scale
to be the union of four graphs of small Lipschitz constant over a
common line.

\begin{theorem}\label{thm:fourgraphs} For any given small
$0<\d<\d_0$, there exist $r_0>0$, $\eps_0>0$ such that if $u$ is a
classical solution of \eqref{FBP_0} in $B_1$, satisfying \eqref{topo_assmpt}, with $0\in F(u)$ and
\begin{equation*}
    \text{dist}(0, \{u=0\}\setminus Z) < \eps_0 r_0,
\end{equation*}
then for some $p\in B_{r_0/3}$, $B_{r_0/2}(p)\cap F(u)$ consists of
two free boundary arcs $F_L$ and $F_R$, the shortest segment between
which is centered at $p$, the separation
\begin{equation*}
    s:=dist(F_L, F_R)<\eps_0 r_0.
\end{equation*}
and for some rotation $\rho$ and functions  $f, g:\real\to\real$ with $f<g$ 
\begin{align*}
    \{u=0\} \cap \big(B_{r_0/2}(p)\setminus B_{4s/\eps_0}(p)\big)
     = p & + \rho \{4s/\eps_0<|x|<r_0/2: f(x_1)\leq |x_2|\leq
    g(x_1)\} \\
\text{where} \quad \|f\|_{L^\infty} + \|g\|_{L^\infty}  & \le \delta r, 
\quad
\|f'\|_{L^\infty} + \|g'\|_{L^\infty}  \le \delta.
\end{align*}
That is, $F(u)\cap \big(B_{r_0/2}(p)\setminus B_{4s/\eps_0}(p)\big)$ consists of four graphs over a common line with Lipschitz norm at most $\d$.
\end{theorem}

The proof will be carried out in Lemmas \ref{lemma:U_confMap} and \ref{lemma:removableSing} below. Assume
$\delta$, $r_0$, $\e_0$ are as in Proposition
\ref{prop:smalldistImpliesHPin}. In view of Remark \ref{remark:
recentering}, we may assume that we are dealing with a solution of
\eqref{FBP_0} in $B_{r_0}$, which satisfies:
\begin{itemize}
\item $F(u)\cap \overline{B_{r_0/2}}$ consists of two arcs $F_L$ and
$F_R$; for some rotation $\rho_{r_0}$ the ends of $F_L$ belong to $\rho_{r_0}(\a_{L,r_0/2})$
and the ends of $F_R$ belong to $\rho_{r_0}(\a_{R,r_0/2})$, where
\begin{equation*}
       \alpha_{L,r} = \{x\in \del B_{r}: x_1<0, |x_2|<\delta r \} \quad \text{and} \quad \alpha_{R,r} = \{x\in \del B_{r}: x_1>0, |x_2|<\delta r \}.
\end{equation*}
\item The minimum distance $d(F_L, F_R) = s$ is realized for $0\in
F_L$ and some point $q\in F_R$ with $0< s < \eps_0 r_0$.
\item For every $s/\eps_0<r\leq r_0$,
\[
    |u(\rho y) - |y_2|| < \delta r \quad \text{in } B_r
    \quad \text{for some rotation } \rho = \rho_{r}.
\]
\item For every $s/\eps_0<r\leq r_0/2$, $F(u)\cap \overline{B_r}$ consists of two
arcs $F_L(r)$ and $F_R(r)$ that attach on $\rho_{r}(\alpha_{L,r})$
and $\rho_{r}(\alpha_{R,r})$.
\end{itemize}

Set
\begin{equation*}
    r_k :=2^{k-1} s/\e_0 \qquad k\in \nat
\end{equation*}
and let $k_0=\lfloor\log_2(r_0\e_0/s)\rfloor$, so that
$r_{k_0}\approx r_0/2$. Define $F_L^{N}$ and $F_L^{S}$ to be the two
(closed) subarcs of $F_L(r_{k_0})$ that $0$ divides $F_L(r_{k_0})$
into: with $F_L^N$ being the one such that the end point
$\rho_{r_{k_0}}^{-1}(F_L^N) \cap \alpha_{L,r_{k_0}}$ has the greater
$x_2$-coordinate than the end point $\rho_{r_{k_0}}^{-1}(F_L^S) \cap
\alpha_{L,r_{k_0}}$. Define $F_R^N$ and $F_R^S$, the two subarcs of
$F_R(r_{k_0})$ that $q$ divides $F_R(r_{k_0})$ into, analogously.
Let $\tau$ be the straight-line close segment connecting $0$ to $q$,
and let $\beta^N$ and $\beta^S$ be the two circular arcs of $\del
B_{r_{k_0}} \cap \{u>0\}$ with $\beta^N$ containing
$\rho_{r_{k_0}}\big((0, r_{k_0})\big)$ and $\beta^S$ containing
$\rho_{r_{k_0}}\big((0, -r_{k_0})\big)$. Then $\tau$ splits
$B_{r_{k_0}}^+(u)$ into two simply-connected regions -- the ``top"
$\Omega_N$, bounded by $\beta^N$, $F_L^N$, $\tau$, $F_R^N$; and the
``bottom" $\Omega_S$, bounded by $\beta^S$, $F_L^S$, $\tau$,
$F_R^S$.

We may choose the coordinate system so that $\rho_{r_1}$ is the
identity. In the following series of arguments we shall adopt
complex notation: denoting the point $(x_1, x_2)\in \real^2$ by the
complex $z= x_1 + i x_2\in \C$.

Let $z_k\in \C$ be the unique point of intersection between $\del
B_{r_k}$ and $F_R^N$, $k= 1, 2, \ldots, k_0$. The region $\Omega_N$
is simply connected with piece-wise smooth boundary, so we may
define the harmonic conjugate $v: \Omega_N \to \real$ of $u$, such
that $v$ is continuous up to the boundary $\del \Omega_N$ and has
the normalization
\begin{equation*}
    v(z_2) = -|z_2|.
\end{equation*}
Now define the holomorphic map $U:\Omega_N \to \C$ by
\begin{equation*}
    U := i u - v.
\end{equation*}

\begin{lemma}\label{lemma:U_confMap} The map $U$ constructed above
is injective on $\overline{\Omega_N\setminus B_{r_2}}$ and its image
\begin{equation}\label{lemma:U_confMap:eq_Annulus}
    U(\overline{\Omega_N\setminus B_{r_2}}) \supseteq \{\xi\in \C: \text{Im}(\xi)\geq 0, r_2(1+C\delta) \leq |\xi| \leq r_{k_0}(1-C\delta)\}
\end{equation}
for some numerical constant $C$.
\end{lemma}

\begin{proof}
%Denote by $A_k$ the dyadic annulus $\overline{B_{r_{k+1}}\setminus
%B_{r_k}}$.
First, let us note that for $k=2,\ldots, k_0-1$, the free boundary
in each dyadic annulus $F(u)\cap \overline{B_{r_{k+1}}\setminus
B_{r_k}}$ consists of four graphs of Lipschitz norm at most
$c'\delta$ for some numerical constant $c'>0$. This is a direct
consequence of Lemma \ref{lemma_fbgraphs} applied to $u$ in
$B_{r_k}(\pm p_k)$, where $p_k=\rho_{3r_k}\big((3r_k/2,0)\big)$,
since the zero phase of $u$ in $B_{3r_k}\supseteq B_{r_k}(\pm p_k)$
is contained in a $(3\delta r_k)$-strip that disconnects
$B_{r_k}(\pm p_k)^+(u)$ into two components. An a result, if we
represent the rotation $\rho_{r_k}$ as a complex phase
$e^{i\theta_k}$, we must have
\begin{equation}\label{lemma:U_confMap:eq_PhaseDif}
    |e^{i\theta_{k+1}} - e^{i\theta_k}|\leq c\delta,
\end{equation}
for the Lipschitz graph pieces $F(u)\cap
\overline{B_{r_{k+1}}\setminus B_{r_k}}$ to be appropriately aligned
in successive dyadic annuli.

We shall carry out the proof of the lemma in a couple of steps.
\item \textbf{Step 1}. Define $A_k := \Omega_N \cap \overline{B_{r_{k+1}}\setminus
B_{r_k}}$. We shall show that
\begin{equation}\label{lemma:U_confMap:eq_Main}
    |U( e^{i\theta_k}\z) - \z| \leq C\delta |\z| \quad \text{for} \quad
    \z\in \tilde{A}_k:= e^{-i\theta_k}A_k \qquad k=2, 3, \ldots,
    k_0-1.
\end{equation}
Define $\tilde{U}(\z):=U( e^{i\theta_k}\z)$ and let $\tilde{u} :=
\text{Im}(\tilde{U})$. First, claim that
\begin{equation}\label{lemma:U_confMap:eq_Der}
    |\tilde{U}'(\z) - 1| \leq c \delta \quad \text{for }  \z\in
    \tilde{A}_k.
\end{equation}
The Cauchy-Riemann equations say
\begin{equation*}
    \tilde{U}'(\z) = i\del_{y_1} \tilde{u} + \del_{y_2} \tilde{u},
    \qquad \z = y_1 + i y_2,
\end{equation*}
so it suffices to show that \[\nabla_y \tilde{u} = e_2 + O(\delta)
\quad \text{in} \quad \tilde{A}_k,\] where $e_2$ is the unit vector
in the direction of $y_2$. This is a straightforward corollary of
\begin{equation*}
    |\tilde{u} - y_2^+| < 3\delta r_k \quad \text{in} \quad e^{-i\theta_k}
    \big(\Omega_N \cap (B_{3r_{k}}\setminus
B_{r_k/2})\big) \supseteq \tilde{A}_k.
\end{equation*}
and the fact that $F(\tilde{u}) \cap (B_{3r_{k}}\setminus
B_{r_k/2})$ consists of two graphs of Lipschitz norm at most
$c'\delta$.

Going back to the complex coordinate $z = e^{i\theta_k} \z$, we see
that \eqref{lemma:U_confMap:eq_Der} becomes
\begin{equation}\label{lemma:U_confMap:eq_Der_Rot}
    |U'(z) - e^{-i\theta_k}| = |U'(z)e^{i\theta_k} - 1| \leq
    c\delta \quad \text{for}\quad z\in A_k.
\end{equation}
Let $z_k$ be defined as the unique intersection point between $\del
B_{r_k}$ and $F_R^N$ for $k= 1, 2, \ldots, k_0$, as above. Since
there is a piece-wise smooth curve $\gamma(z_k,z) \subseteq A_k$ of
length $O(r_k)$ connecting $z_k$ to any other point $z\in A_k$,
integrating $(U'(s) - e^{-i\theta_k})$ along $\gamma(z_k, z)$, we
obtain using \eqref{lemma:U_confMap:eq_Der_Rot}
\begin{equation}\label{lemma:U_confMap:eq_DerInteg}
    |U(z) - e^{-i\theta_k}z - (U(z_k) - e^{-i\theta_k}z_k)| \leq
    C'\delta r_k \quad z\in A_k.
\end{equation}
In order to establish \eqref{lemma:U_confMap:eq_Main}, it suffices
therefore to show that for some large enough numerical constant
$\tilde{c}$
\begin{equation*}
    |U(z_k) - e^{-i\theta_k}z_k|  \leq \tilde{c} r_k, \quad k=2, 3,
    \ldots, k_0-1.
\end{equation*}
We shall use induction. Without of loss of generality, the complex
coordinate $z$ is chosen so that $\theta_2 = 0$. Then, since $z_2\in
\alpha_{R, \delta r_2}$,
\begin{equation*}
    |U(z_2) - e^{-i\theta_2}z_2| = |-v(z_2) - z_2| = ||z_2| - z_2||
    \leq 2 \delta r_2.
\end{equation*}
Assume the statement is true for $k$. Applying
\eqref{lemma:U_confMap:eq_DerInteg} for $z = z_{k+1}\in A_k$
\begin{equation*}
    |U(z_{k+1}) - e^{-i\theta_k}z_{k+1}| \leq C'\delta r_k + |U(z_k) -
    e^{-i\theta_k}z_k| \leq (C'+\tilde{c})\delta r_{k}.
\end{equation*}
Taking into account \eqref{lemma:U_confMap:eq_PhaseDif}, we see that
\begin{equation*}
    |U(z_{k+1}) - e^{-i\theta_{k+1}}z_{k+1}| \leq
    (C'+\tilde{c})\delta r_{k} + |e^{-i\theta_{k+1}} -
    e^{-i\theta_{k}}||z_{k+1}| \leq  (C'/2 + \tilde{c}/2 + c)\delta
    r_{k+1}.
\end{equation*}
and the induction step is complete once we pick $\tilde{c} =
\max\{2, C' + 2c)\}$.

\item \textbf{Step 2}. We are now ready to show that $U$ is injective on $\overline{\Omega_N\setminus B_{r_2}}$. Let $w_1, w_2 \in \overline{\Omega_N\setminus B_{r_2}}$ be such
that $U(w_1) = U(w_2)$; without loss of generality $|w_1|\leq
|w_2|$. Because of \eqref{lemma:U_confMap:eq_Main}, we have
\begin{equation*}
    |U(w_1)|\leq (1+C\delta)|w_1| \quad \text{while} \quad |U(w_2)|
    \geq (1-C\delta)|w_2|.
\end{equation*}
Hence,
\begin{equation*}
    1 \leq \frac{|w_2|}{|w_1|}\leq
    \frac{|U(w_2)|/(1-C\delta)}{|U(w_1)|/(1+C\delta)}  = \frac{1 +
    C\delta}{1-C\delta} < 2.
\end{equation*}
so it has to be the case that both $w_1, w_2$ belong to $A_{k-1}\cup
A_k$ for some $k$. Because of \eqref{lemma:U_confMap:eq_Der_Rot} and
\eqref{lemma:U_confMap:eq_PhaseDif}, we have
\begin{equation*}
    |U'(z) - e^{-i\theta_k}| \leq
    c'\delta \quad \text{for}\quad z\in A_{k-1} \cup A_k.
\end{equation*}
Let $\gamma(w_1, w_2)$ be a piece-wise smooth curve in $D_k:=A_{k-1}
\cup A_k$ connecting $w_1$ to $w_2$. It is not hard to see that,
because $\del D_k$ can be locally represented as a graph of a
Lipschitz function with Lipschitz norm bounded by some universal
constant $L$, $\gamma(w_1, w_2)$ can be taken such that
\begin{equation*}
    \hsd(\gamma(w_1, w_2)) \leq \sqrt{1+L^2} |w_1-w_2|.
\end{equation*}
Then
\begin{equation*}
    0 = U(w_2) - U(w_1) = \int_{\gamma(w_1, w_2)} U'(z) dz =
    e^{-i\theta_k} (w_2 - w_1) + \int_{\gamma(w_1, w_2)} (U'(z) -
    e^{-i\theta_k}) dz,
\end{equation*}
so that
\begin{equation*}
    |w_1 - w_2| = \left|\int_{\gamma(w_1, w_2)} (U'(z) -
    e^{-i\theta_k}) dz\right| \leq c'\delta \hsd(\gamma(w_1, w_2))
    \leq  c'\sqrt{1+L^2} \delta |w_1-w_2|,
\end{equation*}
which implies that $w_1 = w_2$ when $\delta$ is small enough.
\item \textbf{Step 3.} Finally, we see that \eqref{lemma:U_confMap:eq_Annulus}
follows from \eqref{lemma:U_confMap:eq_Main} and the fact that
$\text{Im}(U) = u \geq 0$ with $\text{Im}(U)(z) = 0$ precisely when
$z\in F(u)\cap \overline{\Omega_N\setminus B_{r_2}}$.
\end{proof}

\begin{lemma}\label{lemma:removableSing} The two curves $F(u) \cap \overline{\Omega_N\setminus
B_{r_3}}$ are graphs over the line $\rho_{r_{k_0}}\{y_2=0\}$ with
Lipschitz norm at most $c\delta$ for some numerical constant $c$.
\end{lemma}
\begin{proof}
From Lemma \ref{lemma:U_confMap} we know that the inverse of $U$ is
well defined on the annulus
\begin{equation*}
    A:=\{\xi\in \C: \text{Im}(\xi)\geq 0, r_2(1+C\delta) \leq |\xi| \leq
    r_{k_0}(1-C\delta)\}.
\end{equation*}
Then $U^{-1}\circ \exp$ maps the strip
\begin{equation*}
    S=\{z\in \C: 0\leq \text{Im}(z)\leq
    \pi, \log \big(r_2(1+C\delta)\big)\leq \text{Re}(z) \leq  \log\big(r_{k_0}(1-C\delta)\big)\}.
\end{equation*}
conformally onto its image in $\overline{\Omega_N\setminus
B_{r_2}}$: with $S\cap \{\text{Im}(z)=0\})$ parameterizing a subarc
of the ``right" strand $F_R^N$, and $S\cap \{\text{Im}(z)=\pi\})$
parameterizing a subarc of the ``left" strand $F_L^N$, under
$U^{-1}\circ \exp$ (see the discussion at the beginning of the
section for definitions). Since $U' \neq 0$ on $\Omega_N\setminus
B_{r_2}$ and $\Omega_N\setminus B_{r_2}$ is simply-connected, one
may define a branch of its logarithm $\log U'$. Finally, define the
holomorphic function $\mathcal{F}:S\to\C$ via:
\begin{equation*}
    \mathcal{F}:= \log U'\circ U^{-1}\circ \exp,
\end{equation*}
and let
\begin{equation*}
    f= \text{Re}(\mathcal{F}) \quad \text{and}\quad g=\text{Im}(\mathcal{F}).
\end{equation*}
Since for $\z_1, \z_2\in F(u)$, $\left|\text{Im}\big(\log U'(\z_2) -
\log U'(\z_1)\big)\right|$  measures the angle of turning of $\nabla
u$ along $F(u)$ from $\z_1$ to $\z_2$, we are going to be interested
in estimating the oscillations
\begin{equation*}
    \omega_{g, L} := \text{osc}\{g(z): z \in \tilde{S} \cap \{\text{Im}(z)=\pi\}\} \qquad \omega_{g, R} := \text{osc}\{g(z): z \in \tilde{S} \cap
    \{\text{Im}(z)=0\}\}.
\end{equation*}
where
\begin{equation*}
    \tilde{S}=\{z\in \C: 0\leq \text{Im}(z)\leq
    \pi, ~ \log r_2 + c_0 \leq \text{Re}(z) \leq  \log r_{k_0} - c_0 \} \subseteq
    S \qquad c_0=\log 2.
\end{equation*}
We would like to show that both $\omega_{g, L}$ and $\omega_{g, R} =
O(\d)$, as this would imply that the amount of turning of $\nabla u$
along $F_{L}^N$ ($F_R^N$), from $\del B_{2r_2}$ to $\del
B_{r_{k_0}/2}$, is $O(\d)$, which, in turn, would be enough to
conclude that $F_{L}^N \cap \overline{B_{r_{k_0}}\setminus B_{r_3}}$
and $F_{R}^N \cap \overline{B_{r_{k_0}}\setminus B_{r_3}}$ are in
fact graphs of Lipschitz constant $O(\delta)$ (as we already know
that $F_{L}^N \cap \overline{B_{r_{k_0}}\setminus B_{r_{k_0}/2}}$
and $F_{R}^N \cap \overline{B_{r_{k_0}}\setminus B_{r_{k_0}/2}}$ are
graphs of Lipschitz constant $O(\delta)$). To that goal we would
like to obtain estimates on $|\nabla g|$ in $\del S$, which by the
Cauchy--Riemann equations satisfies
\begin{equation*}
    |\nabla g| = |\nabla f| \quad \text{in } S.
\end{equation*}
For convenience, define the following coordinates on $S$
\begin{align*}
     t = \text{Re}(z) - (\log\big(r_2(1+C\delta)\big) +A) \qquad & \theta = \text{Im}(z)-\pi/2, \quad   \text{where
    } \\ 2A  = \log\big(r_{k_0}(1-C\delta)\big)-\log
    \big(r_2(1+C\delta)\big) & = \log \frac{1-C\delta}{1+C\delta} + (k_0-2)\log
    2 = (k_0-2)\log 2 + O(\d)
\end{align*}
by translating the coordinates $(\text{Re}(z), \text{Im}(z))$
appropriately, so that $S$ is parameterized by
\begin{equation*}
    S = \{|t|\leq A, |\theta| \leq \pi/2\}.
\end{equation*}
Note that since $|U'|=|\nabla u|$ on $F(u)\cap (\Omega_N\setminus
B_{r_2})$, we have
\begin{equation*}
    |f(t, \pm\pi/2)|= \log |\mathcal{F}| = \log 1 =0 \quad
    \text{for} \quad    |t|\leq A.
\end{equation*}
Also, by the estimate \eqref{lemma:U_confMap:eq_Der} of Lemma
\ref{lemma:U_confMap}, we have
\begin{equation*}
    |f(\pm A, \theta)| \leq c\delta \quad
    \text{for} \quad  |\theta|\leq \pi/2.
\end{equation*}
Applying Schwarz reflection across $\theta =-\pi/2$ and $\theta =
\pi/2$, we can extend $f$ to a harmonic function on
$\hat{S}:=\{|t|\leq A,|\theta|\leq 3\pi/2 \}$. By the maximum
principle $|f|\leq c\delta$ in $\hat{S}$. Denote
$\tilde{A}:=A-c_0/2$. Interior estimates for $f$ then yield
\begin{equation}\label{lemma:removableSing:eq_derfSide2}
    |\nabla f(\pm\tilde{A},\th)|\leq \tilde{c}
    \|f\|_{L^{\infty}(\hat{S})} \leq \tilde{C} \delta \quad \text{for}\quad |\th|\leq
    \pi/2,
\end{equation}
which we can integrate to get
\begin{equation*}
    |f(\pm\tilde{A},\th)|\leq \tilde{C}\d \cos\theta \quad \text{for}\quad |\th|\leq
    \pi/2.
\end{equation*}
Using multiples of $\cosh t\cos \th$ as upper and lower barriers for
$f$, we have the bound
\begin{equation*}
    -(c\delta/\cosh \tilde{A}) \cosh t \cos\th\leq f \leq (c\delta/\cosh \tilde{A}) \cosh t \cos\th \quad \text{in} \quad
    \tilde{S},
\end{equation*}
so that, by the Hopf Lemma, we can conclude
\begin{equation}%\label{lemma:removableSing:eq_derfSide1}
    |\nabla f(t, \pm \pi/2)| = |\del_{\th}f(t,\pm\pi/2)| \leq (c\delta/\cosh \tilde{A}) \cosh
    t \quad \text{for} \quad |t|\leq \tilde{A}.
\end{equation}
This in turn implies the desired
\begin{equation*}
    \omega_{g, L} \leq \int_{-\tilde{A}}^{\tilde{A}} |\nabla g(t,\pi/2)|~dt \leq 2c\d,
    \qquad \omega_{g, R} \leq \int_{-\tilde{A}}^{\tilde{A}} |\nabla g(t,-\pi/2)|~dt
    \leq 2c\d.
\end{equation*}
\end{proof}

\section{Curvature bounds.}\label{Sec:CurvBounds}

Let us describe the family of hairpin solutions explicitly. Define
\begin{equation*}
    \f(\z) = i(\z + \sinh \z).
\end{equation*}
Then $\f$ maps the strip $S=\{|\text{Im}\z|<\pi/2\}$ conformally
onto the domain
\begin{equation*}
    \Omega_1:=\{z\in\C: |\text{Re}z|< \pi/2 + \cosh (\text{Im}z) \}
\end{equation*}
which supports the positive phase of the hairpin solution of
\eqref{FBP_0}
\begin{equation}\label{eq:haipin_def}
    H(z) = \left\{\begin{array}{cl}\text{Re}\big(V(z)\big) & \text{when } z\in \Omega_1 \\ 0 & \text{otherwise.} \end{array}\right.
\end{equation}
where $V(z):= \cosh(\f^{-1}(z))$. The dilates
\begin{equation*}
    H_a(z) = a H(z/a)
\end{equation*}
complete the family of hairpin solution (up to rigid motions).
Denote by $\Omega_a = a\Omega_1$.

We note a couple of geometric features of these solutions.
\begin{itemize}
\item The zero phase $\{H_a = 0\}$ consists of two connected components separated by
distance $s=a(2+\pi)$.
\item $H_a\big|_{\Omega_a}$ has a unique critical point (a non-degenerate saddle) and it is situated at
the origin. Indeed, to verify this, we have to simply check this is
the obviously the case for
\begin{equation*}
    H(\f(\z)) = \text{Re} \cosh(\z) = \cosh(y_1)\cos(y_2) \quad \text{when}\quad \z=y_1+iy_2\in S.
\end{equation*}
The value of $H_a$ at the saddle is precisely $H_a(0)=a$.
\item The segments $\tau_{a,L} := [-s/2, 0]\subseteq \C$ and $\tau_{a,R} := [0, s/2]\subseteq
\C$ are the steepest descent paths from $0$ to each of the two
components of $\{H_a=0\}$, respectively. We shall denote
$\tau_a:=\tau_{a, L}\cup\tau_{a,R}$.
\end{itemize}
The following information about the gradient $\nabla H$ will also be
useful.
\begin{lemma}\label{lemma:hairpinGrad} For some numerical constant $c_0>0$, the gradient $\nabla H$ satisfies
\begin{equation*}
|\nabla H(x)|\geq \min(1/2,c_0|x|) \quad \text{when} \quad x\in
\Omega_1.
\end{equation*}
\end{lemma}
\begin{proof}
We have
\begin{equation*}
|\nabla H |(\f(\z)) = \left|\frac{i\sinh\z}{\f'(\z)}\right| =
\sqrt{\frac{\sinh^2 y_1 + \sin^2 y_2}{(1+\cosh y_1\cos y_2)^2 +
\sinh^2y_1 \sin^2 y_2 }} = \frac{\sqrt{\sinh^2 y_1 + \sin^2
y_2}}{\cosh y_1 + \cos y_2} \quad \z=y_1+iy_2 \in S.
\end{equation*}
Thus,
\begin{equation*}
    |\nabla H |(\f(\z)) \geq \frac{|\sinh y_1|}{1 + \cosh y_1} =
    |\tanh (y_1/2)| \geq 1/2 \quad \text{when} \quad |y_1|\geq 1.2
\end{equation*}
We'll be done once we show that
\begin{equation*}
     |\nabla H |(\f(\z)) \geq c_0 |\f(\z)| \quad \text{when }|\text{Re}\z| =
     |y_1| < 1.2.
\end{equation*}
Since for some numerical constants $0<c_1<c_2$
\begin{equation*}
    c_1 |\z| \leq |\sinh \z| \leq c_2 |\z| \quad \text{when } |\text{Re}\z| =
     |y_1| < 1.2
\end{equation*}
we have for $|\text{Re}\z| < 1.2$,
\begin{equation*}
    |\nabla H |(\f(\z)) \geq \frac{|\sinh \z|}{\cosh y_1 + \cos y_2}
    \geq \frac{c_1|\z|}{\cosh(1.2)+ 1} \geq \tilde{c}_1 |\z|.
\end{equation*}
Noting that $|\f(\z)|\leq |\z|+|\sinh z|\leq (1+c_2)|\z|$ when
$|\text{Re}\z| < 1.2$, we complete the proof of the lemma.
\end{proof}

\begin{remark}\label{remark:mapprops}
Finally, we would like to make the following remark regarding the
mapping properties of $V_a(z) := aV(z/a)$ on $\Omega_a$. Claim that
$V_a$ maps both $\Omega_a^+:=\Omega_a\cap \{x_2>0\}$ and
$\Omega_a^-:=\Omega_a\cap \{x_2<0\}$ conformally onto
\begin{equation*}
    \tilde{\H}_a:=\{\xi: \text{Re}(\xi)>0\}\setminus(0,a].
\end{equation*}
Indeed, let $S_{\pm} = \{\zeta \in \C: \pm\text{Re }\z>0, |\text{Im
}\z|<\pi/2\}$. Then $(a\f)$ is a conformal map from $S_{\pm}$ onto
$\Omega_a^{\pm}$ and
\begin{equation*}
    V_a(a\f(\z)) = a\cosh(\z) =  a\cosh y_1\cos y_2 + i a \sinh
    y_1\sin y_2 \quad \text{when}\quad \z=y_1+iy_2 \in S_{\pm}.
\end{equation*}
Write
\begin{equation*}
    V_a(a\phi(\z)) = r(\z)e^{i\th(z)}
\end{equation*}
where
\begin{equation*}
    r(\z)^2 /a^2 = |V_{a}\phi(\z))|^2 = \sinh^2y_1+\cos^2y_2 \qquad
    \tan\theta(\z) = \tanh y_1 \tan y_2.
\end{equation*}
If $\theta_0$ is any angle in $(-\pi/2,0)\cap(0,\pi/2)$ and
$c:=\tan\theta_0$, then $\tan\theta(\z) = c$ whenever $\tan y_2 = c
\coth y_1$, so that for these values of $\zeta$,
\begin{equation*}
    r(\z)^2 /a^2 = \sinh^2 y_1 + 1/(\tan^2 y_2+ 1) =  \sinh^2 y_1 +
    1/(1 + c^2\coth^2(y_1)).
\end{equation*}
We see that $r(\z)\to 0$ as $y_1\to 0$ and $r(\z)\to \infty$ as
$y_1\to \pm \infty$, so $V_a(a\f)|_{S^{\pm}}$ is onto $\H_{\infty}$.
When $\theta_0 = 0$, i.e. $\text{Im}(V_a(a\f)) = 0$, so it must be
that $y_2 = 0$ and thus, $r(\z) = a\cosh y_1$ which ranges in
$(a,\infty)$. Hence, $V_a(a\f)|_{S^{\pm}}$ is surjective onto
$\H_a$. To show that say $V_a(a\f)|_{S^{+}}$ is injective, assume
that for some $y_1 + iy_2, \tilde{y}_1+i\tilde{y_2}\in S^+$ with
$y_1\leq \tilde{y}_1$,
\begin{equation*}
    \cosh y_1\cos y_2 = \cosh \tilde{y}_1\cos \tilde{y}_2 \quad
    \text{and} \quad \sinh y_1\sin y_2 = \sinh \tilde{y}_1\sin
    \tilde{y}_2.
\end{equation*}
Divide the second equation by the first to get
\begin{equation*}
    \tanh y_1\tan y_2 = \tanh \tilde{y}_1\tan \tilde{y}_2.
\end{equation*}
We see that if $y_1=\tilde{y_1}$ we must have $y_2=\tilde{y}_2$ too.
Assume $y_1 < \tilde{y}_1$; then either $y_2 = \tilde{y}_2 = 0$,
which then implies $\cosh y_1 = \cosh \tilde{y_1}$ contradicting the
assumption $0<y_1<\tilde{y}_1$, or $|\tan y_2|>|\tan \tilde{y}_2|$
which implies $\cos y_2 < \cos \tilde{y}_2$. But in the latter case,
\begin{equation*}
    \cosh y_1\cos y_2 <\cosh \tilde{y}_1\cos \tilde{y}_2,
\end{equation*}
so we get a contradiction again. Similarly, we show the injectivity
of $V_a(a\f)|_{S^{-}}$.

\end{remark}

Having amassed enough information about the model hairpin solutions
let us explore how well they approximate classical solutions of
\eqref{FBP_0} whose zero phase has two connected components that are
sufficiently close to each other.

\begin{prop}\label{prop:hairpin_loc} Let $u$ be a classical solution of \eqref{FBP_0} in $B_1$ that satisfies \eqref{topo_assmpt}. Assume that $\{u=0\}$
consists of two connected components $Z_L$ and $Z_R$ and that $0$ is
the midpoint of a shortest segment between $Z_L$ and $Z_R$. For any
given $\d_1>0$ and every $M>0$, there exists $s_1>0$ such that if
\begin{equation*}
    s:=d(Z_L, Z_R) \leq s_1,
\end{equation*}
then after a rotation
\begin{equation}\label{prop:hairpin_loc:eq_estimate}
    |u(ax)/a - H(x)| \leq \d_1 \quad \text{for all} \quad |x|\leq M,
    \end{equation}
where $a=s/(2+\pi)$.
\end{prop}
\begin{proof}
Fix $\d_1$ and $M$ and assume no such $s_1$ exists that makes the
proposition valid. Then for some sequence of $s_k\to 0$, there is a
sequence of counterexamples $u_k$ with the separation between the
two components of $\{u_=0\}$ being $s_k$. Set $a_k = s_k/(2+\pi)$.
We can then define the rescales
\begin{equation*}
    \tilde{u}_k (x):=u(a_k x)/a_k \quad \text{for } x\in B_{1/s_k}
\end{equation*}
so that the separation between the two components of the zero phase
of $\tilde{u}_k$ is precisely $(2+\pi)$ and $0$ is at the midpoint
of a shortest segment connecting them. A subsequence $u_{k_j}$
converges uniformly on $\overline{B_M}$ to a global solution
$\tilde{u}$ and since
\begin{equation*}
    d_H((\overline{B_M})^+(\tilde{u}_{k_j}), (\overline{B_M})^+(\tilde{u}))
    \to 0 \quad \text{as } j\to \infty,
\end{equation*}
it has to be the case that $\tilde{u}$ is a hairpin solution, with
separation between the two components of $\{\tilde{u}=0\}$ precisely
$(2+\pi)$ and $0$ at the midpoint of the shortest segment. Thus, in
a rotated coordinate system
\begin{equation*}
    \tilde{u} = H,
\end{equation*}
so we have for all $j$ large enough
\begin{equation*}
    |u_{k_j}(a_{k_j} x)/a_{k_j} - H(x)|\leq \delta_1 \quad \text{in
    } \overline{B_{M}}.
\end{equation*}
This contradicts the assumption that the $u_{k_j}$ are
counterexamples.
\end{proof}

The next corollary is a direct consequence of interior estimates
applied to the proposition above.
\begin{coro}\label{coro:hairpinC^2} Let $u$ be as in Proposition \ref{prop:hairpin_loc} and let $H$ be the hairpin solution defined above. For every
$M>0$, any compact domains $K, K'$ such that $K\Subset K' \Subset
(\overline{B_M})^+(H)$ and any $\d_1>0$ such that
$\mathcal{N}_{2\d_1}(K')\subseteq (\overline{B_M})^+(H)$, there
exists $s_1>0$, such that if the separation between the two
components $Z_L$ and $Z_R$ of $\{u=0\}$ satisfies
\begin{equation*}
    s: = d(Z_L, Z_R) \leq s_1,
\end{equation*}
then for $a=s/(2+\pi)$, in some rotated coordinate system, the
rescale $u_a:=u(ax)/a$ satisfies
\begin{equation}\label{coro:hairpinC^2:eq_estimate}
    \|u_a - H\|_{C^2(K)} \leq C_{K,K'}\d_1
\end{equation}
for some constant $C_{K,K'}$, dependent on $K$ and $K'$.
Furthermore, if $\d_1 = \d_1(H)$ is small enough, $u$ has a unique
critical point $x_{0}$ in $B_{a M/2}^+(u)$ which is a non-degenerate
saddle point with $|x_{0}| = O(\d_1 s)$, and the steepest descent
paths $\beta_L$, $\beta_R$ for $u$ from $x_0$ to $Z_L$ and $Z_R$,
respectively, are contained in an $O(\d_1s)$-neighborhood of
$\tau_a$ (defined above).
\end{coro}

\begin{proof}
Let $s_1$ be such that according to
\eqref{prop:hairpin_loc:eq_estimate}
\begin{equation*}
|u_a - H(x)| \leq c\d_1\quad \text{for all} \quad |x|\leq M,
\end{equation*}
where we pick $c$ such that $H(x)\geq c d(x,F(H))$ (such a $c$
exists because of the non-degeneracy of the hairpin solution). Then
for all $x \in K'$
\begin{equation*}
    u_a(x) \geq H(x) - c\delta_1 \geq c d(x, F(H)) - c\delta_1 > 2c\delta_1 - c\d_1 >
    0.
\end{equation*}
Hence $v = u_a - H$ is harmonic in $K'$ and we get
\eqref{coro:hairpinC^2:eq_estimate} by standard interior estimates.

Let us use this to show that $u$ has a unique critical point in
$B_{a M/2}^+(u)$ if $\d_1$ is small enough. Fix $\d_0>0$ small and
find a scale $r_0=r_0(\d_0, H)$ such that for every $p\in F(H)\cap
\overline{B_{M/2}}$
\[d_H( F(H)\cap B_{r_0}(p), L(p)\cap B_{r_0}(p)) <\delta_0 r_0,\]
where $L(p)$ denote the straight line tangent to $F(H)$ at $p$. Now,
for all small enough $\d_1<\d_0 r_0$
\[d_H( F(u_a)\cap B_{r_0}(p), L(p)\cap B_{r_0}(p)) <c'\delta_0 r_0.\]
Hence, $F(u_a)$ is $c'\d_0$-flat in $B_{r_0}(p)$ and it must be that
for a small enough $\d_0$,
\begin{equation}\label{coro:hairpinC^2:eq_bryestimate}
    |\nabla u_a - \nabla H(p)| \leq C\delta_0 \quad \text{in} \quad
    B_{r_0/2}(p)^+(u_a) \quad \text{for any } p\in F(H)\cap
\overline{B_{M/2}}
\end{equation}
by the classical Alt-Caffarelli theory \cite{AC}. Thus,
\begin{equation*}
    |\nabla u_a| \geq 1-C\delta_0 \quad \text{in} \quad
    B_{r_0/2}(p)^+(u_a)
\end{equation*}
and so it suffices to show that $u_a$ has a unique critical point in
\begin{equation*}
    K:=\{x\in B_{M/2}^+(H): d(x, \del\Omega_1\}\geq r_0/4\}.
\end{equation*}
Set $K':= \{x\in B_{2M/3}^+(H): d(x, \del\Omega_1\}\geq r_0/8\}$. We
would like to show that
\begin{equation*}
    0 = \nabla u_a = \nabla H + \nabla v
\end{equation*}
has a unique solution $x_{0,a}$ in $K$. Since the Jacobian of
$\nabla H = D^2 H$ is invertible at $0$, the Inverse Function
Theorem implies that for some $c_1>0$, $\nabla H$ maps $B_{c_1}$
diffeomorphically onto a neighborhood $O$ of $0$. As $|\nabla
v|\leq C_{K,K'}\d_1$ in $K$, we can choose $\d_1$ small enough such
that $\nabla v \in O$, whence
\begin{equation*}
    \nabla H(x) = -\nabla v(x)
\end{equation*}
has a unique solution $x=x_{0,a}\in B_{c_1}$. Applying Lemma
\ref{lemma:hairpinGrad}, we obtain
\begin{equation*}
    |x_{0,a}|\leq c_0^{-1} |\nabla H(x_{0,a})| = c_0^{-1} |\nabla v(x_{0,a})| \leq
    c_0^{-1}C_{K,K'}\d_1.
\end{equation*}
Furthermore, the equation cannot have another solution in $K$ if
$\d_1$ is small enough, because Lemma \ref{lemma:hairpinGrad}
implies that
\begin{equation*}
    |\nabla H(x)| \geq \min(1/2, cc_1) > C_{K,K'}\d_1 \geq |\nabla
    v(x)| \quad \text{for all}\quad x\in K\setminus B_{c_1}.
\end{equation*}
Thus, whenever $\d_1$ is small enough, $u_a$ has a unique critical
point $x_{0,a}$ in $K$ and since
\begin{align*}
    |D^2 u_a(x_{0,a}) - D^2 H(0)|& \leq |D^2 u_a(x_{0,a}) - D^2
    H(x_{0,a})| + |D^2 H(x_{0,a}) - D^2 H(0)| \\ & = O(\d_1) +
    O(|x_{0,a}|) = O(\d_1),
\end{align*}
$x_{0,a}$ is a non-degenerate saddle point.

The $O(\d_1)$ proximity between the steepest descent paths for $u_a$
and $H$ from $x_{0,a}$ and $0$, respectively, to their zero sets,
follows from the $O(\d_1)$ bound for $\nabla(u_a - H)$ in $K$ and
\eqref{coro:hairpinC^2:eq_bryestimate}.
\end{proof}

Let $\d_0>0$ be small constant from Remark
\ref{remark:deltaFlatness}. Let us present the set-up that we shall
be working in for the rest of the section. The object of interest is
\begin{enumerate}
\item A classical solution $u$ of \eqref{FBP_0} in $B_1$ that satisfies \eqref{topo_assmpt}
such that $\{u=0\}$ consists of two connected components $Z_L$ and
$Z_R$ with $0$ being at the midpoint of the shortest segment between
the two.
\item The free boundary $F(u)$ consists of two arcs $F_L:=F(u)\cap\del Z_L$ and $F_R=F(u)\cap\del Z_R$.
\item We assume that $\d_1 \leq \d < \d_0$, $M$, $s_1$, $s$ are as in Proposition
\ref{prop:hairpin_loc} and Corollary \ref{coro:hairpinC^2}, i.e. the
fact that $d(Z_L, Z_R) = s < s_1$ implies
\begin{equation*}
    |u - H_a|\leq \d_1 a \quad \text{in}\quad B_{aM} \quad \text{where}\quad a=(2+\pi)s
\end{equation*}
and $u$ has a unique critical point $x_0$ in $B_{aM/2}^+(u)$. We
denote by $\beta_L$ be the steepest descent path for $u$ that
connects $x_0$ to some $p\in F_L$ and by $\beta_R$ be the steepest
descent path from $x_0$ to some $q\in F_R$. Then $\beta:= \b_L\cup
\b_R$ is a smooth arc connecting $F_L$ to $F_R$ and $\beta\subseteq
\mathcal{N}_{a\d_1}(\tau_a)$. Without loss of generality, we may
assume that our coordinate system is chosen in such a way that
\begin{equation*}
    \nabla u(p) = e_1.
\end{equation*}

\item Furthermore, for some rotation $\rho$
\begin{equation*}
    |u(\rho x) - |x_2||\leq \d \quad \text{in all of} \quad B_1.
\end{equation*}
and $F(u)\cap (B_{2/3}\setminus B_{4Ms})$ consists of four graphs
over $\rho(\{x_2=0\})$ of Lipschitz norm at most $C\d$.
\end{enumerate}

% Applying the familiar argument from Section
% \ref{Sec:lipboundFBstrands} we see that $F(u)\cap B_{2/3}\setminus
%B_{1/3}$ consists of four graphs over the line $\rho(\{x_2=0\}$ with
%Lipschitz norm $O(\d_1)$,
Let $\del B_{1/2}$ intersect $F_L$ at the two points $p_N$ and $p_S$
(subscripts $N$ and $S$ are determined by $x_2(\rho^{-1}(p_N))>
x_2(\rho^{-1}(p_S))$) and similarly $\del B_{1/2}$ intersects $F_R$
at the two points $q_N$ and $q_S$. Define $\Omega_N\subseteq
B_{1/2}^+(u)$ to be the domain bounded by the subarc of $F_L$ from
$p_N$ to $p$, the arc $\beta$, the subarc of $F_R$ from $q$ to $q_N$ and by
the circular arc of $\del B_{1/2}$ with ends $p_N$ and $q_N$, which
contains $\rho(0,1/2)$. Analogously, define $\Omega_S$ to be the
domain bounded by subarc of $F_L$ from $p_S$ to $p$, the arc $\beta$, the
subarc of $F_R$ from $q$ to $q_S$ and by the circular arc of $\del
B_{1/2}$ with ends $p_S$ and $q_S$, which contains $\rho(0,-1/2)$.
Then
\begin{equation*}
    B_{1/2}^+(u) = \Omega_N\sqcup\beta\sqcup \Omega_S.
\end{equation*}
Let $v$ be the harmonic conjugate of $u$ in the simply-connected
$B_1^+(u)$ where we choose the normalization
\begin{equation*}
    v(x_0) = 0.
\end{equation*}
Note that this implies $v = 0$ on all of $\beta$, as $\nabla v$ is a
rotation by $\pi/2$ of $\nabla u$ which itself is tangent to
$\beta$. Furthermore $v$ is increasing (decreasing) at unit speed
along $F_L\cap\Omega_N$ ($F_L\cap\Omega_S$) and decreasing
(increasing) at unit speed along $F_R\cap\Omega_N$
($F_R\cap\Omega_S$) as we move towards $\del B_{1/2}$.

Define the holomorphic map $U:B_1^+(u)\to \C$ via
\begin{equation*}
    U = u + i v.
\end{equation*}

The next lemma confirms that the mapping properties of $U$ are
similar to those of $V_a$ (defined in Remark \ref{remark:mapprops}),
which allows us to construct an injective holomorphic map from
$B_{1/2}^+(u)$ to $\Omega_{a_0}$ for some $a_0>0$.

\begin{lemma}\label{lemma:confmapConstruction} Provided $\d$ and $\d_1$ small
enough, $U$ is injective on each of $\Omega_N$ and $\Omega_S$ and
maps each of $\beta_L$ and $\beta_R$ injectively onto $i[0,a_0]$,
where $a_0 := u(x_0) = a(1 + O(\d_1))$. Then the map $\tilde{\psi}:
B_{1/2}^+(u)\setminus \beta \to \Omega_{a_0}$ given by
\begin{equation*}
    \tilde{\psi}(z):=\begin{array}{cl}\big(V_{a_0}|_{\Omega_{a_0}^+}\big)^{-1}\circ U(z) & \text{when} \quad z\in \Omega_N \\
\big(V_{a_0}|_{\Omega_{a_0}^-}\big)^{-1}\circ U(z) & \text{when}
\quad z\in \Omega_S
\end{array}
\end{equation*}
is injective and in fact extends continuously to $\beta$. The
extension $\psi: B_{1/2}^+(u) \to \Omega_{a_0}$ defines, therefore,
an injective holomorphic map whose image contains
\begin{equation*}
     \psi(B_{1/2}^+(u)) \supseteq \Omega_{a_0}\cap B_{1/4}
\end{equation*}

\end{lemma}

\begin{proof}
Let us first show that $U$ is injective in $\Omega_N$. Since $U$
maps each of $\beta_L$ and $\beta_R$ injectively onto $[0,a_0]$ and
since near $x_0$
\begin{equation*}
    U(z) = a_0 + c(z-x_0)^2 + O(|z-x_0|^3)
\end{equation*}
by the smoothness of $U$, for every small enough $\eps>0$ we can
find an arc $\b_{\eps}\subseteq \Omega_N\cap
\mathcal{N}_{\eps}(\beta)$ connecting $p_{\eps}\in F_L$ to
$q_{\eps}\in F_R$, such that $U$ maps it bijectively onto an arc
$\gamma_{\eps}\subseteq \H_{a}$ with ends $U(p_{\eps})$ and
$U(q_{\eps})$, where
\begin{equation*}
    \text{Re}(U(p_{\eps})) = \text{Re} (U(q_{\eps})) = 0 \quad \text{and}
    \quad \text{Im}(U(p_{\eps})) = v(p_{\eps}) > 0 > v (q_{\eps}) = \text{Im} (U(q_{\eps}))
\end{equation*}
Let $\Omega_{N,\eps}$ be the domain bounded by the subarc of $F_L$
with ends $p_N$ and $p_{\eps}$, $\beta_{\eps}$, the subarc of $F_R$
with ends $q_{\eps}$ and $q_N$, and the corresponding circular arc
$\widehat{p_Nq_N}$ of $\del B_{1/2}$. Claim that $U$ is injective on
the closed Jordan arc $\del \Omega_{N,\eps}$. We can easily see that
$U$ maps $(F(u)\cap \del \Omega_{N,\eps})\cup \b_{\eps}$ injectively
onto
\begin{equation*}
    \Gamma_{\eps}:= \gamma_{\eps} \cup \{y_1= 0, y_2\in [-l_L, l_R]\}\setminus \{y_1 = 0, y_2\in (v(q_{\eps}),
    v(p_{\eps}))\}
\end{equation*}
where \[l_L := \hsd(\del \Omega_N \cap F_L) \geq 2/5 \quad
\text{and}\quad l_R := \hsd(\del \Omega_N \cap F_R) \geq 2/5.\] It
remains to confirm that $U$ is injective on $\widehat{p_Nq_N}$ and
that $U(\widehat{p_Nq_N})\cap \Gamma_{\e}=U(\widehat{p_Nq_N})\cap
\gamma_{\eps}=\emptyset$. Those follow easily from the fact that
\begin{equation}\label{lemma:confmapConstruction_eq_halfcirc}
    |U'(z)e^{i\theta} - (-i)| \leq c\delta \quad z\in\widehat{p_Nq_N}
\end{equation}
where $e^{i\theta}$ represents the rotation $\rho$. %The latter is
%inferred from item $(4)$ in the set-up above just like in the proof
%of Lemma \ref{lemma:U_confMap}.

Since $U$ is injective on $\del \Omega_{N,\eps}$, $U(\del
\Omega_{N,\eps})$ is a closed Jordan arc that divides $\C$ into a
bounded domain $D_b$ and an unbounded domain $D_u$. For $\xi_0
\notin U(\del \Omega_{N,\eps})$,
\begin{equation*}
    Q(\xi_0):=\frac{1}{2\pi i}\oint_{\del \Omega_{N,\eps}} \frac{dz}{U(z) - \xi_0} = \frac{1}{2\pi i}\oint_{U(\del \Omega_{N,\eps})} \frac{d\xi}{\xi - \xi_0}
\end{equation*}
equals the winding number of the closed Jordan arc $U(\del
\Omega_{N,\eps})$ around $\xi_0$, i.e. $Q(\xi_0) = 1$ when $\xi_0\in
D_b$ and $Q(\xi_0) = 0$ when $\xi\in D_u$. On the other hand, by the
Argument Principle, $Q(\xi_0)$ equals the number of zeros (with
multiplicities) of $U(z) = \xi_0$ in $\Omega_{N,\eps}$. We can thus
conclude that $U$ is injective on $\Omega_{N,\eps}$.

Taking a sequence $\eps_k\to 0$ we construct a sequence of domains
$\Omega_{N,\eps_k}$ such that $\Omega_N = \bigcup_k
\Omega_{N,\eps_k}$ with $U$ injective on each $\Omega_{N,\eps_k}$.
Therefore, $U$ is injective on all of $\Omega_N$. Analogously, we
establish the injectivity of $U$ on $\Omega_S$.

Finally, let's show that $\psi$ extends continuously to $\beta$. Let
$z$ belong to the interior of the arc $\beta_L$, and let
$\{z_{N,k}\}\subseteq \Omega_N$, $\{z_{S,k}\}\subseteq \Omega_S$ be
two sequences such that both
\begin{equation*}
    z_{N,k} \to z \quad  z_{S,k} \to z.
\end{equation*}
Denote $\xi_{N,k} := U(z_{N,k})$ and $\xi_{S,k}:=U(z_{S,k})$. Then
we see that both
\begin{equation*}
    \xi_{N,k} \to u(z) + i0^+   \quad \text{and} \quad \xi_{S,k} \to u(z) + i 0^+
\end{equation*}
with $u(z)\in (0,a_0)$. Then if $\z_{N,k} =
V_{a_0}^{-1}|_{\Omega_{a_0}^+}(\xi_{N,k})$ and $\z_{S,k} =
V_{a_0}^{-1}|_{\Omega_{a_0}^-}(\xi_{N,k})$, we can easily verify
that
\begin{equation*}
    \z_{N,k} \to b+i0^+ \quad \text{and} \quad \z_{N,k} \to b+i0^-
\end{equation*}
with $b\in \tau_{a_0,L}$ being the unique point of $\tau_{a_0, L}$
that $V_{a_0}$ maps to $u(z)\in (0,a_0)$. Hence, $\tilde{\psi}$ can
be continuously extended on the interior of $\beta_L$ and similarly,
onto the interior of $\beta_R$. Since this extension is bounded in
the vicinity of $x_0$, it further extends to a holomorphic function
$\psi$ in all of $B_{1/2}^+(u)$ with $\psi(x_0)=0$. Since $\psi$
maps $\Omega_N$ injectively into $\Omega_{a_0}^+$ and $\Omega_S$
injectively into $\Omega_{a_0}^-$, $(\beta_L)^{\circ}$ injectively
into $(\tau_{a_0,L})^{\circ}$ and $(\beta_R)^{\circ}$ injectively
into $(\tau_{a_0,R})^{\circ}$, we conclude that $\psi$ is injective
on all of $B_{1/2}^+(u)$.

Lastly, we point out that since $U$ maps $\del\Omega_N\cap F(u)$
onto $[-l_L, l_R]$ and maps $\del \Omega_N \cap \del B_{1/2}$ into a
curve that is $O(\d_1)$-close to a half-circle of radius $1/2$,
according to \eqref{lemma:confmapConstruction_eq_halfcirc}, it has
to be that
\begin{equation*}
    U(\Omega_N) \supseteq \H_{a_0}\cap B_{1/3}.
\end{equation*}
Thus for all small $a_0$, $\psi(\Omega_N)=
(V_{a_0}|_{\Omega_{a_0}^+})^{-1} \big(U(\Omega_n)\big) \supseteq
\Omega_{a_0}^+\cap B_{1/4}$. After applying the same argument for
$\Omega_S$, we establish the full statement $\psi(B_{1/2}^+(u))
\supseteq \Omega_{a_0}\cap B_{1/4}$.

\end{proof}

We shall now use the map $\psi$ to obtain curvature bounds of $F(u)$
in $B_{1/4}$. On the road to do so, we will obtain the following
crucial estimates on $\psi'$ and $\psi''$.

\begin{lemma}\label{lemma:psiEstimates} The injective holomorphic
map $\psi:B_{1/2}^+(u)\to \Omega_{a_0}$ constructed in Lemma
\ref{lemma:confmapConstruction} satisfies:
\begin{equation*}
    |\psi''(z)|  \leq C\delta \quad\text{and}\quad
    |\psi'(z) - 1| \leq C\d (|z|+a_0) \quad \text{for} \quad z\in
    B_{1/4}^+(u).
\end{equation*}
\end{lemma}
\begin{proof}
We know that for $z\in \del B_{1/2}\cap \del B_{1/2}^+(u)$
\begin{equation*}
    |\psi'(z)| = \frac{|U'(z)|}{|V_{a_0}'(\psi(z))|} = 1 + O(\d)
\end{equation*}
because $|U'(z)| = 1+ O(\d)$  for $\del B_{1/2}\cap B_{1}^+(u)$ and
\begin{equation*}
    |V_{a_0}'(\psi(z))| = |V'(\psi(z)/a_0)| = 1 + O(\d) \quad
    z\in\del B_{1/2}\cap B_{1}^+(u)
\end{equation*}
for all $a_0 = a(1+O(\d_1))$ small enough (depending on $\d$),
because according to Lemma \ref{lemma:confmapConstruction}
\begin{equation*}
    |\psi(z)|\geq 1/4 \quad \text{when}\quad
    z\in\del B_{1/2}\cap B_{1}^+(u).
\end{equation*}
Furthermore, $|\psi'| = 1$ on $F(u)\cap B_{1/2}$, so that by the
maximum (and minimum) modulus principle,
\begin{equation}\label{lemma:psiEstimates:eq_psiprime}
    |\psi'| = 1+O(\d) \quad \text{in}\quad  B_{1/2}^+(u).
\end{equation}
Since $B_{1/2}^+(u)$ is simply-connected and since $\psi'\neq 0$ as
$\psi$ is conformal, we can write
\begin{equation*}
    \psi' = e^{G}
\end{equation*}
for some holomorphic function $G$ on $B_{1/2}^+(u)$. Then
\begin{equation*}
    \psi'' = G'\psi'
\end{equation*}
and in view of \eqref{lemma:psiEstimates:eq_psiprime}, we shall have
$\psi'' = O(\d)$ in $B_{1/4}^+(u)$ once we establish
\begin{equation*}
    |G'| \leq c\d \quad \text{in}\quad B_{1/4}^+(u).
\end{equation*}
Let $g = \text{Re}(G)$; as $|G'|=|\nabla g|$ it suffices to obtain
bounds on $|\nabla g|$ and we know that
\begin{equation*}
    g(z) = \log|\psi'(z)| = \left\{\begin{array}{cl} 0 & \ z\in F(u)\cap B_{1/2} \\ O(\d) & z\in  B_{1/2}^+(u)\end{array}\right.
\end{equation*}
In particular $g$ vanishes on $F(u)\cap B_{1/2}$ and we can apply
the boundary Harnack inequality in the $C\d$-Lipschitz domains
$B_{1/4}(z_{\pm})^+(u)$, where $z_{\pm}:=\rho(\pm 1/4,0)$, in order
to establish that
\begin{equation*}
    |g(z)|\leq c\d u(z)/u(z_{\pm} \pm i/8) \leq c'\d
    u(z) \quad \text{in}\quad B_{1/8}(z_{\pm})^+(u)
\end{equation*}
(since by assumption $(4)$, $u(z_{\pm} \pm i/8) \approx 1/8$).
Because we have
\begin{equation*}
    u\geq 1/8-\d\geq 1/10 \quad \text{on} \quad \del B_{1/4}\setminus
    \big(B_{1/8}(z_{+}) \cup B_{1/8}(z_{-})\big),
\end{equation*}
we see that $|g|\leq C\d u$ on $\del B_{1/4}\cap B_1^+(u)$ and thus
by the maximum principle,
\begin{equation*}
|g|\leq C \d u \quad \text{in all of} \quad B_{1/4}^+(u).
\end{equation*}
An application of the Hopf Lemma yields
\begin{equation*}
    |\nabla g| \leq C\d |\nabla u| = C\d \quad \text{on} \quad
    F(u)\cap B_{1/4}.
\end{equation*}
Finally, we have
\begin{equation*}
    |\nabla g| = \frac{|\nabla |\psi'|^2|}{2|\psi'|^2} \leq C\d \quad \text{on} \quad
    B_1^+(u)\cap \del B_{1/4}
\end{equation*}
because of \eqref{lemma:psiEstimates:eq_psiprime} and the fact that
on $B_1^+(u)\cap \del B_{1/4}$
\begin{equation*}
\nabla |\psi'|^2 = 2 \text{Re}\left(\nabla \big(U'/(V'_{a_0}\circ
\psi)\big) \overline{\psi'} \right) = O(|U''| + |V_{a}''\circ \psi|)
= O(\d).
\end{equation*}
Hence $|\nabla g| \leq C\d$ in all of $B_{1/4}^+(u)$ as desired.

To get the first derivative bound, we integrate the second
derivative bound along a curve $\gamma\subseteq B_{1/4}^+(u)$
connecting $p\in F_L\cap \beta$ (the ``left" end of the steepest
path) to $z$:
\begin{equation*}
    \psi'(z) = \psi'(x_0) + \int_{\gamma} \psi''(\z) ~d\z =
    \psi'(p)  + O(\d \hsd(\gamma)).
\end{equation*}
Since by definition $V'(\psi(p)) = U'(p)=1$ (as $\nabla u(p) = e_1$)
\[\psi'(p) = U'(p)/V'(\psi(p)) = 1 \]
As $\gamma$ can be taken to be of length $O(|z|+a_0)$, we obtain the
desired bound
\begin{equation*}
    |\psi'(z)| - 1| \leq C\d(|z|+a_0) \quad z\in
    B_{1/4}^+(u).
\end{equation*}

%Note that $\psi'(x_0)$ satisfies
%\begin{align*}
%    \psi'(x_0) & = \lim_{z\in \Omega_N, z\to x_0}
%    \frac{U'(z)}{V_{a_0}'(\psi(z))} = \lim_{z\in \Omega_N, z\to x_0}
%    \frac{U''(x_0)(z-x_0) + o(|(z-x_0)^2|)}{V_{a_0}''(0)\psi(z) +
%    o(|\psi(z)|^2)} \\
%    & = \lim_{z\in \Omega_N, z\to x_0}
%    \frac{U''(x_0)(z-x_0) + o(|(z-x_0)^2|)}{V_{a_0}''(0)\psi'(x_0)(z-x_0) +
%    o(|z-x_0|^2)} = \frac{U''(x_0)}{V_{a_0}''(0)\psi'(x_0)}
%\end{align*}
%so that
%\begin{equation*}
%    \psi'(x_0)^2 = \frac{U''(x_0)}{V_{a_0}''(0)} = 1+O(\d)
%\end{equation*}
%and since $\psi$ is orientation-preserving, we have $\psi'(x_0) =
%1+O(\d)$.Since $\gamma$ can be taken to be of length $O(1)$ (in
%fact $O(|z|)$),
% we derive the desired

\end{proof}

\begin{theorem}\label{thm:curvaturebounds} Given $\delta>0$ small
enough, there exist $r_0>0$, $\eps_1>0$ such that if $u$ is a
classical solution of \eqref{FBP_0} in $B_1$, $0\in F(u)$ and
\begin{equation*}
    \text{dist}(0, \{u=0\}\setminus Z) < \eps_1 r_0
\end{equation*}
then there exists a point $p\in B_{r_0/3}$ such that
$B_{r_0/2}(p)\cap F(u)$ consists of two free boundary arcs $F_L$ and
$F_R$, the shortest segment between which is centered at $p$, the
separation
\begin{equation*}
    s:=dist(F_L, F_R)<\eps_1 r_0.
\end{equation*}
Furthermore, $u$ has a unique saddle point $x_0$ in $B_{r_0/2}(p)$ and there is an injective holomorphic map \[\psi:
B_{r_0/2}(p)^+(u) \to \Omega_{a} \quad \text{where}\quad a=u(x_0)\] that extends continuously to
$\del B_{r_0/2}(p)^+(u)$, mapping $\psi(x_0)=0$ and \mbox{$F(u)\cap B_{r_0/2}(p)$} into $\del
\Omega_{a}$ and satisfying
\begin{equation}\label{thm:curvaturebounds:eq_psiEstimates}
    |\psi ''| \leq C\d/r_0 \qquad |\psi' - e^{i\theta}| < C\d(|z|+a)/r_0 \quad
     \text{in} \quad B_{r_0/2}(p)^+(u)
\end{equation}
for some $\theta\in\real$.
It relates the curvature $\kappa$ of $F(u)$ in $B_{r_0/2}(p)$ to
the curvature $\k_{a}$ of $\del \Omega_{ar_0}$ via
\begin{equation}\label{thm:curvaturebounds:eq_curv}
    |\k(z) - \k_{a}(\psi(z))| \leq C\d/r_0 \quad z\in F(u)\cap
    B_{r_0/2}(p)
\end{equation}
for some numerical constants $C, c>0$.
\end{theorem}
\begin{proof}
For $\d$ fixed we find $r_0$, $\eps_0$ as in Theorem
\ref{thm:fourgraphs}. Set $\d_1= \d$, $M=8/\eps_0$, apply
Proposition \ref{prop:hairpin_loc} to find $s_1=s_1(\delta, M)$ and
set $\eps_1=\min(\eps_0,s_1)$. Then $\tilde{u}(y):= (2/r_0)u(p+
\rho yr_0/2)$, defined in $B_1$ for some appropriate rotation $\rho\sim e^{i \theta}$, falls under the set-up (1-4) and we
construct the injective holomorophic map $\tilde{\psi}$ as in Lemma
\ref{lemma:confmapConstruction} satisfying the estimates of Lemma
\ref{lemma:psiEstimates}, whose rescaled statement is precisely
\eqref{thm:curvaturebounds:eq_psiEstimates}.

Let $\tilde{U}$ and $V_{a/r_0}$ be the holomorphic extensions of
$\tilde{u} = \text{Re}(\tilde{U})$ and $H_{2a/r_0} =
\text{Re}(V_{2a/r_0})$ such that
\begin{equation*}
    \tilde{U}(z) = V_{2a/r_0}(\tilde\psi(z)) \quad \text{in}\quad
    B_{1/2}.
\end{equation*}
Since the curvature of $F(\tilde{u})$ at $z$
\begin{equation*}
    \tilde{\kappa}(z) = \text{div}\frac{\nabla \tilde{u}}{|\nabla \tilde{u}|} = -\frac{\nabla \tilde{u}\cdot |\nabla \tilde{u}|^2}{2|\nabla
    \tilde{u}|^3} = -\frac{\text{Re}(\tilde{U}''(\overline{\tilde{U}'})^2)}{2|U'|^3} =
    -\frac{1}{2} \text{Re}(\tilde{U}''(\overline{\tilde{U}'})^2)
\end{equation*}
we have, in view of $|\tilde{\psi}'(z)| = 1 =
|V'_{2a/r_0}(\tilde{\psi}(z))|$,
\begin{align*}
    \tilde{\kappa}  & = -\frac{1}{2} \text{Re}\Big(\big(V''_{2a/r_0}\tilde{\psi}'^2 +
    V'_{2a/r_0}\tilde{\psi}''\big)\big(\overline{V'_{2a/r_0}\tilde{\psi}'}\big)^2
    \Big)  = -\frac{1}{2}\text{Re}\big(V''_{2a/r_0}(\overline{V'_{2a/r_0}})^2
+ \tilde{\psi}'' \overline{V'_{2a/r_0}(\tilde{\psi})'^2} \big) \\
& = \k_{2a/r_0}\circ\tilde{\psi} + O(\d),
\end{align*}
which is the rescaled version of
\eqref{thm:curvaturebounds:eq_curv}.
\end{proof}

We now have all the ingredients for Theorems \ref{thm:BigPicture}
and Theorem \ref{thm:CurvBounds}.

\begin{proof}[Proof of Theorem \ref{thm:BigPicture}]
Fix $\delta > 0$  to be smaller than the flatness constant $\d_0$ (Remark \ref{remark:deltaFlatness}). Let $r_0$, $\eps_0$ be as in Theorem \ref{thm:fourgraphs}. Set $\d_1=\d$, $M=8\eps_0$ and apply Proposition \ref{prop:hairpin_loc} to find $s_1=s_1(\delta, M)$. Finally, set $\eps_1 = \min\{s_1, \eps_0\}$.  For any point $q\in F(u)$, let
$Z_q$ be the component of the zero phase to which $q$ belongs.
Define the set of points
\begin{equation*}
    \mathcal{C}_{\text{prox}} = \{q\in F(u)\cap \overline{B_{1/2}}: \text{dist}(q, \{u=0\}\setminus Z_q)<\eps_1
    r_0\}
\end{equation*}
in whose neighborhood we expect to see a hairpin structure.
According to Theorem \ref{thm:fourgraphs}, for every $q\in
\mathcal{C}_{\text{prox}}$ there exists a $z(q)\in B_{r_0/3}(q)$
such that $F(u)\cap B_{r_0/2}(z)$ is an approximate hairpin centered
at $z=z(q)$ in the sense that: 
\begin{itemize}
\item $F(u)\cap B_{r_0/2}(z)$ consists of two
arcs $F_L$ and $F_R$;
\item if $s=\text{dist}(F_L,F_R)$, we have for some rotation
$\rho$ and functions $f,g : \real\to\real$ with $f<g$,
\begin{align*}
    \{u=0\} \cap \big(B_{r_0/2}(z)\setminus B_{4s/\eps_0}(z)\big)
     = z & + \rho \{4s/\eps_0<|x|<r_0/2: f(x_1)\leq |x_2|\leq
    g(x_1)\} \\
\text{where} \quad \|f\|_{L^\infty} + \|g\|_{L^\infty}  & \le \delta r, 
\quad
\|f'\|_{L^\infty} + \|g'\|_{L^\infty}  \le \delta.
\end{align*}
\end{itemize}

%\begin{align*}
%    \{u=0\} \cap \big(B_{r_0/2}(z)\setminus B_{4s/\eps_0}(z)\big)
%     = z & + \tilde{\rho} \{4s/\eps_0<|x|<r_0/2: f_{-}(x_1)\leq |x_2|\leq
%    f_+(x_1)\} \\
%    \text{where} \quad \|f_-'\|_{L^{\infty}(-r_0/2, r_0/2)} & + \|f_+'\|_{L^{\infty}(-r_0/2,
%    r_0/2)} \leq C\d
%\end{align*}
At the same time, Proposition \ref{prop:hairpin_loc} says that inside
$B_{8s\eps_0}(z)$,
\begin{equation*}
    |u(z+\tilde{\rho} x) - H_{a}(x)| \leq \d a.
\end{equation*}
for $a=s/(2+\pi)$ and some rotation $\tilde{\rho}$. Since the free boundary outside $B_{8s\eps_0}(z)$ has to match with the one inside, we may take $\tilde{\rho}=\rho$.

A standard covering argument yields a finite number of disks
$\{B_{r_0/2}(z_j)\}_{j=1}^N$, where $N\leq N_0 = O(r_0^{-2})$, which
cover $\mathcal{C}_{\text{prox}}$ with the centers $z_j$ constructed as above. For points $p\in F(u)\cap
B_{1/2}\setminus \bigcup_{j=1}^N B_{r_0/2}(z_j)$, we know that
\[\text{dist}(p, \{u=0\}\setminus Z_p) \geq \eps_1 r_0,\] so by
Proposition \ref{prop:bounded curvature}, the curvature of $F(u)$ at
$p$ is at most $\kappa :=\kappa_0(\d)$.

Defining $r:=4r_0$, $\eps:= \eps_0/4(2+\pi)$, we get the precise form of the statements in Theorem \ref{thm:BigPicture}. 
% The collection of disks
%$\{B_{r/2}(p(q))\}_{q\in \mathcal{C}}$ is a covering of the compact
% set $\overline{\mathcal{C}_{\text{prox}}}$
\end{proof}

\begin{proof}[Proof of Theorem \ref{thm:CurvBounds}]
Fix $0<\d<1/100$ small and let $r_0 = r(\d)/2$ where $r$ is as in Theorem \ref{thm:BigPicture}. Running the same covering argument in the proof above, we have a collection of disks $\{B_{4r_0}(p_j)\}$ for each of which Theorem \ref{thm:curvaturebounds} gives:  a unique saddle point $z_j$ of $u$ in $B_{4r_0}(p_j)$ and an injective holomorphic map \[\psi_j:B_{4r_0}(p_j)^+(u) \to \Omega_{a_j} \quad \text{where}\quad a_j=u(z_j)\] with all the enumerated properties in Theorem \ref{thm:curvaturebounds}. Defining $\phi_j: B_{2r_0}\cap \Omega_{a_j} \to \real^2$ by $\phi_j := \psi_j^{-1}$ we obtain the precise form of the statements in Theorem \ref{thm:CurvBounds}.
\end{proof}

\section{The minimal surface analogue.}\label{Sec:Traizet}

In \cite{Traizet} (see Theorems $9$ and $10$) Traizet discovered a
remarkable correspondence between global solutions of \eqref{FBP_0}
with $|\nabla u| < 1$ and complete embedded minimal bigraphs
(minimal surfaces symmetric with respect to a plane with the two
halves, ``above" and ``below" the plane, being graphical). The
correspondence is expressed via the Weierstrass representation
formula for \emph{immersed} minimal surfaces. Recall, if
$X:M\subseteq \real^3$ denotes the minimal immersion, the coordinate
$X_3$ is a harmonic function on $M$ and one can locally define a
harmonic conjugate $X_3^*$, so that
\begin{equation*}
    dh = dX_3+ i dX_3^*
\end{equation*}
is a well-defined holomorphic differential on $M$ (viewed as a
Riemann surface), the so-called \emph{height differential}.
Furthermore, the stereographically projected Gauss map $g:M\to
\C\cup\{\infty\}$ is a meromorphic function on $M$. The pair $(g,
dh)$ is called the Weierstrass data of the minimal surface and the
minimal immersion $X$ is given, up to translation, by
\begin{equation}\label{eq:Weierstrass}
    X(p)= (X_1(p), X_2(p), X_3(p)) = \text{Re}\int_{p_0}^p
    \left(\frac{1}{2}(g^{-1}- g)dh, \frac{i}{2}(g^{-1}+g)dh, dh  \right)
\end{equation}
where $p_0$ is a fixed point in $M$. Conversely, if $M$ is a Riemann
surface, and $(g, dh)$ is a pair of a meromorphic function and a
holomorphic $1$-form on $M$, satisfying certain compatibility
conditions (\cite{osserman}), then \eqref{eq:Weierstrass} defines a
minimal immersion of $M$ in $\real^3$.

Traizet's brilliant insight was to define
\begin{equation*}
    g = 2\frac{\del u}{\del z} \quad \text{and} \quad dh = 2\frac{\del u}{\del
z} dz
\end{equation*}
in terms of a solution $u$ of \eqref{FBP_0}, and show that, under
certain conditions, the Weierstrass data $(g, dh)$ give rise to the
upper half $(X_3>0)$ of a minimal bigraph. Conversely, a solution
$u$ of \eqref{FBP_0} can be constructed using the Weierstrass data
of a complete embedded minimal bigraph.

We have used Traizet's correspondence to state Corollary
\ref{thm:minimal_annulus}, the minimal surface version of Theorem
\ref{thm:CurvBounds}. We can now turn to the proof.

%\begin{theorem}\label{thm:minimal_annulus} Let $M\subseteq\mathcal{B}_1 = \{r:=\sqrt{x_1^2+x_2^2+x_3^2}<1\}$ be an embedded
%minimal surface, homeomorphic to an annulus, with $\del M\subseteq
%\del B_1$. Assume that $M$ is symmetric with respect to $\{x_3=0\}$
%and that $M^+=M\cap \{x_3>0\}$ is a simply-connected graph. Let
%$\gamma$ denote a shortest closed geodesic on $M$ and assume that
%$0\in \gamma$. For any $\delta>0$ there exists an $r_0>0$ and an
%$\eps_0>0$ such that if
%\begin{equation*}
%    \rho:=\hsd(\gamma)/2\pi < \eps_0 r_0
%\end{equation*}
%then there exists an injective holomorphic map $\psi: M\cap
%B_{r_0}\to \Sigma_{s_0}$, where $\Sigma_{\rho_0}$ is the standard
%horizontal catenoid of neck length $2\pi\rho_0 = 2\pi\rho(1 + O(\d
%\rho))$:
%\begin{equation*}
%\Sigma_{\rho_0} = \{(x_1, x_2, x_3)\in \real^3: (x_2/\rho_0)^2 +
%(x_3/\rho_0)^2 = \cosh^2(x_1/\rho_0))\}.
%\end{equation*}
%This map is a isometry up to a factor $1+O(\delta(r+\rho))$ and
%relates the Gauss curvature $K$ of $M$ to the Gauss curvature
%$K_{\rho_0}$ of $\Sigma_{\rho_0}$ by
%\begin{equation*}
%    |K(p) - K_{\rho_0}(\psi(p))| \leq C \delta (\delta + \sqrt{|K_{\rho_0}(\psi(p))|})  \quad \text{for}
%    \quad p\in M\cap B_{r_0}.
%\end{equation*}

\begin{proof}[Proof of Corollary \ref{thm:minimal_annulus}]

%By scale invariance, it suffices to show that if $$ 
%Denote $M^+ := M\cap \{X_3 > 0\}$. 

Following the argument of
\cite[Theorem 10]{Traizet}, we shall construct a solution of
\eqref{FBP_0}, corresponding to the minimal bigraph $M$. Let $\z$ be
a complex coordinate on $M$, let $g$ be the stereographically
projected Gauss map and $dh = (2 \del X_3/\del \z) ~d\z$ be the
height differential. Note that $|g|=1$ on $M\cap \{X_3 = 0\}$ as the
normal points horizontally there and we may assume that the
orientation of $M$ is chosen so that the normal points down in
$M^+$, i.e. $|g|<1$ in $M^+$. Furthermore $g$ has the same zeros and
poles as $dh$ (with same multiplicities), thus $g^{-1}dh$ defines a
holomorphic non-vanishing one-form on $M^+$. Since $M^+$ is
simply-connected,
\begin{equation*}
    \f(p)=\int_{0}^p g^{-1}dh \quad p\in M^+
\end{equation*}
defines a holomorphic function on $M^+$ (recall $0\in M$). Claim
that $\f$ is injective. Define
\begin{equation*}
    \Xi:= X_1 + i X_2
\end{equation*}
on $M^+$ and let $\hat{\Omega} = \Xi(M^+)$ be the projection of
$M^+$ down to the horizontal plane $\{X_3=0\}$. Since $M^+$ is a
graph, $\Xi$ is a diffeomorphism from $\overline{M^+}$ to
$\overline{\hat{\Omega}}$, so $\f$ will be injective if and only if
$\phi:=\f\circ \Xi^{-1}$ is injective on $\hat{\Omega}$. Let $a, b$
be arbitrary points of $\hat{\Omega}$ and let $[a,b]\subseteq\C$
denote the straight-line closed segment from $a$ to $b$. Then for
some $N\in\nat$ we can write
\begin{equation*}
    [a,b]= \bigcup_{k=1}^N [z_{2k-1}, z_{2k}] \cup
    \bigcup_{k=1}^{N-1} [z_{2k}, z_{2k+1}],
\end{equation*}
where $z_1 = a$, $z_{2N}=b$, the interior of $[z_{2k-1}, z_{2k}]$
belongs to $\hat{\Omega}$, while $z_{2k}$ and $z_{2k+1}$ belong to
the same connected component of $\del \hat{\Omega}$. Claim that
\begin{equation}\label{thm:minimal_annulus:eq_Inject}
    \langle \overline{\phi}(z_{2k})-\overline{\phi}(z_{2k-1}), \frac{b-a}{|b-a|}
    \rangle > |z_{2k} - z_{2k-1}|,
\end{equation}
where $\langle w_1, w_2 \rangle := \text{Re}(\overline{w_1} w_2)$
denotes the standard inner product on $\C$. Let $\alpha:[0,1]\to
M^+$ be such that $\Xi\circ \alpha$ is the constant speed
parameterization of $[z_{2k-1}, z_{2k}]$. For each fixed time $t\in
(0,1)$, denote
\begin{equation*}
v := \frac{1}{2}\overline{g^{-1}dh}(\alpha'(t)) \qquad w := -
\frac{1}{2} gdh(\alpha'(t)),
\end{equation*}
we have $|v|>|w|$ because $|g|<1$. Since $ d\f(\alpha'(t)) =
g^{-1}dh (\alpha'(t)) = 2 \overline{v}$ and
\begin{equation*}
     z_{2k} - z_{2k-1}= d\Xi(\alpha'(t)) = (dX_1 + i dX_2)(\alpha'(t)) = \frac{1}{2}\overline{g^{-1}dh}(\alpha'(t)) -
    \frac{1}{2} gdh(\alpha'(t)) = v+w,
\end{equation*}
we have
\begin{equation*}
    \langle \overline{d\f}(\alpha'(t)), \frac{b-a}{|b-a|} \rangle =
    |z_{2k}-z_{2k-1}|^{-1}\langle 2 v, v+w \rangle > |z_{2k}-z_{2k-1}|^{-1}
    |v+w|^2 = |z_{2k}-z_{2k-1}|
\end{equation*}
which leads to \eqref{thm:minimal_annulus:eq_Inject} once we
integrate in $t$ from $0$ to $1$. On the other hand,
\begin{equation}\label{thm:minimal_annulus:eq_Inject2}
    \langle \overline{\phi}(z_{2k+1})-\overline{\phi}(z_{2k}), \frac{b-a}{|b-a|}
    \rangle = |z_{2k+1} - z_{2k}|.
\end{equation}
This is the case, because on the component $\beta$ of $M\cap
\{X_3=0\}$, to which $\Xi^{-1}(z_{2k+1})$ and $\Xi^{-1}(z_{2k})$
belong, we know $g^{-1} = \overline{g}$ and $\overline{dh} = - dh$,
so that
\begin{equation*}
    \overline{d \f}(\beta') = \overline{g^{-1}dh}(\beta') =
    -g dh(\beta') = \frac{1}{2}\overline{g^{-1}dh}(\beta') -
    \frac{1}{2} gdh(\beta') = d\Xi(\beta')
\end{equation*}
and thus, $\overline{\phi}(z_{2k+1})-\overline{\phi}(z_{2k}) =
z_{2k+1} - z_{2k}$. Adding up \eqref{thm:minimal_annulus:eq_Inject}
and \eqref{thm:minimal_annulus:eq_Inject2} from $k=1$ to $N$, we
derive
\begin{equation*}
    \langle \phi(b) -\phi(a), (b-a)/|b-a| \rangle > |b-a|
\end{equation*}
from which the injectivity of $\phi$ follows.

We can now define the function
\begin{equation*}
    u = X_3\circ \f^{-1}
\end{equation*}
on the domain $\Omega = \f(M^+)$, %extending it by $0$ outside, 
and we can easily verify that $u$
is a positive, harmonic function in $\Omega$ that vanishes on $\del \Omega \cap
B_R$ where, for $z=\f(\z)\in F(u)$
\begin{equation*}
    |\nabla u|(z) = \left|2\frac{\del X_3}{\del \z} \frac{1}{\f'(\z)}\right| =  |g(\z)|
    =1.
\end{equation*}
Furthermore, the metric induced on $\Omega$ by the conformal
immersion $X\circ \f^{-1}$ is given by the standard formula
\begin{equation*}
    ds = \frac{1}{2}(|g||dh| + |g|^{-1}|dh|) = \frac{1}{2}(|g|^2+ 1)
    |dz| = \lambda(z) |dz|
\end{equation*}
where $\frac{1}{2}\leq \lambda(z) \leq 1$, as $|g|\leq 1$ on $M^+$.
So, if $\gamma^+ = \gamma\cap M^+$ denotes the piece of the shortest
geodesic lying in $M^+$, it is mapped by $\f$ to a curve
$\tilde{\gamma}=\f(\gamma^+) \subseteq \Omega$ with Euclidean length
$O(\hsd(\gamma^+))  = O(\eps)$ which connects the two pieces of $\del \Omega$. 

Fix $\d<1/1000$ a small positive numerical constant and let $r_0$, $\eps_1$ be as in Theorem  \ref{thm:curvaturebounds}. Set $R_0=1/r_0$ and $\eps_0=\eps_1$. Extend $u$ by zero in $B_{R_0}\setminus \Omega$. Then $u$ is a solution of \eqref{FBP_0} in $B_{R_0}$, satisfying \eqref{topo_assmpt}, so Theorem \ref{thm:curvaturebounds} gives us an injective conformal map $\tilde{\psi}: B_{4}^+(u) \to \Omega_{a}$ for some appropriate $a = O(\eps)$, such that 
\begin{equation}\label{thm:minimal_annulus:eq_psiEst}
    \tilde{\psi}'(z) = 1 + O(\d(|z|+a)), \qquad 
    \tilde{\psi}''(z)= O(\d) \qquad \text{for}\quad z\in B_{4}^+(u)
\end{equation}
and $U := V_{a}\circ \tilde{\psi}$ is a holomorphic extension of $u$ in $B_4^+(u)$ (recall $V_{a}$ is the
holomorphic extension of $H_{a}$ given in Section \ref{Sec:CurvBounds}). It's easy to see that $\tilde{\psi}$ gives
rise to an injective conformal map from $M^+\cap \mathcal{B}_2$ into the standard catenoid $\Sigma_{\rho}^+ := \Sigma_{\rho}\cap \{X_3>0\}$ (the counterpart to $\Omega_{a}$ in the Traizet correspondence), which then extends
by symmetry to a conformal map $\psi$ on all of $M\cap \mathcal{B}_{2}$. 
The metric on $\Omega_{a}$ induced by its immersion as
$\Sigma_{\rho}$ is
\begin{equation*}
    ds_{\text{cat}} = (1+ |V_{a}'|^2)|dz|
\end{equation*}
while the metric on $B_{4}^+(u)$ is
\begin{equation*}
    ds = (1+ |U'|^2)|dz|
\end{equation*}
and we check that the pull-back metric
$\tilde{\psi}^*(ds_{\text{cat}})$ satisfies
\begin{equation*}
\tilde{\psi}^*(ds_{\text{cat}}) = (1 +
|U'|^2/|\tilde{\psi'}|^2)|\tilde{\psi}'| |dz| =
\big(1+O(\d (|z|+a))\big)ds.
\end{equation*}
Since $a \sim \rho$, the induced conformal map $\psi$ is an isometry up to a factor
of $\big(1+ O(\d (|x|+\rho)\big)$% , where $r$ is the distance to the
%origin, and the length of the shortest geodesic on $M$,
and
\begin{equation*}
    \eps =\hsd(\gamma) = (1+ O(\d))2\pi\rho \quad \implies \quad |\eps - 2\pi \rho| = O(\d \eps) < \eps/100.
\end{equation*}
Furthermore, the Gauss curvature of $M$ is given by the standard
formula for the curvature of a conformal metric $\lambda(z)|dz| =
(1+|U'|^2)|dz|$
\begin{equation*}
    K = -\frac{\Delta \log\lambda(z)}{\lambda^2(z)} = - \frac{4|U''|^2}{(1+|U'|^2)^4}
\end{equation*}
Plugging in $U(z) = V_{a}(\tilde{\psi}(z))$ and applying the
estimates \eqref{thm:minimal_annulus:eq_psiEst}, we get
\begin{align*}
    K & = - \frac{4|V_{a}''(\tilde{\psi}')^2 + V_{a}'\tilde{\psi}''|^2}{(1 +
    |V_{a}'|^2|\tilde{\psi'}|^2)^4} = -
    \frac{4|V_{a}''|^2}{(1+|V_{a}'|^2)^4} (1 + O(\delta (|z|+ \rho)) + O\Big(\d \frac{2|V_{a}''|}{(1+|V_{a}'|^2)^2} \Big)
   + O(\delta^2)
    = \\
    & = K_{\rho}+ O\big(\delta (r+\rho) K_{\rho} \big) + O(\delta
    \sqrt{|K_{\rho}|}) + O(\delta^2) 
\end{align*}
Noting that
\begin{equation*}
    \sqrt{|K_{\rho}(q)|} = O(\rho/r(q)^2)
\end{equation*}
and that $|r(\psi(p)) - r(p)| \sim \rho + \delta r(p)$ we obtain the
desired estimate
\begin{equation*}
    K(p) = K_{\rho}(\psi(p)) + O\left(\delta \frac{\rho}{r+\rho}\sqrt{|K_{\rho}|}\right) + O\left(\delta
    \sqrt{|K_{\rho}|}\right) + O(\delta^2) = K_{\rho_0} +  O\left(\delta +
    \sqrt{|K_{\rho}|}\right)\delta.
\end{equation*}
\end{proof}

\bibliography{FB_Bib}
\end{document}